\documentclass{aims}

\usepackage[dvipsnames]{xcolor}
\usepackage{lscape}
\usepackage{amssymb}           
\usepackage[numbers]{natbib}   
\usepackage{paralist}
\usepackage[misc]{ifsym}       
\usepackage{epsfig}
\usepackage{epstopdf}
\usepackage{algpseudocode}
\usepackage[shortlabels]{enumitem}
\usepackage[algo2e,linesnumbered,vlined,ruled]{algorithm2e}
\usepackage{multirow}
\usepackage{siunitx}
\usepackage{graphicx}
\usepackage{cancel}
\usepackage{subfig}
\usepackage{upgreek}
\usepackage{tikz}
\usepackage{tkz-euclide}
\usepackage[colorlinks=true]{hyperref}
\usepackage{xurl}
\usepackage{cleveref}
\usepackage{float}
\usepackage{placeins}
\usepackage{array}     
\usepackage{rotating}
\usepackage{pdflscape}

\usepackage{booktabs}  
\usepackage{ulem}      
\hypersetup{urlcolor=blue, citecolor=red}

\definecolor{honopiilani}{RGB}{82,81,245}
\definecolor{front}{RGB}{21,249,78}
\definecolor{wainee}{RGB}{249,236,21}
\definecolor{minor}{RGB}{178,0,255}
\definecolor{west_source}{RGB}{244,0,111}
\definecolor{east_source}{RGB}{255,114,178}
\definecolor{bypass}{RGB}{179,15,15}
\definecolor{keawe}{RGB}{255,51,51}
\definecolor{ll}{RGB}{244,127,57}
\definecolor{dirt}{RGB}{100,63,19}

\allowdisplaybreaks

\newtheorem{theorem}{Theorem}[section]

\newtheorem{proposition}[theorem]{Proposition}

\theoremstyle{definition}              

\newcommand{\ml}{\left(}
\newcommand{\mr}{\right)}

\newcommand{\mlm}{\left|}
\newcommand{\mrm}{\right|}

\newcommand{\D}{\mathrm{d}}

\newcommand{\p}{\partial}

\DeclareRobustCommand{\okina}{%
  \raisebox{\dimexpr\fontcharht\font`A-\height}{%
    \scalebox{0.8}{`}%
  }%
}

\def\currentvolume{X}
\def\currentissue{X}
\def\currentyear{200X}
\def\currentmonth{XX}
\def\ppages{X-XX}
\def\DOI{10.3934/xx.xxxxxxx}

\title[Traffic flow network modeling for wildfire evacuation]
{Macroscopic traffic flow network modeling for wildfire evacuation: A game-theoretic junction optimization approach with application to the Lahaina fire}

\author[A. Lu, H. K. Tan, A. Xue, A. Koniges and A. Bertozzi]{}



\thanks{This work was supported by NSF grants CCF-2345255 and CCF-2345256.  The authors used AI tools (Claude and ChatGPT) to help with code development and suggestions for some of the English exposition.}

\thanks{$^*$Corresponding author: Hong Kiat Tan}

\begin{document}
\maketitle

\centerline{\scshape
Annie Lu$^{{\href{mailto:annieclu@g.ucla.edu}{{\textrm{\Letter}}}}\dagger 1}$,
Hong Kiat Tan$^{{\href{mailto:maxtanhk@g.ucla.edu}{{\textrm{\Letter}}}}*\dagger 1}$,
Alex Xue$^{{\href{mailto:alexxue@g.ucla.edu}{{\textrm{\Letter}}}}\dagger 1}$,
Alice Koniges$^{{\href{mailto:koniges@hawaii.edu}{{\textrm{\Letter}}}}2}$
and Andrea  L. Bertozzi$^{{\href{mailto:bertozzi@g.ucla.edu}{{\textrm{\Letter}}}}1,3}$}

\medskip

{\footnotesize
\centerline{$^\dagger$These authors contributed equally to this work.}
}

\medskip

{\footnotesize
\centerline{$^1$Department of Mathematics, University of California, Los Angeles, United States}
}

\medskip

{\footnotesize
\centerline{$^2$Department of Information and Computer Sciences, University of Hawai\okina i at M\=anoa, United States}
}

\medskip



{\footnotesize
\centerline{$^3$Department of Mechanical and Aerospace Engineering, University of California, Los Angeles, United States}
}

\bigskip

\centerline{(Communicated by Handling Editor)}


\begin{abstract}
The 2023 Lahaina wildfire killed 102 people on a peninsula served by a single two-lane highway, making exit lane capacity the binding constraint on evacuation time. We model the evacuation as a system of hyperbolic scalar conservation laws on a directed graph with game-theoretic junction conditions that maximize total network flux, an evacuation-calibrated piecewise linear-quadratic flux function, and a loss-driven optimization framework that tunes traffic distribution toward priority corridors. Analytical results on a toy network and numerical simulations of the Lahaina road network reveal a phase transition in exit lane capacity. Additional lanes improve throughput linearly until a computable critical threshold, beyond which no route optimization yields further benefit. For Lahaina, reversing one southbound lane captures nearly all achievable improvement, and a fourth lane can be reserved for emergency vehicles with negligible impact on civilian clearance time. These results provide a rigorous mathematical basis for contraflow recommendations in wildland-urban interface evacuations.
\end{abstract}


\section{Introduction}

Wildfires have become an increasingly severe threat to communities worldwide, with climate change intensifying both their frequency and destructive potential \citep{Jolly2015Climate}. Wildland-urban interface communities, where developed areas meet vegetation, are particularly vulnerable to these disasters. The tragic August 8-9, 2023, Lahaina fire in Maui, Hawai'i exemplifies the devastating consequences when wildfire evacuation systems fail. The fire claimed 102 lives and destroyed over 2,200 structures \citep{lahainafirereport}, making it the deadliest U.S. wildfire in over a century. Traffic congestion during the evacuation contributed significantly to this tragedy, with reports indicating vehicles becoming trapped on gridlocked roads as the fire advanced \citep{lahainafirereport}. Such events underscore the critical need for sophisticated traffic flow models that can inform evacuation planning and real-time emergency management in wildland-urban interface communities.

The modeling of traffic flow on networks has a rich mathematical foundation dating back to the seminal work of Lighthill, Whitham, and Richards \cite{LighthillWhitham1955Traffic,Richards1956Shock}, who formulated traffic flow as a macroscopic continuum governed by hyperbolic conservation laws. The extension to networks was pioneered by Holden and Risebro \cite{PDETraffic1}, and Coclite et al. \cite{PDETraffic2} later streamlined this framework via preference matrices at junctions. Furthermore, at congested junctions, models such as the classical FIFO (first-in-first-out) node models scale the incoming flux of vehicles proportional to the maximum capacities of the outgoing roads while respecting the drivers' preference. As illustrated in the four-leg scenario of Tampère et al. \cite{Tampere2011GenericNodeModels}, this leaves capacity on an alternate exit unused once one of the outgoing roads saturates. Such behavior then motivates an optimization-based junction treatment for evacuation settings where any spare capacity is critical.

Addressing the gap between theoretical network models and real-world evacuation dynamics is challenging because macroscopic traffic models calibrated to routine conditions struggle to capture evacuation behavior. Recent empirical studies of the 2019 Kincade Fire \citep{kincade_2019} and 2020 Glass Fire \citep{glass_2020} reveal that evacuation traffic exhibits fundamentally different characteristics from normal traffic patterns. Speeds decrease by approximately 3.5 km/h across all density ranges, and maximum flow rates drop by 2--5\%. These values provide empirically validated baselines, though their applicability varies by network topology. In bottleneck-dominated regimes such as the Lahaina network, where the system quickly saturates, capacity constraints rather than these behavioral parameters govern evacuation outcomes. These behavioral changes, driven by reduced visibility from smoke  \citep{INTINI2022103211}, increased caution to avoid accidents, unfamiliarity with evacuation routes, and vehicles loaded with belongings, necessitate evacuation-specific traffic models. Moreover, wildfire scenarios introduce dynamic network topology as roads become blocked by fire, debris, or emergency operations, requiring models that can handle time-varying network structures. Notably, this bottleneck-dominated behavior is not unique to Lahaina. Many wildland-urban interface communities, including Topanga Canyon \cite{cova2013mapping} and the Woolsey fire corridor \citep{lacounty2019woolsey}, share similar funnel-to-arterial road topologies with small exit cut-sets \citep{cova1997modelling, cova2013mapping}, 
making the framework directly transferable by substituting the local graph 
and recalibrating the loss function weights. \\

This paper makes four contributions.
\begin{enumerate}[(i)]
\item \textbf{Game-theoretic junction conditions (Section \ref{section:math}).} We formulate junction dynamics as an optimization problem, where drivers maximize throughput subject to entropy conditions, allowing dynamic adjustment of route preferences during congestion. The resulting solver reallocates flux towards under-utilized legs while preserving a unique solution (Theorem \ref{opt-sol}), providing the evacuation-oriented alternative to FIFO and capacity-proportional rules \citep{Tampere2011GenericNodeModels, Jin2003Distribution}. The junction resolution explicitly handles three regimes based on capacity ratios, automatically determining when drivers should deviate from preset preferences to maximize evacuation efficiency.

\item \textbf{Phase transition in exit lane capacity (Section \ref{section:results}).} We identify a sharp transition separating an \emph{exit-capacity-limited} regime from a \emph{supply-limited} regime. The critical exit lane threshold is computable in closed form from the flux ratios of incoming and outgoing roads (Proposition \ref{closedform-CD}). Below this threshold, adding exit lanes increases throughput linearly; beyond it, route optimization provides no marginal benefit. This dichotomy explains why, in bottleneck-dominated networks, infrastructure changes dominate any traffic management strategy.

\item \textbf{Evacuation-calibrated flux functions (Section \ref{section:modeling}).} We construct a piecewise linear-quadratic fundamental diagram informed by wildfire observations \citep{kincade_2019} and the measured capacities of the Lahaina road network, capturing free-flow and congested regimes with initial conditions and scale parameters that reflect evacuation traffic volumes while remaining analytically compatible with our PDE framework.

\item \textbf{Computational methodology (Sections \ref{section:optimization}--\ref{section:results}).} We develop an end-to-end computational pipeline comprising (a) a stochastic block coordinate descent scheme that tunes junction distribution parameters via a loss function prioritizing critical corridors, (b) a Godunov solver with supply--demand junction coupling, and (c) a full case study of the August 8--9, 2023 Lahaina evacuation \citep{lahainafirereport} including scenario analyses and contraflow lane reversal recommendations. The source code is publicly available at \url{https://github.com/alexxue99/Traffic-Simulation}.
\end{enumerate}

The following diagram summarizes our end-to-end pipeline contributions, from modeling and optimization through full numerical simulations of the Lahaina network.

\begin{figure}[htbp]
    \centering
\includegraphics[scale=0.3]{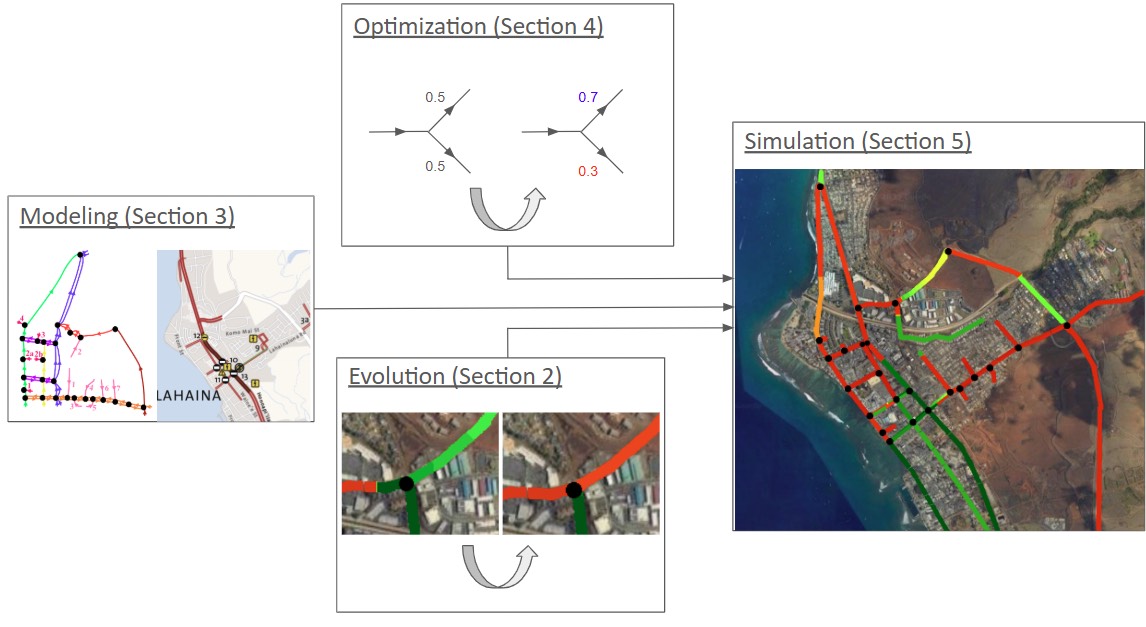}
\caption{Overview of the paper's pipeline, from network modeling and traffic evolution to optimization and simulation of the Lahaina network.}
\label{fig:workflow}
\end{figure}

Applied to the 2023 Lahaina wildfire, our framework yields the following actionable findings.
\begin{itemize}
\item Optimizing driver preferences to favor exit-facing roads improves evacuation throughput when the network is not yet saturated (Sections \ref{sec:phase-1}--\ref{sec:phase-2}).
\item Once roads are gridlocked, route optimization provides diminishing returns. The binding constraint is road capacity, not driver behavior.
\item Reversing one southbound lane on Honoapi`ilani Highway (creating 3 total northbound exit lanes) captures nearly all achievable improvement; a fourth lane yields no further benefit and can be reserved for emergency vehicles (Section \ref{sec:phase-4}).
\item Opening southbound exits in the afternoon was critical for clearing residents, suggesting that future plans should prioritize having multiple exit directions over optimizing a single route (Section \ref{sec:phase-3}).
\end{itemize}

The remainder of this paper is organized as follows. Section \ref{section:math} develops the mathematical framework for traffic flow on networks, bringing together the classical Lighthill-Whitham-Richards traffic model, our game-theoretic junction conditions, and the numerical scheme. Section \ref{section:modeling} calibrates the flux functions and initial conditions using both routine Lahaina traffic data and wildfire-era observations. Section \ref{section:optimization} then introduces the network optimization framework based on nonsmooth stochastic block coordinate descent. Section \ref{section:results} assembles these ingredients in numerical experiments on a toy network and the Lahaina evacuation, identifying the phase transition in exit lane capacity and examining how alternative weights, preference updates, and capacity assumptions alter system performance. Section \ref{section:conclusion} concludes with discussion and future directions.

\section{Mathematics of Continuum Models for Traffic Flow on a Network}\label{section:math}

We model traffic propagation on a road network as a collection of scalar conservation laws posed on directed road segments and coupled through junction conditions. To establish notation, we review the single-road Lighthill-Whitham-Richards (LWR) model before turning to its network extension and junction entropy conditions.

\subsection{Lighthill-Whitham-Richards (LWR) Model for Traffic Flow on a Road}

The Lighthill-Whitham-Richards (LWR) model \citep{LighthillWhitham1955Traffic, Richards1956Shock} describes traffic via a scalar conservation law for the vehicle density $\rho(t,x)$.
\begin{equation}\label{MathCont1}
\frac{\p\rho}{\p t}(t,x) + \frac{\p f}{\p x}(t,x) = 0.
\end{equation}
The flux depends on density alone, $f(\rho) = u(\rho)\rho$, giving
\begin{equation}\label{MathCont2}
f(\rho) = u(\rho)\rho,
\end{equation}
which rewrites \eqref{MathCont1} as
\begin{equation}\label{MathCont3}
\frac{\p}{\p t}\rho (t,x) + \frac{\p}{\p x}f(\rho(t,x)) = 0.
\end{equation}
The flux is assumed strictly concave with $f(0)=f(\rho_{\text{max}})=0$, where $\rho_{\text{max}}>0$ is the jam density. Assuming a linear speed--density relation gives the standard quadratic (Greenshields) flux
\begin{equation}\label{MathCont4}
f(\rho) = u_{\text{max}}\rho \ml 1 - \frac{\rho}{\rho_{\text{max}}} \mr,
\end{equation}
where $u_{\text{max}}$ is the free-flow speed. Solutions to \eqref{MathCont3} generically develop discontinuities, necessitating weak solutions with entropy conditions; see, e.g., \citep{Evans, LeVeque}.

\subsection{Traffic Flow on a Network}

Next, we will extend the model from a single road to multiple roads on a traffic flow network. A continuum model for such a network was first introduced in a seminal paper written by Holden and Risebro \cite{PDETraffic1} and was later simplified tremendously by Coclite et al. \cite{PDETraffic2} with the introduction of a preference matrix at each road junction. The model as described in Coclite et al. \cite{PDETraffic2} serves as the mathematical basis for modeling traffic flow in Lahaina.

\subsubsection{Definitions and Terminologies}

We model the road network as a finite, connected directed graph (possibly with edges extending to infinity), where directed edges represent roads and vertices represent junctions (Figure~\ref{Fig1}).

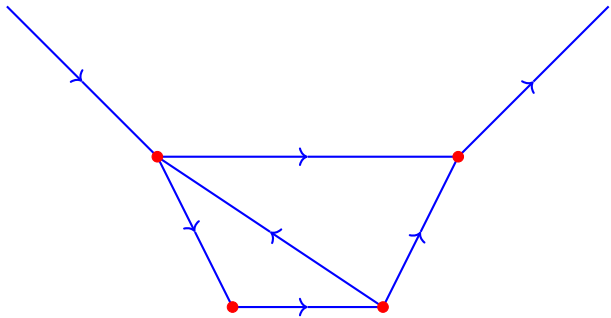
\begin{figure}
\centering
\begin{tikzpicture}[scale=1]
\begin{scope}
\draw[->,thick,blue] (0,4) to (1,3); \draw[thick,blue] (1,3) to (2,2);
\draw[->,thick,blue] (2,2) to (4,2); \draw[thick,blue] (4,2) to (6,2);
\draw[->,thick,blue] (6,2) to (7,3); \draw[thick,blue] (7,3) to (8,4);
\draw[->,thick,blue] (2,2) to (2.5,1); \draw[thick,blue] (2.5,1) to (3,0);
\draw[->,thick,blue] (5,0) to (5.5,1); \draw[thick,blue] (5.5,1) to (6,2);
\draw[->,thick,blue] (3,0) to (4,0); \draw[thick,blue] (4,0) to (5,0);
\draw[->,thick,blue] (5,0) to (3.5,1); \draw[thick,blue] (3.5,1) to (2,2);
\draw[red,fill=red] (2,2) circle (2pt);
\draw[red,fill=red] (6,2) circle (2pt);
\draw[red,fill=red] (3,0) circle (2pt);
\draw[red,fill=red] (5,0) circle (2pt);
\end{scope}
\end{tikzpicture}
\caption{An example of a traffic network of interest. Here, the scalar conservation law on each road is solved away from the junction, indicated by the {\color{blue} blue directed edges}. On the other hand, junctions are labelled with {\color{red}red vertices}. Determining the density of cars at the junction would require resolving the junction conditions.}
\label{Fig1}
\end{figure}

Suppose there are $N$ roads and $M$ junctions. On road $i$, the density $\rho_i(t,x)$ with $x\in[a_i,b_i]$ (traffic flows from $a_i$ to $b_i$) satisfies
\begin{equation}\label{PDE1}
\p_t \rho_i + f_i(\rho_i)_x = 0 \quad \text{ in } (t,x) \in (0,\infty) \times [a_i,b_i], \text{ for each } i = 1,\cdots, N.
\end{equation}
with flux $f_i:[0,\rho_{\text{max},i}]\to\mathbb{R}$ satisfying
\begin{itemize}[leftmargin = 20pt]
\item[$(\mathcal{F})$] For each $i = 1,\cdots, N$,  $f_i$ is Lipschitz continuous and strictly concave function with $f(0) = f(\rho_{\text{max},i}) = 0$.
\end{itemize}
Let $\sigma_i$ be the unique maximizer of $f_i$. Condition $(\mathcal{F})$ implies the existence of a symmetrizing function $\tau_i:[0,\rho_{\text{max},i}] \rightarrow [0,\rho_{\text{max},i}]$ satisfying
\begin{itemize}[leftmargin=*]
\item $f_i(\tau_i(\rho_i)) = f_i(\rho_i)$ for each $\rho_i \in [0,\rho_{\text{max},i}]$, and
\item $\tau_i(\rho_i) \neq \rho_i$ for each $\rho_i \in [0,\rho_{\text{max},i}] \setminus \{\sigma_i\}$.
\end{itemize}
A \textbf{weak solution} to \eqref{PDE1} is defined in the standard distributional sense. Each $\rho_i$ satisfies the conservation law in the sense of distributions on its road segment,
\begin{equation}\label{PDE2}
\int_0^\infty \int_{a_i}^{b_i} \ml \rho_i \frac{\p \varphi}{\p t} + f_i(\rho_i)\frac{\p \varphi}{\p x}\mr (t,x) \D x \D t = 0
\end{equation}
for all $\varphi \in C^\infty_c((0,\infty)\times[a_i,b_i])$, with an analogous distributional identity coupling the roads at each junction.
\begin{equation}\label{PDE3}
\sum_{i=1}^{n+m} \int_0^\infty \int_{a_i}^{b_i} \ml \rho_i \frac{\p \varphi_i}{\p t} + f_i(\rho_i)\frac{\p \varphi_i}{\p x}\mr (t,x) \D x \D t = 0.
\end{equation}
Here we index the $n$ incoming and $m$ outgoing roads at a junction $J_j$ by $i=1,\dots,n$ and $i=n{+}1,\dots,n{+}m$, respectively, and require the test functions $\varphi_i$ to be smooth across the junction. As in the single-road case \citep{Evans, LeVeque}, the Rankine-Hugoniot conditions yield the \textbf{Kirchhoff law} at each junction. For a.e.\ $t>0$,
\begin{equation}\label{PDE4}
\sum_{i=1}^n f_i(\rho_i(t,b_i)) = \sum_{i=n+1}^{n+m} f_i(\rho_i(t,a_i)).
\end{equation}

Since a junction with $n$ incoming and $m$ outgoing roads has $n+m$ unknown boundary flux values but only the single conservation constraint \eqref{PDE4}, the system is underdetermined. Determining the flux values at the junction uniquely is essential for obtaining a unique solution to the Cauchy problem for \eqref{PDE1}. This motivates imposing additional entropy and junction conditions, which we describe next.

\subsubsection{Entropy and Junction Conditions}

Next, we will extend the notion of a weak solution to a \textbf{weak entropic solution} by demanding that the weak solution to the system satisfies additional conditions, which we denote as the \textbf{entropy conditions}  $(\mathcal{E})$. This is done in a manner that generalizes the formulation for Riemann problems as in \citep{PDETraffic2}.
Fix any $t > 0$ and a junction with $n$ incoming and $m$ outgoing roads, and consider the following notations:
\begin{itemize}[leftmargin=*]
\item $\rho_{i,0}$ denote the density of cars right before the junction (that is, $\rho_{i,0} = \rho_i(t,b_i^-)$ along incoming roads $i = 1,\cdots, n$ and $\rho_{i,0} = \rho_i(t,a_i^+)$ along outgoing roads $i=1,\cdots,m$),
\item $\gamma_{i,0} := f_i(\rho_{i,0})$ be the incoming/outgoing flux along the respective incoming and outgoing roads, and
\item $\hat{\gamma}_i$ be the assigned fluxes and $\hat{\rho}_i$ be the corresponding density upon resolving the relevant entropy and junction conditions at each junction along the respective roads.
\end{itemize}

In the notation above, we can express the Kirchhoff Law in \eqref{PDE4} as 
\begin{equation}\label{PDE17}
\sum_{i=1}^n \hat{\gamma}_j= \sum_{j=n+1}^{n+m} \hat{\gamma}_i.
\end{equation}

\begin{enumerate}[leftmargin = *]
\item Along each road $i$, the typical entropy conditions for a scalar conservation law holds. Mathematically, for every $k \in \mathbb{R}$ and $\tilde{\varphi} \in C^\infty_c([0,\infty) \times [a_i,b_i])$ and $\tilde{\varphi} > 0$, we have
$$\int_0^\infty \int_{a_i}^{b_i} \ml |\rho_i - k| \frac{\p \tilde{\varphi}}{\p t} + \text{sgn}(\rho_i - k)(f_i(\rho_i)-f_i(k))\frac{\p \tilde{\varphi}}{\p x}\mr (t,x) \D x \D t \geq 0.$$
If $f$ is not differentiable at $\sigma$, by the entropy conditions, we take the derivative from the left - that is - $f'(\sigma) := f'(\sigma^-).$
\item Waves along incoming roads must have negative velocity and waves along outgoing roads must have positive velocity. This is physically consistent with the fact that information about a traffic jam (or clearance) at the junction would propagate backward from the junction along incoming roads and forward from the junction along outgoing roads. This thus restricts the corresponding choices of the flux functions and densities at the junction along each road. For each junction $J_j$ for $j = 1,\cdots, M$, with $n$ incoming roads and $m$ outgoing roads, we have for each $i = 1,\cdots, n$ (along incoming roads),
\begin{equation}\label{PDE5}
\hat{\rho_i} \in \begin{cases}
\{\rho_{i,0}\} \cup \left( \tau_i(\rho_{i,0}),\rho_{\text{max},i}\right] &\text{ if } \rho_{i,0} \in [0,\sigma_i], \\
[\sigma_i,\rho_{\text{max},i}] &\text{ if } \rho_{i,0} \in [\sigma_i,\rho_{\text{max},i}], \\
\end{cases}
\end{equation}
and for each $j = n+1,\cdots, n+m$ (along outgoing roads), we have
\begin{equation}\label{PDE6}
\hat{\rho_j} \in \begin{cases}
[0,\sigma_j] &\text{ if } \rho_{j,0} \in [0,\sigma_j], \\
[0,\tau_j(\rho_{j,0})) \cup \{\rho_{j,0}\} &\text{ if } \rho_{j,0} \in [\sigma_j,\rho_{\text{max},j}]. \\
\end{cases}
\end{equation}
Equivalently, for the assigned fluxes at the junctions, we have for each $i = 1,\cdots, n$ (along incoming roads), 
\begin{equation}\label{PDE7}
\hat{\gamma}_i \in \begin{cases}
[0,\gamma_{i,0}] &\text{ if } \rho_{i,0} \in [0,\sigma_i], \\
[0,f_i(\sigma_i)] &\text{ if } \rho_{i,0} \in [\sigma_i,\rho_{\text{max},i}], \\
\end{cases}
\end{equation}
and for each $j = n+1,\cdots, n+m$ (along outgoing roads), we have
\begin{equation}\label{PDE8}
\hat{\gamma}_j \in \begin{cases}
[0,f_j(\sigma_j)] &\text{ if } \rho_{j,0} \in [0,\sigma_j], \\
[0,\gamma_{j,0}] &\text{ if } \rho_{j,0} \in [\sigma_j,\rho_{\text{max},j}]. \\
\end{cases}
\end{equation}
For notational convenience, we denote $\hat{\gamma}_i \in [0,c_i]$ for each $i = 1,\cdots,n+m$ with the flux capacity $c_i$ to be either $\gamma_{i,0}$ or $f_i(\sigma_i)$ depending on the scenarios as described in \eqref{PDE7} and \eqref{PDE8}, and define
$$\Omega := \prod_{i=1}^{n+m}[0,c_i]$$
to be the admissible set representing the admissible fluxes at the junction along each road.
\item Conditions at the junction. Here, we assume that the following happens at the junction:
\begin{itemize}
\item[(A)] There are some \textbf{prescribed preferences of drivers}, that is, the traffic from incoming roads is distributed onto outgoing roads according to fixed coefficients. Mathematically, for each junction $J_j$ for $j = 1,\cdots, M$ with $n$ incoming and $m$ outgoing roads, we fix a traffic distribution matrix $A = \{\alpha_{ji}\}_{j=n+1,\cdots,n+m;i = 1,\cdots,n}\in \mathcal{M}_{m \times n}(\mathbb{R})$ such that
\begin{equation}\label{PDE9}
0 < \alpha_{ji} < 1, \quad \text{ and } \quad \sum_{j=n+1}^{n+m} \alpha_{ji} = 1.
\end{equation}
This implies that
\begin{equation}\label{PDE11}
\begin{pmatrix}
\hat{\gamma}_{n+1} \\ \vdots \\ \hat{\gamma}_{n+m} 
\end{pmatrix}
= A \begin{pmatrix}\hat{\gamma}_1 \\ \vdots \\ \hat{\gamma}_n \end{pmatrix}.
\end{equation}
Furthermore, the matrix $A$ must have the property that the row vector $(1,\cdots,1) \in \mathbb{R}^n$ is not in the span of any $n-1$ combinations of either rows of $A$ or standard basis elements in $\mathbb{R}^n$. The reader can easily verify that imposing \eqref{PDE11} implies that \eqref{PDE17} is automatically satisfied. For more information on this technical condition, refer to \citep{PDETraffic2}. 
\item[(B)] There is a \textbf{right of way} to the traffic at the junction if the junction is packed. In other words, if there is an outgoing road such that the total maximum flux allowed on it is smaller than the total maximum flux supplied by incoming roads (in accordance to the drivers' preferences in \eqref{PDE11}), not all drivers on each of the incoming roads would be able to pass through the junction at the next instance. Mathematically, this is represented by the following vector inequality being true for at least one of the components:
\begin{equation}\label{PDE13}
\begin{pmatrix}
c_{n+1} \\ \vdots \\ c_{n+m} 
\end{pmatrix}
< A \begin{pmatrix}c_1 \\ \vdots \\ c_n \end{pmatrix}.
\end{equation}
In this scenario, the drivers' preferences are ignored, and the right of way is activated instead. We assume that the incoming fluxes and the outgoing fluxes scale proportional to the incoming and outgoing capacities, that is,
\begin{equation}\label{PDE16}
\begin{pmatrix}
\hat{\gamma}_{1} \\ \vdots \\ \hat{\gamma}_{n} 
\end{pmatrix}
= \mu \begin{pmatrix}c_{1} \\ \vdots \\ c_{n} \end{pmatrix} \quad \text{ and } \quad \begin{pmatrix}
\hat{\gamma}_{n+1} \\ \vdots \\ \hat{\gamma}_{n+m} 
\end{pmatrix}
= \lambda \begin{pmatrix}c_{n+1} \\ \vdots \\ c_{n+m} \end{pmatrix}
\end{equation}
for some constants $\mu,\lambda > 0$ to be optimized.
\item[(C)] Respecting (A) and (B), the incoming drivers maximize fluxes at the junction for a given preference matrix $A$. Mathematically, we can express this as a conditional optimization as follows:
\begin{equation}\label{PDE12}
\begin{cases}
\text{If not } \eqref{PDE13}, \quad 
&\begin{aligned} &\underset{(\hat{\gamma}_1,\cdots,\hat{\gamma}_{n+m})}{\text{maximize}} && \,  \sum_{i=1}^n \hat{\gamma}_i\\
&\text{subject to } &&\, (\hat{\gamma}_1,\cdots,\hat{\gamma}_{n+m}) \in \Omega, \text{ and } \eqref{PDE11}.
\end{aligned} \\ 
\text{Else, } \quad
&\begin{aligned} &\underset{(\lambda,\mu)}{\text{maximize}} && \,  \sum_{i=1}^n \hat{\gamma}_i\\
&\text{subject to } &&\, \eqref{PDE17}, \eqref{PDE16}, \text{ and } (\lambda,\mu) \in [0,1]^2.
\end{aligned}
\end{cases}
\end{equation}
\end{itemize}
\end{enumerate}
The solution to the optimization problem is given in the theorem below.

\begin{theorem}\label{opt-sol}
Let 
\begin{equation}\label{opt-sol1}
r_j := c_{n+j}^{-1} A_j \begin{pmatrix} c_1 \\ \vdots \\ c_n \end{pmatrix}
\end{equation}
and hence
\begin{equation}\label{opt-sol2}
r := \max_{1 \le j \le m} r_j,
\end{equation}
where $A_j$ is the $j$th row of $A$, and $c_{\text{in}} = \sum_{i=1}^n c_i$ and $c_{\text{out}} = \sum_{j=n+1}^{n+m} c_j$ to represent the total incoming and outgoing capacities. The unique solution to \eqref{PDE12} is as follows.
\begin{enumerate}[(i)]
\item If $r \leq 1$, \begin{equation}\label{PDE18}
\begin{pmatrix}
\hat{\gamma}_{1} \\ \vdots \\ \hat{\gamma}_{n} 
\end{pmatrix}
=  \begin{pmatrix}c_{1} \\ \vdots \\ c_{n} \end{pmatrix} \quad \text{ and } \quad \begin{pmatrix}
\hat{\gamma}_{n+1} \\ \vdots \\ \hat{\gamma}_{n+m} 
\end{pmatrix}
= A \begin{pmatrix}\hat{\gamma}_1 \\ \vdots \\ \hat{\gamma}_n \end{pmatrix}.
\end{equation}
\item Else if $r > 1$ and $c_{\text{in}} \leq c_{\text{out}}$, we have $\lambda = 1, \mu = \frac{c_{\text{in}}}{c_{\text{out}}},$ and 
\begin{equation}\label{PDE19}
\begin{pmatrix}
\hat{\gamma}_{1} \\ \vdots \\ \hat{\gamma}_{n} 
\end{pmatrix}
=  \begin{pmatrix}c_{1} \\ \vdots \\ c_{n} \end{pmatrix} \quad \text{ and } \quad \begin{pmatrix}
\hat{\gamma}_{n+1} \\ \vdots \\ \hat{\gamma}_{n+m} 
\end{pmatrix}
= \frac{c_{\text{in}}}{c_{\text{out}}} \begin{pmatrix}c_{n+1} \\ \vdots \\ c_{n+m} \end{pmatrix}.
\end{equation}
\item Otherwise with $r > 1$ and $c_{\text{in}} > c_{\text{out}}$,  we have $\lambda = \frac{c_{\text{out}}}{c_{\text{in}}}, \mu = 1,$ and 
\begin{equation}\label{PDE20}
\begin{pmatrix}
\hat{\gamma}_{1} \\ \vdots \\ \hat{\gamma}_{n} 
\end{pmatrix}
=  \frac{c_{\text{out}}}{c_{\text{in}}} \begin{pmatrix}c_{1} \\ \vdots \\ c_{n} \end{pmatrix} \quad \text{ and } \quad \begin{pmatrix}
\hat{\gamma}_{n+1} \\ \vdots \\ \hat{\gamma}_{n+m} 
\end{pmatrix}
= \begin{pmatrix}c_{n+1} \\ \vdots \\ c_{n+m} \end{pmatrix}.
\end{equation}
\end{enumerate}
\end{theorem}

\begin{proof}[Proof sketch]
In case~(i), $r\leq 1$ ensures that setting incoming fluxes to their capacities respects all outgoing capacity constraints, so the drivers' preferences can be followed. In cases~(ii) and~(iii), the preferences are infeasible ($r>1$) and the right-of-way regime activates. The Kirchhoff constraint \eqref{PDE17} then uniquely determines the scaling constants $\lambda$ and $\mu$ depending on which side --- incoming or outgoing --- is the bottleneck. See Appendix~\ref{appendix:theorem_proof} for the full proof.
\end{proof}
\noindent In cases (ii) and (iii), observe by comparing \eqref{PDE19} and \eqref{PDE20} with \eqref{PDE11}, the assigned fluxes $(\hat{\gamma}_1,\cdots,\hat{\gamma}_{n+m})$ do not respect the drivers' preferences as in \eqref{PDE18}. This is not surprising given that the drivers have to ``give way" and not follow through with their original preferences to maximize the fluxes in the outgoing road to optimize for flow at the junction. 

Next, we remark that the entropy condition in \eqref{PDE12} has some positive properties. First, having a unique solution implies that there would be a unique assignment of fluxes at the junction on each road connected to the junction, which would support the well-posedness of the numerical and analytic solutions to the conservation law in \eqref{MathCont3} with the right initial data. Furthermore, this condition encourages drivers to change their preferences in the event of a congested junction in an attempt to maximize flux through the junction, which is consistent with most physical scenarios, especially for evacuation purposes. In addition, this condition is vastly different from the first-in first-out (FIFO)/invariance conditions imposed by most papers in the literature, such as in \citep{DelleMonacheGoatinPiccoli2018PriorityRiemann} and \citep{Tampere2011GenericNodeModels}, in which the drivers' preferences are to be respected. Hence, there would be scenarios in which there is a vastly empty outgoing road and a heavily congested outgoing road, but the empty outgoing road would not be utilized as the incoming flux from each road are throttled proportionately to respect drivers' preferences, which in turn limits the flux of cars through the relatively empty outgoing road.

Last but not least, under some additional mild conditions, we end off by quoting the following theorem from Coclite et al. \cite{PDETraffic2} for smooth initial data:
\begin{theorem} 
Let $f_i \equiv f$ for some common flux function $f$, with a common maximum density $\rho_{\text{max}}$. Let $\rho_i(0,x) \in [0,\rho_{\text{max}}]$ for each $x \in [a_i,b_i]$ and for each $i = 1,\cdots, N$ such that $\rho_i(0,x)$ is of bounded variation on $[a_i,b_i]$. If there are equal or more outgoing roads compared to incoming roads, then a unique weak entropic solution to \eqref{PDE1} exists.
\end{theorem} 
\noindent From the proof in \citep{PDETraffic2}, it is not difficult to see that this generalizes to arbitrary flux functions satisfying $(\mathcal{F})$ as long as the junction conditions determine unique fluxes at the junction along each road, which is true for our version of the entropy conditions with the corresponding optimization problem as described in \eqref{PDE12}.

\subsubsection{Numerical Scheme}\label{numerical-scheme}

In this section, we present the numerical scheme used to simulate the evolution of traffic flow on networks of interest. We employ Godunov's method with an exact Riemann solver to evolve the internal densities along each road as described in \citep{NumericalTraffic}.

We restrict our attention to flux functions $f:[0,\rho_{\max}] \to \mathbb{R}_+$ that are continuous, piecewise smooth (piecewise $C^2$), and piecewise-quadratic in the following sense (the motivation for this particular form will be explained in Section 4): there exists $\sigma \in (0,\rho_{\max})$ such that
\begin{equation}\label{FluxPiecewise}
f(\rho) = \begin{cases}
v_f \rho, & \rho \in [0,\sigma] \\
A(\rho - \sigma)^2 + f_{\max}, & \rho \in [\sigma, \rho_{\max}]
\end{cases}
\end{equation}
where $v_f > 0$, $A < 0$, and $f_{\max} > 0$. This results in a concave flux with $f(0) = f(\rho_{\max}) = 0$ and a unique maximum at $\rho = \sigma$. The characteristic speed $f'(\rho)$ is constant ($f'(\rho) = v_f$) on the linear branch and strictly decreasing ($f'(\rho) = 2A(\rho - \sigma) < 0$) on the quadratic branch. Godunov's method requires only that the flux be Lipschitz continuous \citep{LeVeque2}, which is satisfied here. Similar piecewise quadratic flux functions have been successfully employed in traffic flow models \citep{Lu2008PiecewiseQuadratic}, where entropy solutions are constructed via the convex hull method.

First, we consider the evolution of \eqref{PDE1} on a given finite road $[a_i,b_i]$ in the absence of a junction with flux function $f$. We consider the numerical grid with the following notations:
\begin{itemize}[leftmargin = *]
\item $\Delta x$ be the spatial grid size,
\item $\Delta t$ be the temporal grid size, 
\item $(t_n,x_m) = (n\Delta t, m\Delta x)$ with $n \in \mathbb{N}$ and $m \in \mathbb{Z}$ be the grid points of interest, 
\item $M = \frac{b_i - a_i}{\Delta x}$ be the number of interior cells (excluding ghost cells), and
\item As a finite volume method, $\rho^n_m$ denotes the numerical density of cars at $(t_n,x_m)$ that is assumed to be constant on $[x_{m-1/2},x_{m+1/2}]$.
\end{itemize} 
Consider the $m$-th cell average for density given by
$$\rho_m(t) = \frac{1}{\Delta x}\int_{x_{m} - \Delta x/2 }^{x_{m} + \Delta x/2} \rho(x,t) \D x = \frac{1}{\Delta x}\int_{x_{m-1/2} }^{x_{m+1/2}} \rho(x,t) \D x.$$
Differentiating the above equation in time, applying \eqref{PDE1}, and integrating it from $t = t_n$ to $t_{n+1}$, we have
\begin{equation}\label{Scheme3}
\frac{\rho_{m}^{n+1} - \rho_{m}^n }{\Delta t}  = \frac{1}{\Delta x}\ml F(\rho_{m}^n,\rho_{m+1}^n) -  F(\rho_{m-1}^{n},\rho_m^{n})\mr.
\end{equation}
In the expression above, we can view the numerical flux term $F^n_{m+1/2}:= F(\rho^n_m,\rho^n_{m+1})$ as the flux between cells with density $\rho^n_m$ on $[x_{m-1/2},x_{m+1/2}]$ and $\rho^n_{m+1}$ on $[x_{m+1/2},x_{m+3/2}]$ at the cell boundary $x_{m+1/2}$, with positive flux defined as flux going into the cell $[x_{m+1/2},x_{m+3/2}]$.

\noindent \textbf{Exact Riemann Solver.} The numerical flux is computed by solving the Riemann problem exactly at each cell interface. For our piecewise quadratic flux function $f$, the exact solution follows from the convex hull construction \citep{LeVeque2, Lu2008PiecewiseQuadratic}:
\begin{equation}\label{Scheme4}
F^{n}_{m+1/2} = \begin{cases}
\min_{\rho \in [\rho^n_m,\rho^n_{m+1}]} f(\rho) &\text{ if } \rho^n_{m} \leq \rho^n_{m+1} \\
\min_{\rho \in [\rho^n_{m+1},\rho^n_{m}]}  f(\rho) &\text{ if } \rho^n_{m} \geq \rho^n_{m+1}.
\end{cases}
\end{equation}
For the piecewise linear-quadratic flux \eqref{FluxPiecewise}, this reduces to the following cases:
\begin{itemize}[leftmargin = *]
    \item If $\rho^n_m < \rho^n_{m+1} \leq \sigma$, then both states lie on the linear branch where $f'(\rho) = v_f > 0$ is constant. Since the characteristic speed is positive and does not vary across the states, we obtain a contact discontinuity (not a rarefaction), and information propagates from left to right. Thus, $F^n_{m+1/2} = f(\rho^n_m) = v_f \rho^n_m$, taking the upstream value.
    \item If $\sigma \leq \rho^n_m < \rho^n_{m+1}$, then both states lie on the quadratic branch where $f'(\rho) = 2A(\rho - \sigma) < 0$ is strictly decreasing with $\rho$. Since $\rho^n_m < \rho^n_{m+1}$ implies $f'(\rho^n_m) > f'(\rho^n_{m+1})$ and both characteristic speeds are negative, the characteristics point from right to left. This results in a shock wave with $F^n_{m+1/2} = f(\rho^n_{m+1})$, taking the upstream value from the right.
    \item If $\rho^n_m \leq \sigma < \rho^n_{m+1}$, then the states straddle the critical density $\sigma$. The flux is determined by the Rankine-Hugoniot shock speed $s = (f(\rho^n_{m+1}) - f(\rho^n_m))/(\rho^n_{m+1} - \rho^n_m)$; if $s > 0$, then $F^n_{m+1/2} = f(\rho^n_m)$, otherwise $F^n_{m+1/2} = f(\rho^n_{m+1})$.
    \item If $\rho^n_{m+1} < \rho^n_m$, then the characteristic directions are reversed, and symmetric cases apply to determine the minimum flux.
\end{itemize}
This construction ensures the numerical solution respects the characteristic structure and satisfies the entropy condition.
Consequently, rearranging \eqref{Scheme3}, we obtain
\begin{equation}\label{Scheme6}
\rho^{n+1}_m = \rho^n_m - \frac{\Delta t}{\Delta x}\ml F^n_{m+1/2} - F^n_{m-1/2} \mr.
\end{equation}
In addition, the choice of temporal and spatial step size must satisfy the CFL (Courant-Friedrichs-Lewy) condition:
\begin{equation}\label{Scheme7}
\Delta t \leq \nu \frac{\Delta x}{\max_{\rho}\mlm f'(\rho) \mrm},
\end{equation}
where $\nu$ is the CFL number. This condition ensures that only waves from neighboring Riemann problems interact in one time step. In our implementation, we use $\nu = 0.5$ for numerical stability.

\noindent \textbf{Boundary Conditions.} For a finite road, at both ends (i.e., at $x = a_i$ on the left and $x = b_i$ on the right), the road must either connect to a junction or have a boundary condition applied. Physically, each boundary represents one of three scenarios: (i) the road extends infinitely beyond the computational domain, (ii) the road has a prescribed incoming traffic density, or (iii) the road is connected to a junction. We distinguish two cases in our implementation:

\begin{itemize}[leftmargin = *]
\item \textit{Roads not connected to junctions.} These boundaries use either prescribed time-dependent boundary data or non-reflecting conditions. Non-reflecting boundary conditions are designed to simulate infinitely long roads extending beyond the computational domain, allowing traffic waves to exit the domain without artificial reflections that would contaminate the interior solution. A \emph{wave speed causality check} is performed, with prescribed boundary data imposed only when the characteristic speed points inward (i.e., when information can physically flow from the boundary into the domain). Specifically:
\begin{itemize}[leftmargin = 10pt]
    \item At the left boundary ($x = a_i$): if $f'(\rho^n_1) < 0$ (characteristic points outward), we impose the non-reflecting condition $\rho^n_0 = \rho^n_1$ to simulate an infinitely long road to the left; otherwise, prescribed data is used.
    \item At the right boundary ($x = b_i$): if $f'(\rho^n_M) > 0$ (characteristic points outward), we impose the non-reflecting condition $\rho^n_{M+1} = \rho^n_M$ to simulate an infinitely long road to the right; otherwise, prescribed data is used.
\end{itemize}
If no boundary data is specified, non-reflecting conditions $\rho^n_{0} = \rho^n_{1}$ or $\rho^n_{M+1} = \rho^n_M$ are imposed by default \citep{LeVeque2}.

\item \textit{Roads connected to junctions.} The boundary densities at junction interfaces are determined by the junction resolution procedure described in Section 3.1, which solves the entropy maximization problem \eqref{PDE12} with traffic distribution preferences to compute fluxes, then uses the inverse flux functions to update boundary densities.
\end{itemize}

\noindent The well-posedness of the initial-boundary value problem is ensured by respecting the direction of characteristic waves when imposing boundary conditions. By only prescribing boundary data when the characteristic speed points inward, and using non-reflecting conditions when it points outward, we prevent the system from being over- or under-determined. This approach, combined with the entropy-satisfying Godunov flux in the interior, guarantees unique, stable numerical solutions \citep{LeVeque2}. \\ 

\noindent \textbf{Supply-Demand Formulation for Junction Coupling.} Following the framework introduced by Lebacque \cite{Lebacque1996Godunov} and Daganzo \cite{Daganzo1994CTM}, the maximum admissible flux at junction boundaries is determined by supply and demand functions. In traffic flow terminology, \emph{demand} refers to the flux that drivers wish to send into the junction from incoming roads, while \emph{supply} refers to the capacity to receive flux on outgoing roads. For a road entering a junction (incoming road) ending at position $b_i$, the junction interface is located at $x_{M+1/2} = b_i$, and we evaluate the \emph{demand} (or sending function) at the last interior cell center $x_M$. For a road leaving a junction (outgoing road) starting at position $a_i$, the junction interface is at $x_{1/2} = a_i$, and we evaluate the \emph{supply} (or receiving function) at the first interior cell center $x_1$.

For our piecewise linear-quadratic flux, the demand and supply functions are:

\noindent \textit{Demand} $D$ (incoming road at cell center $x_M$ before the junction interface $x_{M+1/2}$, with density $\rho^n_M$):
\begin{equation}\label{Demand}
D(\rho^n_{M}) = \begin{cases}
\min\left(f(\rho^n_{M}), \dfrac{\rho^n_{M}}{\Delta t}\right) & \text{if } \rho^n_{M} \leq \sigma, \\[8pt]
f_{\max} & \text{if } \rho^n_{M} > \sigma.
\end{cases}
\end{equation}

\noindent \textit{Supply} $S$ (outgoing road at cell center $x_1$ after the junction interface $x_{1/2}$, with density $\rho^n_1$):
\begin{equation}\label{Supply}
S(\rho^n_1) = \begin{cases}
f_{\max} & \text{if } \rho^n_1 \leq \sigma, \\[8pt]
\min\left(f(\rho^n_1), \dfrac{\rho_{\max} - \rho^n_1}{\Delta t}\right) & \text{if } \rho^n_1 > \sigma.
\end{cases}
\end{equation}

The rationale for these constraints is as follows. In the demand function, when $\rho^n_{M} \leq \sigma$ (free flow), we limit the flux to prevent extracting more vehicles than are available in one time step. In the supply function, when $\rho^n_1 > \sigma$ (congested), we limit the flux to prevent injecting more vehicles than the remaining capacity allows in one time step. These constraints ensure numerical stability and are consistent with the discretization of the Godunov scheme \citep{Lebacque1996Godunov, Daganzo1994CTM}.

Importantly, the demand and supply functions \textbf{further restrict} the entropy admissible flux values $c_i$ and $c_j$ from \eqref{PDE7} and \eqref{PDE8}. For incoming road $i$, the CTM discretization modifies $c_i$ to:
\begin{equation}\label{CTMDemand}
c_i^{\mathrm{CTM}} = \begin{cases}
\min\left(c_i, \dfrac{\rho^n_M}{\Delta t}\right) & \text{if } \rho^n_M \leq \sigma, \\[8pt]
c_i & \text{if } \rho^n_M > \sigma,
\end{cases}
\end{equation}
and for outgoing road $j$, the CTM discretization modifies $c_j$ to:
\begin{equation}\label{CTMSupply}
c_j^{\mathrm{CTM}} = \begin{cases}
c_j & \text{if } \rho^n_1 \leq \sigma, \\[8pt]
\min\left(c_j, \dfrac{\rho_{\max} - \rho^n_1}{\Delta t}\right) & \text{if } \rho^n_1 > \sigma.
\end{cases}
\end{equation}
While the entropy conditions provide continuous admissible flux values based on characteristic theory, the CTM discretization constraints ensure the discrete scheme remains stable and converges to the entropy solution. The junction resolution uses these CTM-modified values $c_i^{\mathrm{CTM}}$ and $c_j^{\mathrm{CTM}}$ in place of $c_i$ and $c_j$.

The complete numerical method for solving traffic flow on a network is summarized in Algorithm \ref{alg:num-scheme-network}, which integrates all the components described above.

\begin{algorithm2e}[ht]
\caption{Numerical Scheme for Traffic Flow on a Network}
\label{alg:num-scheme-network}
\SetKwInOut{Input}{Input}
\SetKwInOut{Output}{Output}
\SetKwComment{Comment}{}{}
\DontPrintSemicolon
\BlankLine
\Input{ [List of Roads], [List of Junctions], global time step $\Delta t$, number of time steps $N$.}
\BlankLine
\textbf{Initialization:}\;
\For{Road in [List of Roads]}{
    Initialize flux function \eqref{FluxPiecewise}, road length, and number of lanes\;
    Set spatial grid $\Delta x$ according to CFL condition \eqref{Scheme7} with $\nu = 0.5$\;
    Attach each end to either a junction or specify boundary condition type;
}
\BlankLine
\textbf{Time Evolution:}\;
\For{$n \leftarrow 1 \; \KwTo \; N$}{
    \For{Road in [List of Roads]}{
        \textbf{Flux computation:} Compute Godunov fluxes $F^n_{m+1/2}$ for $m = 0, \ldots, M$ via \eqref{Scheme4} using densities $\rho^n_m$ and $\rho^n_{m+1}$. For roads connected to junctions, replace boundary fluxes $F^n_{1/2}$ or $F^n_{M+1/2}$ with stored junction fluxes $\gamma$ from previous time step\;
        \textbf{Density update:} Update interior cell densities $\rho^{n+1}_m$ for $m = 1, \ldots, M$ via \eqref{Scheme6}\;
        \textbf{Boundary conditions:} For roads not connected to junctions, update ghost cells $\rho^{n+1}_0$ or $\rho^{n+1}_{M+1}$ using boundary conditions (non-reflecting or prescribed) with wave speed check\;
    }
    \For{Junction in [List of Junctions]}{
        \textbf{Demand/Supply:} Compute CTM-restricted fluxes $c_i^{\mathrm{CTM}}$ via \eqref{CTMDemand} for incoming roads and $c_j^{\mathrm{CTM}}$ via \eqref{CTMSupply} for outgoing roads\;
        \textbf{Junction resolution:} Solve entropy maximization problem \eqref{PDE12} using $c_i^{\mathrm{CTM}}$, $c_j^{\mathrm{CTM}}$ to obtain junction fluxes $\gamma_i$, $\gamma_j$\;
        \textbf{Ghost cell update:} For incoming road $i$, set $\rho^{n+1}_{M+1} = f^{-1}(\gamma_i)$ where inverse is on branch with $f'(\rho) > 0$. For outgoing road $j$, set $\rho^{n+1}_0 = f^{-1}(\gamma_j)$ where inverse is on branch with $f'(\rho) < 0$\;
        \textbf{Store junction fluxes:} Store $\gamma_i$ and $\gamma_j$ for use as boundary fluxes in the next time step\;
    }
}
\end{algorithm2e}

\section{Modelling Flux Functions for Fire Evacuation}\label{section:modeling}
In this section, we will describe the model for a flux fundamental diagram during a wildfire evacuation scenario.

\subsection{Real World Data and Motivation} In 2014, Dixit and Wolshon \cite{Hurricane} showed that hurricane evacuations exhibit statistically significant differences from regular traffic, including changes in free-flow speeds and maximum flow rates. More recently, data from the 2019 Kincade wildfire \citep{kincade_2019} and the 2020 Glass wildfire \citep{glass_2020} suggest that wildfire evacuations produce a reduction in speed at all densities and a decrease in maximum flow rates. The Kincade and Glass data indicate only modest reductions in free-flow speeds ($\approx 1$--$3.5$ km/hr) and maximum flow rates ($\approx 2\%$--$5\%$), whereas simulation-based studies such as Intini et al.\ \cite{INTINI2022103211} suggest much larger reductions depending on visibility conditions.

\subsection{Model Description}
As discussed in Section \ref{section:math}, the two main changes in traffic flow during a wildfire evacuation are a decreased maximum flow rate ($f_c$) and a reduced free-flow speed ($v_f$). Since existing wildfire data suggests these changes are relatively minor, we use conservative estimates for $v_f$ and $f_c$. Following the Highway Capacity Manual (HCM) \citep{HCM}, we set the free-flow speed equal to the posted speed limit, with speed limit data obtained from Maui County's Code of Ordinances \citep{Ordinances_2024}. To estimate maximum flow rates, we first classify each road according to its Federal Highway Administration (FHWA) designation using the District of Maui's 2035 Transportation Plan \citep{2035_plan}, and then according to its designation in the County of Maui's Street Design Manual \citep{street_design_manual}. Using both classifications, we estimate the maximum flux capacity of each road segment from values in Maui County's Proposed Roadway Development Program \citep{roadway_development}, the Lahaina Bypass Record of Decision \citep{bypass_record_of_dec}, and traffic station data from the State of Hawai'i Department of Transportation. See the Appendix for detailed road segment data and capacity computations.

We use a jam density of 
\begin{equation}\label{jam-density}
\rho_j = 200 \text{ veh/mi/lane}
\end{equation}
based on recommendations from the HCM, as well as the approximate scaling of passenger cars to vehicles and average length of a passenger car. To match the real-world traffic flow diagrams, such as those provided in the Kincade and Glass wildfires, we implement a piecewise-continuous flux function with two regimes separated by a capacity density $\sigma$, where the behavior is linear in the uncongested region, and quadratic in the congested region:
\begin{equation}
    f(\rho) = \begin{cases}
v_f \rho, \quad&0\leq \rho < \sigma\\
A(\rho - \sigma)^2 + f_{c}, \quad& \sigma < \rho \leq \rho_j
\end{cases} := \begin{cases}
f_1(\rho), \quad&0\leq \rho < \sigma\\
f_2(\rho), \quad& \sigma < \rho \leq \rho_j
\end{cases}
\end{equation}
Using the inputs $f_c, v_f, p_j$, we can solve for the coefficients:
\begin{gather*}
    (1): f_{-}(\sigma) = f_{c} \Rightarrow \sigma = \frac{f_{c}}{v_f}\\
(2): f_{+}(\rho_j) = 0 \Rightarrow A = -\frac{f_{c}}{(\rho_j - \sigma)^2}  
\end{gather*}
We would like to normalize the flux functions such that $\rho_j = 1$. Thus, we consider the transformation
\begin{gather}\label{flux_funct}
    \tilde{f}(\rho) = \frac{1}{\rho_j} f\left(\rho_j \rho\right)\\
    \Rightarrow \tilde{f}_1(\rho) = \frac{1}{\rho_j}(v_f \rho_j \rho) = v_f \rho,\\\tilde{f}_2(\rho) = \frac{A}{\rho_j}(\rho_j\rho - \sigma)^2 + \frac{f_c}{\rho_j} = \frac{A}{\rho_j}(\rho_j)^2(\rho - \frac{\sigma}{\rho_j})^2 + \frac{f_c}{\rho_j} = A\rho_j(\rho - \frac{\sigma}{\rho_j})^2 + \frac{f_{c}}{\rho_j}
    \\\Rightarrow \tilde{f}(\rho) = \begin{cases}
v_f \rho, \quad&0< \rho \leq \frac{\sigma}{\rho_j}\\
A\rho_j(\rho - \frac{\sigma}{\rho_j})^2 + \frac{f_{c}}{\rho_j}, \quad& \frac{\sigma}{\rho_j} < \rho \leq 1
\end{cases}
\end{gather}
Thus, if we consider the normalized coefficients 
\begin{gather*}
    \tilde{f}_c = \frac{f_c}{\rho_j}, \quad \tilde{\sigma} = \frac{\sigma}{\rho_j} = \frac{\tilde{f_c}}{v_f}, \quad \tilde{A} = A\rho_j = -\frac{f_c \cdot \rho_j}{\rho_j^2(1-\frac{\sigma}{\rho_j})^2} = - \frac{\tilde{f_c}}{(1-\tilde\sigma)^2}
\end{gather*}
the normalized flux functions are given by
\begin{gather}\label{norm_flux_func}
    \tilde{f}(\rho) = \begin{cases}
v_f \rho, \quad&0\leq \rho < \tilde{\sigma}\\
\tilde{A}(\rho - \tilde{\sigma})^2 + \tilde{f}_c, \quad& \tilde{\sigma} < \rho \leq 1
\end{cases}
\end{gather}
\subsubsection{Scaling for \texorpdfstring{$n$}{n} Lanes}
We consider the normalized flux functions as functions for roads with one lane (in one direction), and assume that roads with $n$ lanes simply scale both the road density and flow rate values by a factor of $n$, ie 
$$f_n(\rho) = n \tilde f\left(\frac{\rho}{n}\right)$$
so the scaled flux functions are given by 
\begin{gather}
    f_n(\rho) = \begin{cases}
v_f \rho, \quad&0\leq \rho < n\tilde{\sigma}\\
\frac{\tilde{A}}{n}(\rho - n\tilde{\sigma})^2 + n\tilde{f}_c, \quad& n\tilde{\sigma} < \rho \leq n
\end{cases}
\end{gather}
\subsubsection{Initial densities} \label{subsection:initial_densities} Let $\rho_0$ be the normalized initial density along a road segment with $n$ lanes, where $0 \leq \rho_0 \leq 1$. Since network simulations are run sequentially, only the AM Base Network requires initial density estimates. These are obtained by assigning each road segment a level of service (LOS) based on Google Maps typical traffic data \citep{googlemaps_traffic_lahaina} and annual average daily traffic (AADT) volumes \citep{aadt}, using HCM level-of-service classifications \citep{HCM, init_den}. The full procedure, including FHWA facility type determination and service volume table lookups, is described in Appendix \ref{appendix:roaddata}.

Given the LOS, we estimate the initial velocity $v_0$ using the average travel speed for the given LOS $v_{av}$,
\begin{equation}
    v_0 \approx v_{av} = \alpha v_f
\end{equation}
where $0<\alpha\leq 1$ is a scaling parameter applied to the free-flow speed, given by \citep{IDOT2012AppendixG}, with the modification that $\alpha = 1$ for LOS A, due to the non-smooth nature of our flux model. Then, we obtain our initial density for the road segment using its corresponding normalized flux function $\tilde f(\rho)$, along with the fundamental relationship $\tilde f(\rho_0) = v_0\rho_0$.

If the road is LOS A, then this corresponds to the uncongested/linear region of our flux, and we assume that our initial hourly flow is estimated from AADT as above, to obtain
\begin{gather}
    \rho_0 = \frac{1}{\rho_jv_f}\left( \frac{\text{AADT}\cdot 0.1\cdot 0.57}{n }\right) 
\end{gather}
If the road is LOS B, C, D, or E, this corresponds to the congested/quadratic region of our flux, and we have 
\begin{gather*}
    \rho_0 = \frac{\tilde f(\rho_0)}{v_0} = \frac{\tilde{A}(\rho_0 - \tilde{\sigma})^2 + \tilde{f}_c}{v_0} 
\end{gather*}
After simplification and substituting $v_0 = \alpha v_f$, we have
\begin{gather*}
    \rho_0 = \frac{\alpha v_f}{2\tilde{A}} + \tilde{\sigma} + \sqrt{\frac{\tilde f_c}{\tilde{A}}(\alpha - 1) + \left(\frac{\alpha v_f}{2\tilde{A}}\right)^2}
\end{gather*}
Exact computations of $\rho_0$ for all roadway segments in the AM Base Network can be found in Appendix \ref{appendix:initdata}.

\section{Optimizing Road Networks - An Application of Non-smooth Optimization}\label{section:optimization}

In this section, we discuss the possibility of optimizing the traffic flow network to aid evacuation. To do so, we consider a traffic road network with $N$ roads, $M$ junctions, and $P$ preference parameters $\boldsymbol{\alpha} := (\alpha_1 ,\cdots, \alpha_P)$, a fixed time elapsed $T > 0$ and consider the following loss function:
\begin{equation}\label{Opt-Network1}
L(\boldsymbol{\alpha}) = -\sum_{i} w_i \int_0^T \int_{a_i}^{b_i} \rho_i(x,t;\boldsymbol{\alpha}) \D x \D t.
\end{equation}
In \eqref{Opt-Network1} above, the loss function represents the weighted negative sum of the time-integrated number of cars on all roads for an appropriate choice of weights $w_i$. Minimizing the loss function is equivalent to maximizing the cars on the outgoing roads, as intended. The corresponding constrained optimization problem is thus given by 
\begin{equation}\label{Opt-Network2}
\begin{aligned}
\min_{\boldsymbol{\alpha}} \quad & L(\boldsymbol{\alpha}), \\
\text{subject to} \quad & \boldsymbol{\alpha} \in \mathcal{C} \\
\end{aligned}
\end{equation}
where the set $\mathcal{C}$ corresponds to
\begin{equation}\label{Opt-Network3}
\mathcal{C} = \{ \alpha_i \in [0,1]: \alpha_i \text{ satisfies } \eqref{PDE9} \text{ on its corresponding junction for } i \in \{1,2,\cdots,P\}\}
\end{equation}
with $P$ as the number of parameters to optimize for our given network. 

\subsection{Appropriate Choice of Weights} 

To assign appropriate weights on each road, we proceed as follows. We compute the distance between the exiting road (which we label as index $i = 1$) and every other road in the network. To do so, we assume that there is an exiting junction (labeled with index $j = 1$) that has only a single exiting road ($i = 1$) directed out of it. We then run the Dijkstra's algorithm to compute the distance of all other junctions to the exiting junction in our directed network, initializing the initial distance at the exiting junction as $1$. Since each road is directed into a unique junction, we define the distance $d_i$ of a given road $i = 2,\cdots,N$ to be the distance of the junction it is directed into. For the exiting road $i = 1$, we set $d_1 := 0$. Finally, we set the corresponding weights $w_i$ to be
\begin{equation}\label{weights}
w_i = 2^{-d_i}
\end{equation}
for each $i = 1,\cdots,N$.

For instance, in our AM Base Network, the distance of the other roads and other junctions to the exiting road, Hwy30[7], and junction, h5, respectively, are illustrated in Figure \ref{fig:assign_weights}.

\begin{figure}[htbp]
    \centering
    \subfloat[Illustrated network]{\includegraphics[width=0.4\textwidth,height=2.5in]{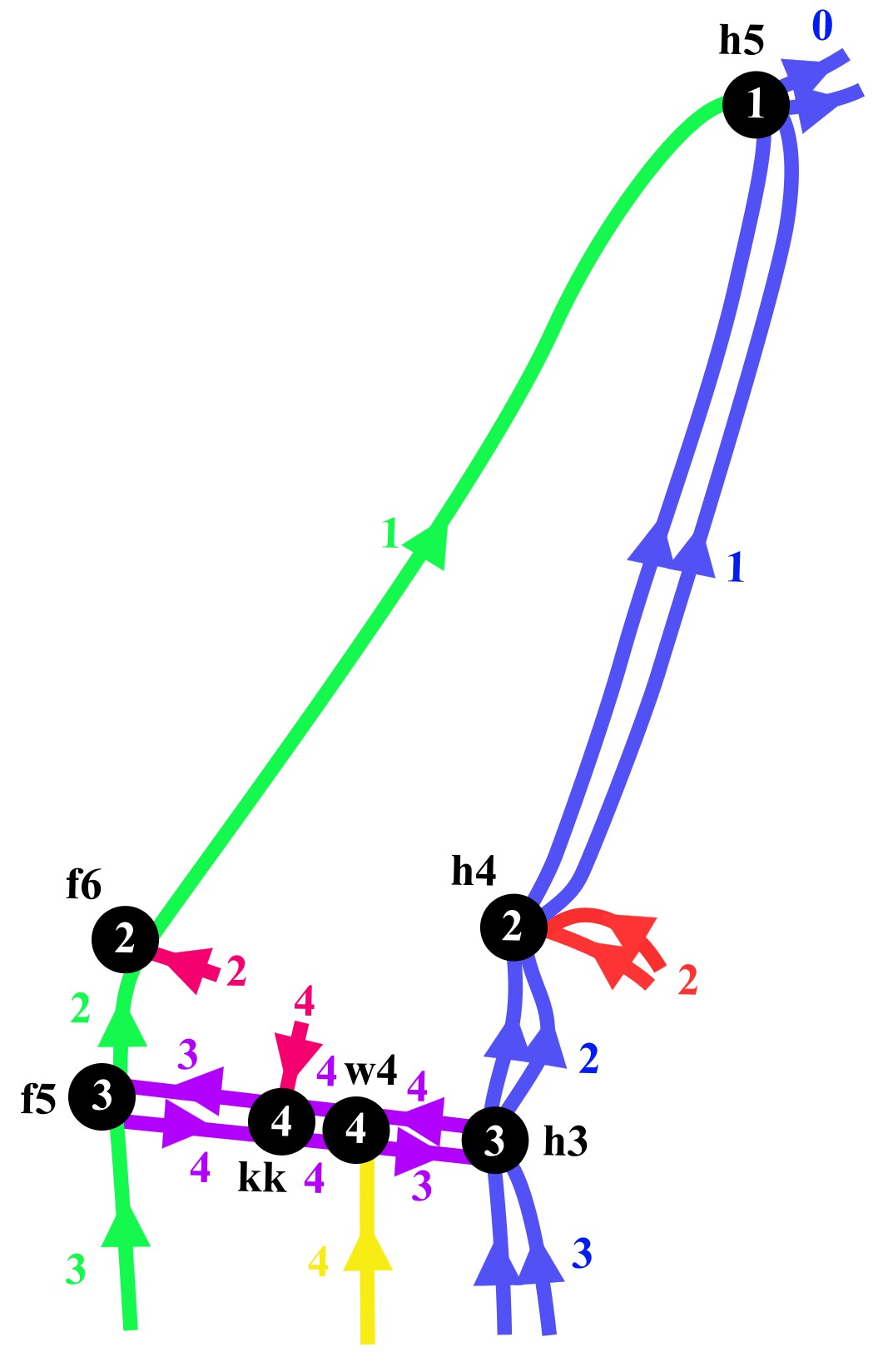}}
    \hskip5ex
    \subfloat[Network snapshot]{\includegraphics[width=0.4\textwidth, height=2.5in]{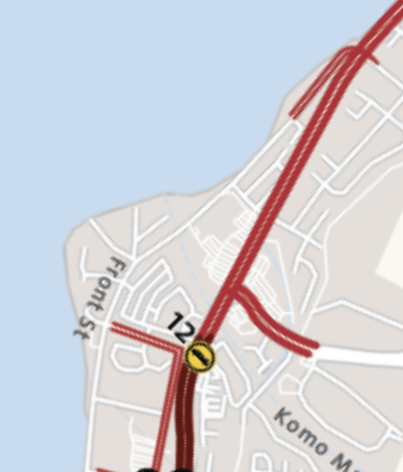}}
    \caption{An example illustrating how distances along each road (indicated by a number close to each road) and junction (indicated by a number inside each circle representing a junction) are being assigned in a subgraph of the original network containing the exiting road (Hwy30[7]) and junction (h5). In (A), the illustrated network diagram is presented, representing the segment of the real-world network in (B) (reproduced from \citep{lahainafirereport}).}
    \label{fig:assign_weights}
\end{figure}

\subsection{Optimization Algorithm}

With the choice of weights as described in $w_i$, we perform the optimization as follows. We run a stochastic block coordinate descent algorithm (SBCD) by picking a sampled index set  $\mathcal{S}$ (without replacement) for a total of $|\mathcal{S}| = s$ components and a fixed number of iterations $N_{\text{iter}}$. For each iteration, we compute the numerical gradient of the loss function in those components while using a backtracking line search with the Armijo-Goldstein condition. In particular, we have that the condition is satisfied if
\begin{equation}\label{AG-cond}
\mathcal{L} \ml \boldsymbol{\alpha} +  \tau \frac{\mathbf{p}(\boldsymbol{\alpha})}{\|\mathbf{p}(\boldsymbol{\alpha})\|} \mr \leq \mathcal{L}(\boldsymbol{\alpha}) - \tau \cdot c \cdot \|\mathbf{p}(\boldsymbol{\alpha})\|
\end{equation}
for some step size $\tau$ and some fixed control factor $c \in (0,1)$, with the descent direction given by
\begin{equation}\label{descent}
p_i(\boldsymbol{\alpha}) = \begin{cases}
-\frac{\p \mathcal{L}}{\p \alpha_i} &\text{ for } i \in \mathcal{S}, \\
0 &\text{ for } i \notin \mathcal{S}. \\
\end{cases}
\end{equation}
Numerically, the descent direction can be approximated by
\begin{equation}\label{descent-approx}
p_i(\boldsymbol{\alpha}) \approx \begin{cases}
-\frac{ \mathcal{L}(\boldsymbol {\alpha} + \epsilon \alpha_i) - \mathcal{L}(\boldsymbol{\alpha})}{\epsilon} &\text{ for } i \in \mathcal{S}, \\
0 &\text{ for } i \notin \mathcal{S}, \\
\end{cases}
\end{equation}
where $\epsilon$ is small.

We use SBCD rather than full gradient descent or derivative-free methods such as the discrete gradient method \citep{Num-Nonsmooth-DGM} for several reasons. First, the loss $\mathcal{L}$ is non-smooth, and coordinate descent methods are known to perform well on non-smooth objectives \citep{Wright2015-CD, Shi2016-PrimerCD, Nesterov2012}. Second, each evaluation of $\mathcal{L}$ requires a full network simulation, so computing only $s$ partial derivatives per iteration rather than the full gradient substantially reduces cost. Third, numerically only a small fraction of partial derivatives are significant at any given $\boldsymbol{\alpha}$, making full gradient computation wasteful. Derivative-free methods such as \citep{Num-Nonsmooth-DGM} also require multiple full-gradient calls and converge slowly in the presence of many local extrema.

To handle the constraints in \eqref{Opt-Network2}, we set a boundary tolerance $\varepsilon_{\text{tol}} > 0$ and define the boundary excess
\begin{equation}\label{eqn:boundary-excess}
\varepsilon_{\text{excess}}(\boldsymbol{\alpha}) := \max\{ 0,\max_i \alpha_i - (1-\varepsilon_{\text{tol}}), \varepsilon_{\text{tol}} - \min_i \alpha_i\}
\end{equation}
which measures how far $\boldsymbol{\alpha}$ lies outside $\mathcal{C}$. The mollified loss function $\tilde{\mathcal{L}}$ is then
\begin{equation}\label{eqn:mollified-loss}
\tilde{\mathcal{L}}(\boldsymbol{\alpha}) =
\begin{cases}
\mathcal{L}(\boldsymbol{\alpha}) &\text{if } \varepsilon_{\text{excess}}(\boldsymbol{\alpha})  = 0, \\
100\min\{\varepsilon_{\text{excess}}(\boldsymbol{\alpha}),1 \}T &\text{otherwise}
\end{cases}
\end{equation}
We then solve the unconstrained surrogate problem
\begin{equation}\label{eqn:unconstrained-opt}
\min_{\boldsymbol{\alpha} \in \mathbb{R}^P} \quad \tilde{\mathcal{L}}(\boldsymbol{\alpha})
\end{equation}
as an approximation to \eqref{Opt-Network2} via the SBCD method described above. The full algorithm is given in Algorithm \ref{alg:num-opt}.

\begin{algorithm2e}[ht]
\caption{Numerical Optimization for Parameters}
\label{alg:num-opt}
\SetKwInOut{Input}{Input}
\SetKwInOut{Output}{Output}
\SetKwComment{Comment}{}{}
\DontPrintSemicolon
\BlankLine
\Input{Initial parameters $\boldsymbol{\alpha}_{\text{init}}$, number of iterations $n_{\text{iter}}$, number of sampled indices $s$, control factor $c$, decay factor $f$, initial step size $\tau_{\text{init}}$, maximum number of decay $N_{\text{decay}}$.}
\Output{Output parameters $\boldsymbol{\alpha}$.}
\BlankLine
$\boldsymbol{\alpha} \gets \boldsymbol{\alpha}_{\text{init}}$; \\
\For{$i \leftarrow 1, \cdots, n_{\text{iter}}$}{
    $n_{\text{decay}} \leftarrow 0$, step size $\tau \leftarrow \tau_{\text{init}}$, number of decays effected $n_{\text{decay}} \gets 0$ \;
    Uniformly sample $s$ components (with the indices as $\mathcal{S}$) out of a total of $P$ components\;
    Compute $\mathbf{p}$ numerically using finite differences as described in \eqref{descent-approx} for components in $\mathcal{S}$\;
    \texttt{check} $\gets \mathcal{L}(\boldsymbol{\alpha} + \tau \frac{\mathbf{p}}{\|\mathbf{p}\|}) - \ml \mathcal{L}(\boldsymbol{\alpha}) - c\tau\|\mathbf{p}\|\mr$ \\
    \While{\texttt{check} $> 0$ and $n_{\text{decay}} < N_{\text{decay}}$}{
        $\tau \gets f\tau$, $n_{\text{decay}} \gets n_{\text{decay}} + 1$; \\
    }
    \If{$n_{\text{decay}} < N_{\text{decay}}$;}{
        $\boldsymbol{\alpha} \gets \boldsymbol{\alpha} + \tau \frac{\mathbf{p}}{\|\mathbf{p}\|}$
    }
}
\end{algorithm2e}

\section{Numerical Results}\label{section:results}
In this section, we apply the optimization algorithm to a synthetic road network and the Lahaina wildfire evacuation scenario. The simulations can be found at \hyperlink{https://github.com/alexxue99/Traffic-Simulation}{https://github.com/alexxue99/Traffic-Simulation}.

We report three metrics over a horizon $T$: the weighted time-integrated cars, the cars entered, and the cars exited. The weighted time-integrated cars metric is
\begin{equation}\label{time-integrated}
\text{Weighted Time-Integrated Cars}(T) =  \sum_{i} W_i(T)
\end{equation}
with the contribution on each road as
\begin{equation}\label{time-integrated2}
W_i(T) := 2^{-d_i} \int_0^T \int_{0}^{L_i} \rho_i(x,t) \D x \D t
\end{equation}
as per equation \eqref{Opt-Network1} and weights in \eqref{weights}. The cars entered metric is
\begin{equation}\label{cars-entered}
\text{Cars Entered}(T) = \rho_{\text{jam}}\sum_{\text{Source roads } i}\int_0^T f_i(0^+,t)\D t
\end{equation}
and the cars exited metric is
\begin{equation}\label{cars-exited}
\text{Cars Exited}(T) = \rho_{\text{jam}}\sum_{\text{Exit roads } i}\int_0^T f_i(L_i^-,t)\D t
\end{equation}
where $\rho_\text{jam}$ is the jam density used in \eqref{jam-density}. The implicit dependence on $\alpha$, the drivers' preferences, is assumed in each of these expressions.

All heatmaps use the color scheme shown in Figure~\ref{fig:los_color}. Each point along a road segment is assigned a level of service (LOS) and colored according to the local traffic density. The procedure used to compute LOS is described in Section~\ref{subsection:initial_densities}.
\begin{figure}[htbp]
    \centering
    \includegraphics[scale=0.8]{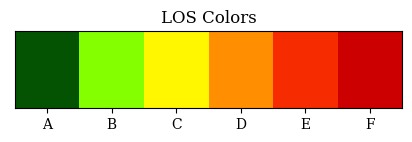}
    \caption{Color legend for LOS classification.}
    \label{fig:los_color}
\end{figure}

\subsection{Example 1: Toy Network}
\begin{figure}[htbp]
    \centering
    \includegraphics[scale=0.15]{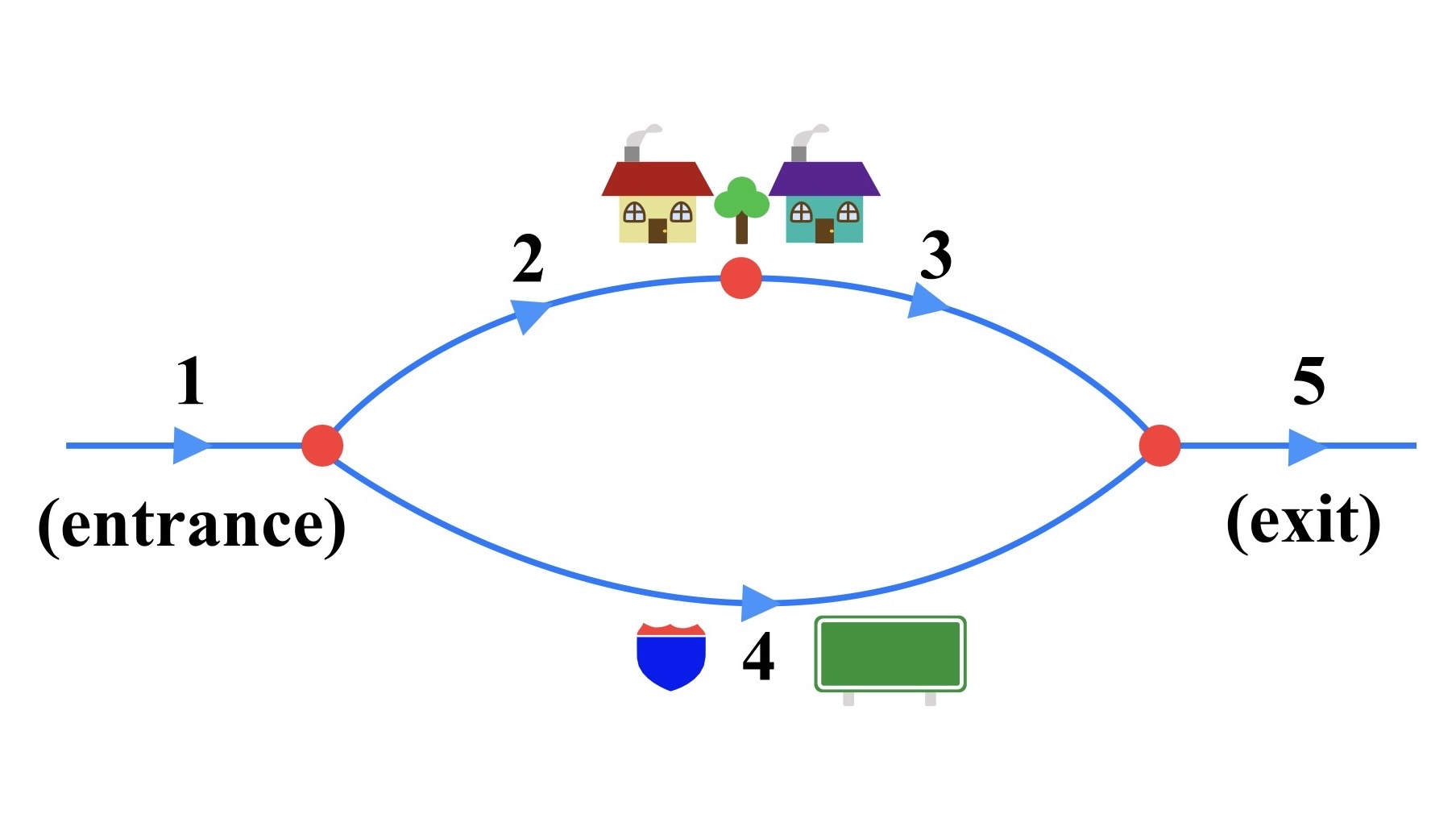}
    \caption{A simple toy model. Road 1 is the entrance to the network, and Road 5 is the exit. Roads 2 and 3 represent a path through a residential area with multiple stops, while Road 4 represents a continuous path via a highway.}
    \label{fig:toy}
\end{figure}

We demonstrate the framework on the toy network shown in Figure~\ref{fig:toy}. We let the free-flow speed along road $i$ be denoted by $v_i$, the maximum flux along road $i$ be $f_{c,i}$ (this is after accounting for scaling by the number of lanes on that road), $\rho_i \in [0,n_i]$ for $n_i$ lanes on road $i$ with capacity density $\sigma_i$, and $A_i$ the coefficient in front of the quadratic term for the flux function in the congested regime for road $i$. In other words, we have
\begin{equation}\label{example_eq1}
f_i(\rho) = \begin{cases}
v_i \rho, & \rho \in [0,\sigma_i), \\
A_i(\rho - \sigma_i)^2 + f_{c,i}, & \rho \in (\sigma_i,n_i].
\end{cases}
\end{equation}
The parameter $\alpha\in(0,1)$ denotes the drivers' preference at the left junction towards road 2.

\newcommand{\tzero}{\frac{L}{v_4}}
\newcommand{\tone}{\frac{L}{v_2}}
\newcommand{\tauFive}{\frac{L_5}{v_1}}

Table~\ref{table:toyparameters} summarizes the road parameters.

\begin{table}[h]
\centering
\begin{tabular} {|l|c|c|c|c|}
\hline
Road & Length & Lanes ($n_i$) & Free-flow speed ($v_i$) & Flux capacity ($f_{c,i}$) \\
\hline
 Entry (road 1) & $L_1$ & $1$ & $v_1$ & $f_{c,1}$ \\
 Road 2 & $L_2$ & $1$ & $v_2$ & $f_{c,2}$ \\
 Road 3 & $L_3$ & $1$ & $v_2$ & $f_{c,2}$ \\
 Road 4 & $L_4$ & $1$ & $v_4$ & $f_{c,4}$ \\
 Exit (road 5) & $L_5$ & $n_5$ & $v_1$ & $n_5 f_{c,1}$ \\
 \hline
\end{tabular}
\caption{Symbolic parameters for the toy network. Roads 2 and 3 share parameters. Each lane on the exit road has the same characteristics as road 1, so $f_{c,5}=n_5 f_{c,1}$, $\sigma_5=n_5\sigma_1$, and $A_5=A_1/n_5$.}
\label{table:toyparameters}
\end{table}

For the simulations, we use the following numerical values. The entry road (road 1) and exit road (road 5) each have length $0.5$ mi, free-flow speed $v_1=25$ mi/hr, and flux capacity $f_{c,1}=500$ veh/hr/lane. Roads 2 and 3 (the residential route) each have length $0.5$ mi, free-flow speed $v_2=15$ mi/hr, and flux capacity $f_{c,2}=400$ veh/hr/lane. Road 4 (the highway route) has length $1$ mi, free-flow speed $v_4=20$ mi/hr, and flux capacity $f_{c,4}=500$ veh/hr/lane. The jam density is $\rho_{\text{jam}}=100$ veh/mi/lane throughout.

The network satisfies the following structural conditions:
\begin{enumerate}[(i)]
\item \textit{Residential area.} $v_2 < v_4$, reflecting that roads 2 and 3 pass through a residential area with lower speed limits compared to the highway route via road 4.
\item \textit{Equal path lengths.} $L := (L_2 + L_3) = L_4$, so the two routes from the left junction to the right junction have equal total length.
\item \textit{Exit road scaling.} Each lane on the exit road has the same characteristics as road 1, i.e., $f_{c,5} = n_5 f_{c,1}$, $v_5 = v_1$, $\sigma_5 = n_5 \sigma_1$, and $A_5 = A_1/n_5$.
\end{enumerate}

We study the congested regime in which all interior roads are initialized at a density above capacity ($\rho_{\text{init}} > \sigma_i$ for $i \in \{1,2,3,4\}$), with the exit road initially empty ($\rho_5(0,x) \equiv 0$). This scenario is the most relevant for evacuation planning. Additional experiments with uncongested and single-source initial data, which admit closed-form solutions independent of $n_5$, are treated in Appendix~\ref{appendix:proofs}.

Under these initial conditions, the toy network exhibits the following phase transition.

\begin{proposition}\label{closedform-CD}
\textbf{(Phase Transition in Exit Lane Capacity.)} Consider the toy network in Figure~\ref{fig:toy} with parameters as in Table~\ref{table:toyparameters} satisfying conditions~(i)--(iii), with congested initial data $\rho_{\text{init}} > \sigma_i$ for $i \in \{1,2,3,4\}$ and $\rho_5(0,x) \equiv 0$. The critical exit lane threshold is
\begin{equation}\label{n5-star}
n_5^\star := \frac{f_{c,2}+f_{c,4}}{f_{c,1}}.
\end{equation}
This result holds as long as the initial queues on the incoming roads have not depleted, i.e., during the bottleneck-dominated regime most relevant for evacuation planning.
\end{proposition}

\begin{proof}
Since roads 2 and 4 are congested at the right junction, the incoming capacities are $c_2 = f_{c,2}$ and $c_4 = f_{c,4}$. Since road 5 is initially empty, $c_5 = f_{c,5} = n_5 f_{c,1}$. The capacity ratio at the right junction is
$r = (f_{c,2}+f_{c,4})/(n_5 f_{c,1}),$
so $r \leq 1$ if and only if $n_5 \geq n_5^\star$. By Theorem~\ref{opt-sol}, exit throughput is $\gamma_5 = \min\{f_{c,2}+f_{c,4},\, n_5 f_{c,1}\}$, which scales linearly in $n_5$ for $n_5 < n_5^\star$ and is constant for $n_5 \geq n_5^\star$. Detailed closed-form expressions for all experiments on this network are given in Appendix~\ref{appendix:proofs}.
\end{proof}

The threshold $n_5^\star$ marks a phase transition corresponding to a shift from case~(iii) to case~(i) of Theorem~\ref{opt-sol} as exit lanes increase. For $n_5 < n_5^\star$, the exit road is the bottleneck: exit throughput equals $n_5 f_{c,1}$ and scales linearly with $n_5$ (exit-capacity-limited regime), and drivers must abandon their original preferences to maximize junction throughput. For $n_5 \geq n_5^\star$, exit throughput equals $f_{c,2}+f_{c,4}$ and is independent of $n_5$ (supply-limited regime), meaning the exit can absorb all incoming flux while respecting preferences.

Substituting the numerical values from above with $\rho_{\text{init}} = 0.9$, the phase transition occurs at
\[
n_5^*=\frac{f_{c,2}+f_{c,4}}{f_{c,1}}=\frac{400+500}{500}=1.8.
\]
Figure~\ref{fig:cars_exited_per_lane} confirms this behavior, showing that the number of cars exited scales linearly in $n_5$ up to $n_5^*=1.8$, after which it plateaus. Table~\ref{table:toy-exp-D} reports the simulation metrics for integer lane counts, showing that the transition from $n_5=1$ to $n_5=2$ yields a substantial improvement, while $n_5=3$ provides no additional benefit. Figure~\ref{fig:toy_flooded} illustrates this contrast: for $n_5=2>n_5^*$, backward rarefaction fans on roads 2 and 4 are visually pronounced as the junction boundary densities drop to $\sigma_i \ll \rho_{\text{init}}$, whereas for $n_5=1<n_5^*$ the network remains heavily congested.

\begin{table}[h]
\centering
\begin{tabular}{c c c c}
\hline
$n_5$ & Weighted Time-Int Cars & Cars Exited & Optimal $\alpha$ \\
\hline
1 & 851.56 & 126.37 & 0.50 \\
2 & 762.25 & 227.48 & 0.50 \\
3 & 762.25 & 227.48 & 0.50 \\
\hline
\end{tabular}
\caption{Congested toy network ($\rho_{\text{init}} = 0.9$ on all interior roads) with $T = 1000$ seconds and $nt_{\text{opt}} = 1$ second. The transition from 1 to 2 lanes improves cars exited by 80\%, while 3 lanes provides no further benefit, consistent with $n_5^*=1.8$.}
\label{table:toy-exp-D}
\end{table}

\begin{figure}[htbp]
\centering
\includegraphics[width=0.8\textwidth]{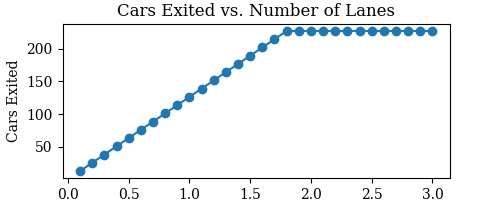}
\caption{Cars exited as a function of the number of exit road lanes $n_5$ (treated as a continuous variable), for the congested toy network with $\rho_{\text{init}} = 0.9$ and $T = 1000$ seconds. The phase transition at $n_5^*=1.8$ is clearly visible.}
\label{fig:cars_exited_per_lane}
\end{figure}

\begin{figure}[htbp]
\centering

\begin{minipage}[t]{0.48\textwidth}
    \centering
    \includegraphics[width=\linewidth]{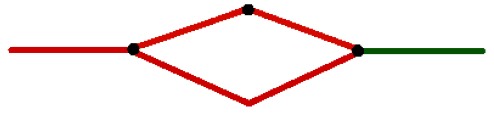}
\end{minipage}\hfill
\begin{minipage}[t]{0.48\textwidth}
    \centering
    \includegraphics[width=\linewidth]{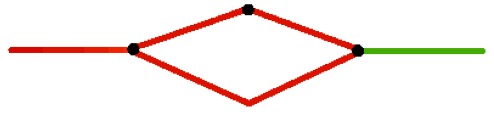}
\end{minipage}

\vspace{0.5em}

\begin{minipage}[t]{0.48\textwidth}
    \centering
    \includegraphics[width=\linewidth]{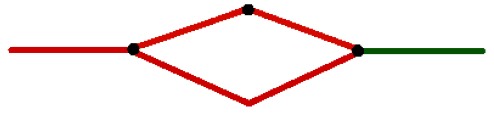}
\end{minipage}\hfill
\begin{minipage}[t]{0.48\textwidth}
    \centering
    \includegraphics[width=\linewidth]{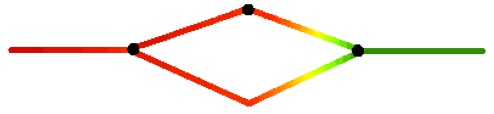}
\end{minipage}
\caption{Snapshots of the congested toy network (all interior roads initialized at $\rho_{\text{init}}=0.9$) at $t = 0$ (left) and $t = 1000$ seconds (right), for $n_5 = 1$ (top) and $n_5 = 2$ (bottom). See Appendix~\ref{appendix:proofs} for a detailed rarefaction analysis.}
\label{fig:toy_flooded}
\end{figure}

\subsection{Example 2: Lahaina Wildfire Simulations}
In this subsection, we will be applying the optimization algorithm to the Lahaina fire that occurred on August 8-9, 2023. Figure~\ref{fig:lahaina} shows a map of Lahaina from~\citep{lahainafirereport}, including the major roads and the congestion points during the time frame 11:00-15:00 on August 8th.
\begin{figure}[htbp]
    \centering
    \includegraphics[scale=0.2]{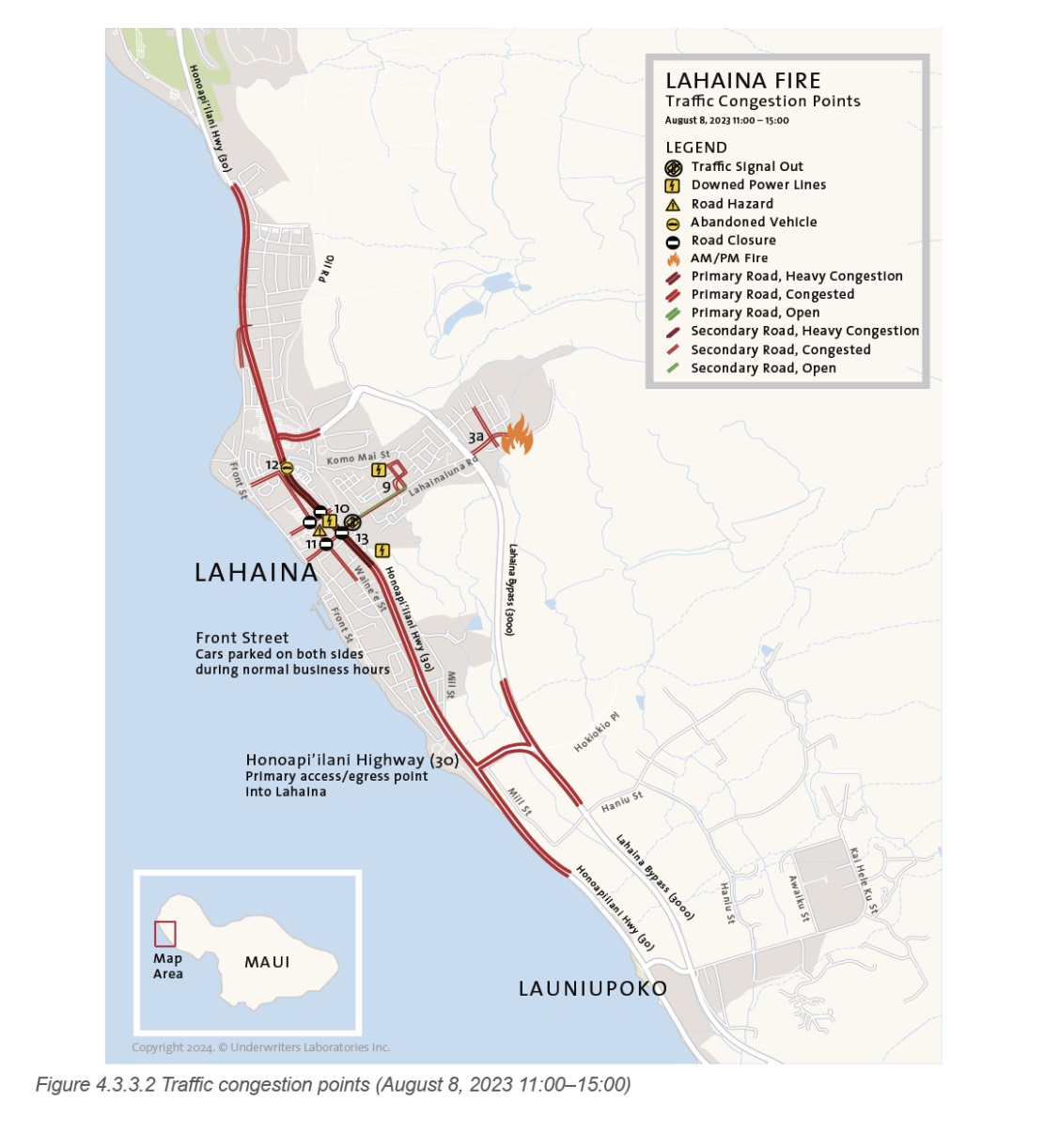}
    \caption{Traffic congestion points during the Lahaina fire. Reproduced from \citep{lahainafirereport}.}
    \label{fig:lahaina}
\end{figure}

Below, we run a couple of simulations for this time range. We also run some simulations for a later time range of approximately 15:00-17:10. To differentiate between these times, we label the first set of simulations as ``AM" and the second set as ``PM". The goal of these simulations is to optimize how traffic could have been directed in order to maximize the overall efficiency of the network. We measure the overall efficiency of the network as the time integrated number of cars along the main roads.

Since there are several optimization and model parameters, we run simulations in four phases. In Phase 1, we run an initial sweep of the optimization parameters by examining different combinations on the AM Base Network. Specifically, we look at the effect of $\gamma$, the initial densities from source roads, and $nt_{opt}$, the amount of time to optimize preferences over. 

In Phase 2, we first use the results from Phase 1 to select a smaller subset of parameters to examine in depth. This is done by performing robust simulations of initial networks corresponding to the first part of the evacuation, before the fire overtook Lahaina Bypass. These initial networks are AM Base, AM 2, AM3, and PM Base. 

In Phase 3, we simulate the full sequence of networks for a fixed $nt_{opt} = 1$, chosen based on the results in the previous phases. Since the second period in the evacuation (after the bypass was overtaken by the fire) resulted in residents of Eastern Lahaina having less access to evacuation routes, we now assume that the initial density of source roads in these areas may be greater than or equal to the roads in Western Lahaina. We relabel the initial density in Eastern Lahaina as $\gamma_1$, and denote the varying initial density of the Western area as $\gamma_2$. Consequently, we study the effect on the later networks (PM 2 through PM 5) when this alternative initial density $\gamma_2\in [\gamma_1, 1]$ is varied. 

Finally, in Phase 4, we examine the effect of reversing one or both southbound lanes on the northern exit to increase exit capacity. This is a practical intervention that requires no additional infrastructure, and we assess its impact on network efficiency to provide recommendations for improving future evacuation outcomes.

\subsubsection{Network Descriptions}Figure~\ref{fig:am_base_with_legend} shows the base network we use for the start of the AM simulations. We included the major roads, as well as important source roads for where roads can enter the network.

\begin{figure}[htbp]
\centering
\begin{minipage}[c]{0.45\textwidth}
    \centering
    \includegraphics[width=\linewidth]{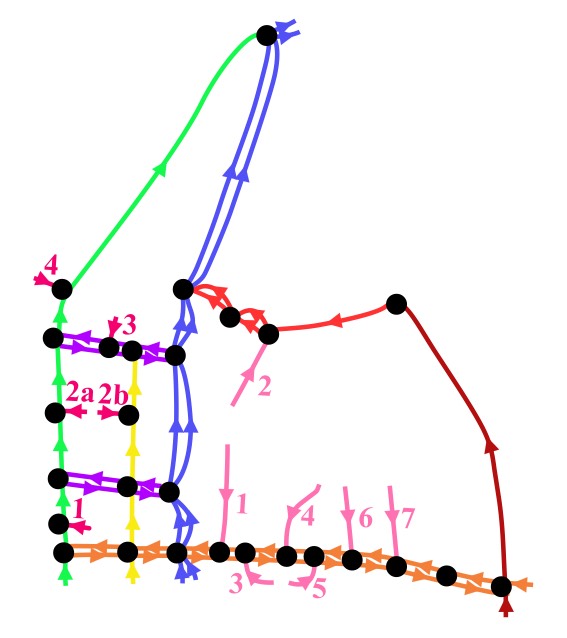}
\end{minipage}\hfill
\begin{minipage}[c]{0.5\textwidth}
    \centering
    \small
    \fbox{
    \begin{tabular}{@{}ll@{}}
    \textcolor{honopiilani}{$\blacksquare$} & Honopiilani Highway / HI-30 \\
    \textcolor{front}{$\blacksquare$} & Front Street\\
    \textcolor{wainee}{$\blacksquare$} & Waine'e Street\\
    \textcolor{minor}{$\blacksquare$} & Minor residential roads\\
    & (Papalaua Street, Kenui Street)\\
    \textcolor{keawe}{$\blacksquare$} & Keawe Street\\
    \textcolor{bypass}{$\blacksquare$} & Lahaina Bypass\\
    \textcolor{ll}{$\blacksquare$} & Lahainaluna Road\\
    \textcolor{east_source}{$\blacksquare$} & Eastern source roads\\
    & (1. Kuhua, 2. Komo Mai, 3. Pauoa,\\ & 4. Kale,
    5. Paunau, \\
    &6. Kelawea, 7. Kalena)\\
    \textcolor{west_source}{$\blacksquare$} & Western/default source roads\\
    & (1. Wahie, 2. Baker,\\
    & 3. Kahoma Village Loop, 4. Puunoa)
    \end{tabular}}
\end{minipage}
\caption{AM Base Network and Road Color Legend. Simulated from 11{:}00--13{:}25.}
\label{fig:am_base_with_legend}
\end{figure}

As major incidents occur, we modify the simulated network to reflect any changes. The information about all of the complications can be found in~\citep{lahainafirereport}. Detailed diagrams of all intermediary networks can be found in Appendix \ref{appendix:num_impl}. The first incident occurred at approximately 13:25, where multiple power lines were downed between Papalaua Street and Lahainaluna Road. We represent this in AM Network 2 by removing Papalaua Street from the network.

The next incident occurred at approximately 14:21, where metal roofs were blown onto the roads on Lahainaluna Road. Specifically, the eastbound part from Front Street through Wainee Street became unusable, which we represent in AM Network 3 by removing the parts from Lahainaluna Road between Front Street and Wainee Street and between Wainee Street and Hwy-30.
\begin{figure}[htbp]
\centering
\begin{minipage}[c]{0.45\textwidth}
    \centering
    \includegraphics[width=\linewidth]{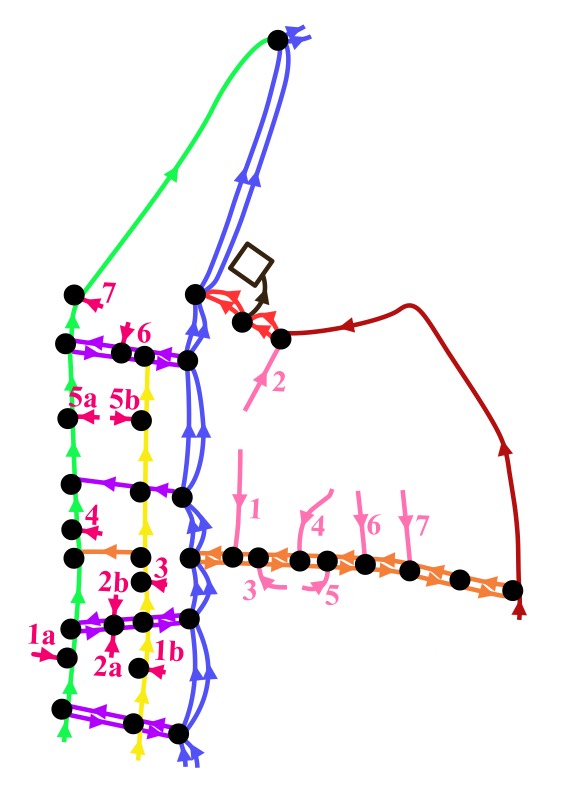}
\end{minipage}\hfill
\begin{minipage}[c]{0.5\textwidth}
    \centering
    \small
    \fbox{
    \begin{tabular}{@{}ll@{}}
    \textcolor{honopiilani}{$\blacksquare$} & Honopiilani Highway / HI-30 \\
    \textcolor{front}{$\blacksquare$} & Front Street\\
    \textcolor{wainee}{$\blacksquare$} & Waine'e Street\\
    \textcolor{minor}{$\blacksquare$} & Minor residential roads\\
    & (Prison, Dickenson, Papalaua, Kenui)\\
    \textcolor{keawe}{$\blacksquare$} & Keawe Street\\
    \textcolor{bypass}{$\blacksquare$} & Lahaina Bypass\\
    \textcolor{ll}{$\blacksquare$} & Lahainaluna Road\\
    \textcolor{dirt}{$\blacksquare$} & Oil Road\\
    $\blacksquare$ & Gateway Shopping Center\\
    \textcolor{east_source}{$\blacksquare$} & Eastern source roads\\
    & (1. Kuhua, 2. Komo Mai, 3. Pauoa,\\
    &4. Kale,
    5.  Paunau, 6. Kelawea, \\
    &7. Kalena, 8. dirt road )\\
    \textcolor{west_source}{$\blacksquare$} & Western/default sources\\
    & (1a. Canal, 1b. Hale, 2. Luakini,\\
    & 3. Panaewa, 4. Wahie, 5. Baker,\\
    & 6. Kahoma Loop, 7. Puunoa)
    \end{tabular}}
\end{minipage}
\caption{PM Base Network and Road Color Legend. Note that Road 8, a nondescript dirt road, is not in PM Base, but is found in other PM Networks.}
\label{fig:pm_base_with_legend}
\end{figure}
For the PM networks, Figure~\ref{fig:pm_base_with_legend} shows the base PM network. It looks similar to AM Network 3 (Figure~\ref{fig:am3}), but we removed the right-most portion of Lahainaluna Road to reflect the fire being done with the east side of Lahaina, and we added some additional source roads for the increasing number of cars that came into the network as more people evacuated.

At approximately 15:25, the fire took over the bypass. As a consequence, traffic in the Kelawea community funneled towards Komo Mai Street and Lahainaluna Road, causing traffic to slow heavily. At the time, officers therefore decided to help alleviate traffic by diverting it from Komo Mai Street to the Gateway shopping center.  These changes are implemented in PM Network 2.

Next, at approximately 16:29, the oil road at the intersection of Komo Mai Street and Keawe Street became passable. We added this road into the network with PM Network 3.

At approximately 16:35, the fire took over parts of Dickenson Street. We represent this change in PM Network 4 by removing the right part of Dickenson Street.

Lastly, at approximately 16:47, the southbound roads for Front Street and Hwy-30 were opened up, allowing for two additional exits at the bottom of the network along these roads. This change is reflected in PM Network 5.

\FloatBarrier
\subsubsection{Phase 1: Evaluating Optimization Parameters for AM Base}\label{sec:phase-1}
In order to determine what hyperparameters of $\gamma \in (0, 1]$ and $nt_{opt}$ to choose 
for the final simulations, we conduct preliminary studies on the AM Base network, simulating 
it for its full time frame of 8,700 seconds (2 hours and 25 minutes). Since all source roads in the AM Base network have normalized capacity densities 
of either $\sigma = 0.075$ or $\sigma = 0.125$, we consider $\gamma \in \{0.01, 
0.0375, 0.075, 0.125, 1.00\}$, representing the extreme cases of no congestion 
and fully congested source roads, as well as densities below, near, and above 
capacity. 
For $nt_{opt}$, we consider $\{0, 1, 10, 60, 600\}$ seconds, representing 
no optimization ($nt_{opt} = 0$), short-term planning at the second timescale 
($nt_{opt} = 1, 10$), and longer-term planning at the minute timescale 
($nt_{opt} = 60, 600$). While ideally controllers would optimize over the entire 
simulation, accurate long-range prediction is infeasible in practice, as a traffic network is highly complex, and it is impossible to accurately model behavior very far in advance using only the current and previous data, motivating 
this restricted set. All other relevant parameters used for numerical implementation are detailed in Appendix \ref{appendix:num_impl}.

The weighted time-integrated cars across all combinations are summarized in 
Table~\ref{tab:phase1_summary}. For large $\gamma$ values ($0.075, 0.125, 1.00$), the network 
completely floods by the end of the simulation, and optimization provides only marginal 
improvement. For smaller $\gamma$ values ($0.01, 0.0375$), the optimization has a more 
pronounced effect, with increasing $nt_{opt}$ generally improving network efficiency as 
measured by cars exited. Snapshots illustrating the flooding behavior for $\gamma = 0.075$ 
are shown in Figure~\ref{fig:gamma_0.075_snapshots}; a detailed discussion  and per-$\gamma$ tables and 
snapshot comparisons across $nt_{opt}$ values are provided in Appendix~\ref{appendix:num_impl}.

\begin{table}[htbp]
\centering
\begin{tabular}{lccccc}
\toprule
$nt_{opt}$ & $\gamma = 0.01$ & $\gamma = 0.0375$ & $\gamma = 0.075$ & $\gamma = 0.125$ & $\gamma = 1.00$ \\
\midrule
0   & 747.73   & 6,365.52 & 7,754.90 & 7,857.96 & 7,874.50 \\
1   & 728.50   & 6,485.01 & 7,755.87 & 7,858.34 & 7,874.90 \\
10  & 729.01   & 6,495.34 & 7,755.91 & 7,858.39 & 7,874.91 \\
60  & 735.51   & 6,591.61 & 7,756.04 & 7,858.42 & 7,874.99 \\
600 & 701.94   & 6,828.82 & 7,776.05 & 7,858.61 & 7,875.26 \\
\bottomrule
\end{tabular}
\caption{Weighted time-integrated cars by $nt_{opt}$ and $\gamma$ for Phase 1 (AM Base Network).
For $\gamma = 0.01$, the loss function is not monotone in $nt_{opt}$ due to the network being 
nearly empty; see Appendix~\ref{appendix:num_impl} for discussion.}
\label{tab:phase1_summary}
\end{table}

\begin{figure}[htbp]
     \subfloat[]{\includegraphics[width=0.3\textwidth]{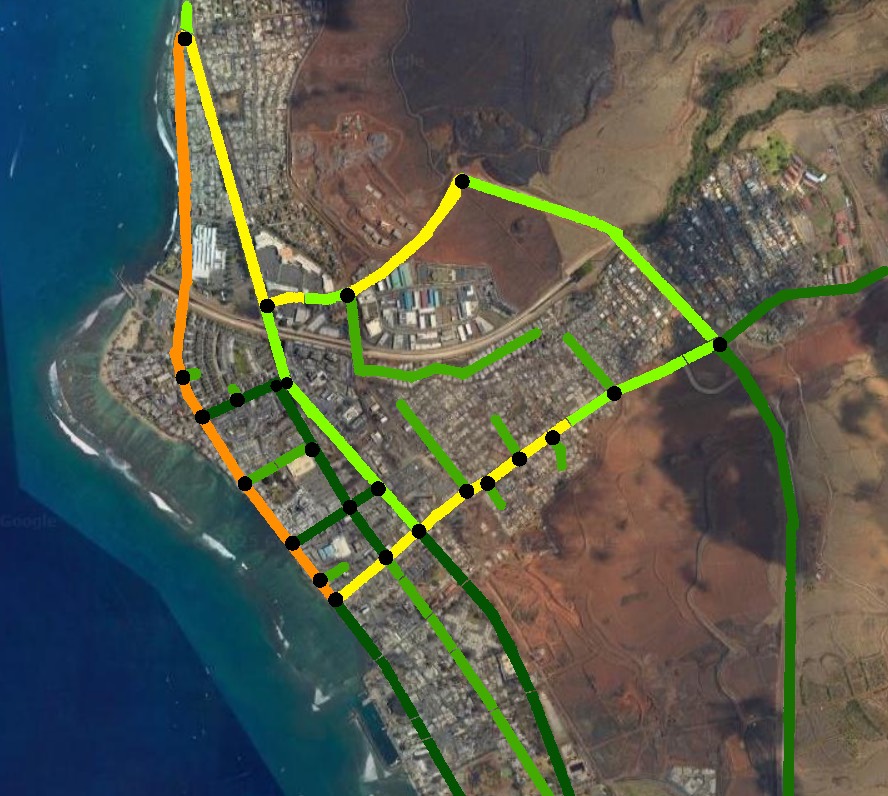}\label{fig:step0}}\hskip1ex
    \subfloat[]{\includegraphics[width=0.3\textwidth]{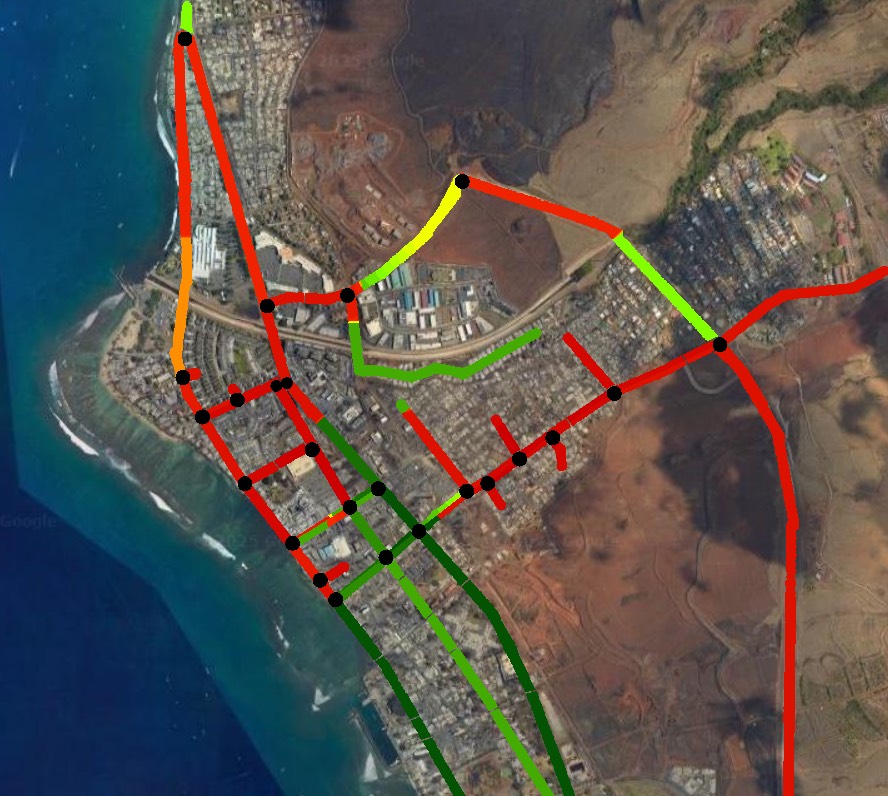}\label{fig:step30k}}\hskip1ex
    \subfloat[]{\includegraphics[width=0.3\textwidth]{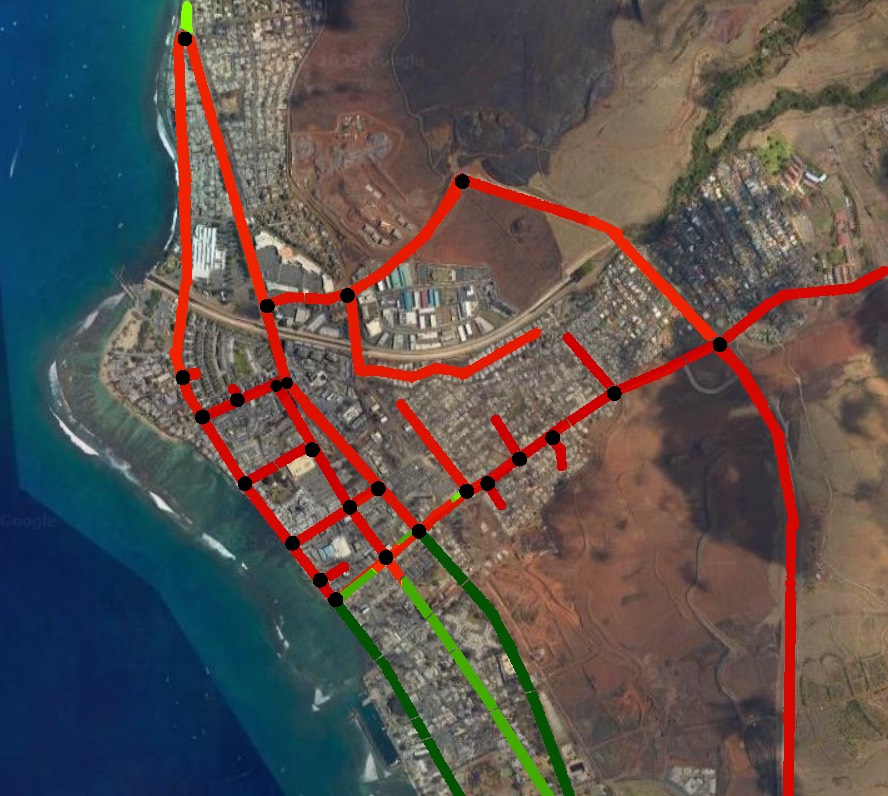}\label{fig:step60k}}\\
     \makebox[0.3\textwidth]{$t=0$}\hskip1ex
    \makebox[0.3\textwidth]{$t=900$}\hskip1ex
    \makebox[0.3\textwidth]{$t=1800$}
    \caption{Snapshots of the AM Base Network at $t=0, 900, 1800$ seconds for $\gamma = 0.075$ 
    with $nt_{opt}=0$. The network quickly floods, with all roads fully congested by $t=2550$.}
    \label{fig:gamma_0.075_snapshots}
\end{figure}

Based on these results, we draw two conclusions for the subsequent phases. First, since the 
impact of optimization is most pronounced for smaller $\gamma$, we restrict to 
$\gamma \in \{0.01, 0.0375\}$ for Phases 2 and 3. Second, while $nt_{opt} = 600$ technically 
achieves the best loss for most $\gamma$ values, the improvement over $nt_{opt} = 1$ is 
marginal, and the difference in computation time is substantial. We therefore fix $nt_{opt} 
\in \{0, 1, 60\}$ for Phase 2, and $nt_{opt} = 1$ for Phase 3.

\subsubsection{Phase 2: Simulating Initial Networks}\label{sec:phase-2} We first examine the effect of the parameters $\gamma, nt_{opt}$ on the network simulated for the initial evacuation period, where it can be assumed that those who evacuated in advance are equally distributed across all sources. This corresponds to the AM Base, AM 2, AM 3, and PM Base networks, simulated consecutively from 11:00  - 15:25. In order to ensure the simulation results are robust against the randomness introduced during optimization, we run the optimization for each network 10 times, and average the results. We note that averaging the final results of the simulations produces identical network behavior to first averaging the optimal parameters found from the 10 runs of the optimization, then running the simulation using the averaged parameters. Thus, the two methods are used interchangeably throughout Phase 2 and 3. 

Since the previous section showed that the impacts on the AM Base network were largest for smaller $\gamma$, we restrict our simulations to $\gamma = 0.01, 0.0375$. Furthermore, since the changes in traffic behavior were noticeable when comparing $nt_{opt} = 0, 1, 60$, we also simulate only these select values.

The cumulative optimization metrics across all the initial networks for $\gamma=0.01$ is given in Table \ref{tab:cum_ntopt_comparison_gamma0.01_0.0375_combined}. Similarly to when AM Base alone was simulated, the cumulative loss is highest for $nt_{opt}=0$, but the cumulative cars exited increases with increased $\gamma$, showing the improved efficiency of the optimized networks. This can also be seen in Figure \ref{fig:0.01_cum_diff}, which shows the cumulative difference between the cars entered for the $nt_{opt}$ values and the baseline of $nt_{opt}=0$ across all networks. It is also worth noting that the improvement in cars exiting with $nt_{opt}=1$ is most pronounced in the AM Base simulation, but the cumulative impact is only slightly lower than with $nt_{opt}=60$. 

\begin{table}[htbp]
\centering
\resizebox{\textwidth}{!}{
\begin{tabular}{lcccccc}
\toprule
& \multicolumn{3}{c}{$\gamma = 0.01$} & \multicolumn{3}{c}{$\gamma = 0.0375$}  \\
\cmidrule(lr){2-4} \cmidrule(lr){5-7}
$nt_{opt}$ & Weighted Time & Cars & Cars & Weighted Time & Cars & Cars \\
           & -Int Cars & Entered & Exited & -Int Cars & Entered & Exited \\
\midrule
0  & 1,221.03 & 3,530.33 & 3,781.12 & 12,391.88 & 9,730.70 & 8,834.23 \\
1  & 1,197.97 & 3,530.33 & 3,792.66 & 12,577.69 & 9,742.63 & 8,834.23 \\
60 & 1,197.97 & 3,530.33 & 3,795.40 & 12,627.03 & 9,741.70 & 8,834.23 \\
\bottomrule
\end{tabular}}
\caption{Cumulative Optimization Metrics ($\gamma = 0.01, 0.0375$).}
\label{tab:cum_ntopt_comparison_gamma0.01_0.0375_combined}
\end{table}

\begin{figure}[htbp]
    \subfloat[$\gamma = 0.01$]{\includegraphics[width=0.49\textwidth]{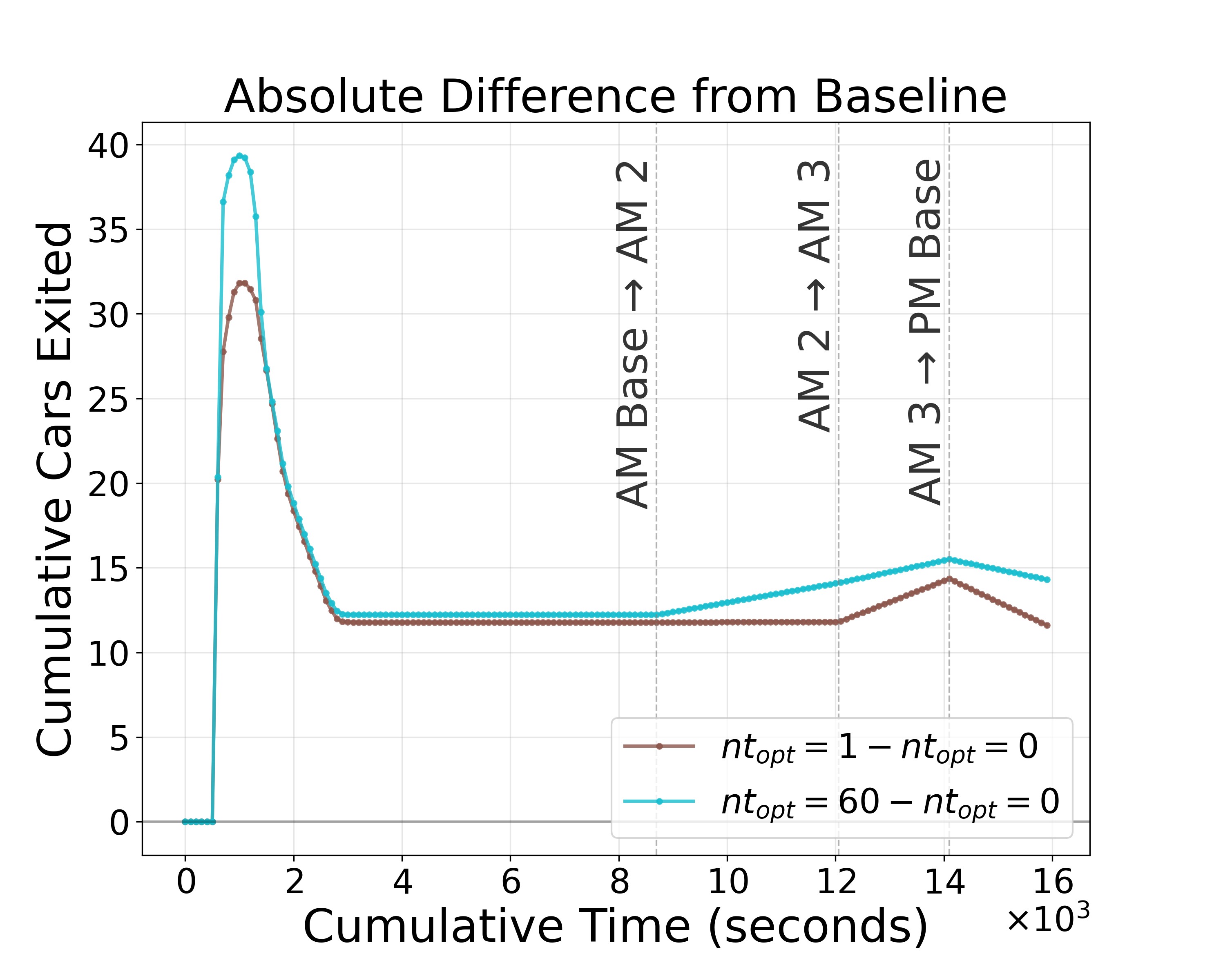}\label{fig:0.01_cum_diff}}\hskip1ex
    \subfloat[$\gamma = 0.0375$]{\includegraphics[width=0.48\textwidth]{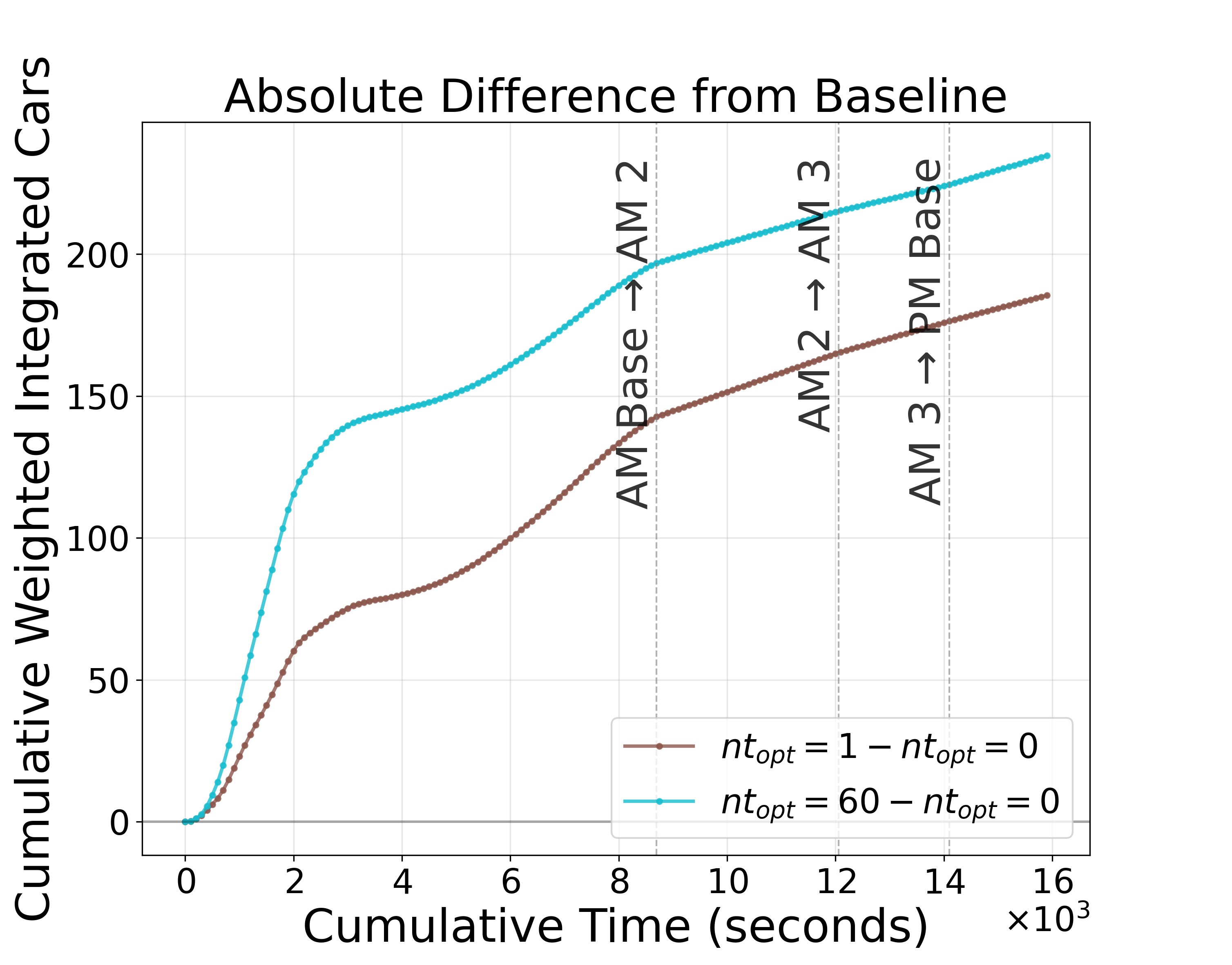}\label{fig:0.0375_cum_diff}}
    \caption{Cumulative absolute difference between $nt_{opt} = 0$ and $nt_{opt} = 1, 60$. (A) Difference in cars exited ($\gamma = 0.01$). (B) Difference in weighted integrated cars ($\gamma = 0.0375$). Time is given in thousands of seconds.}
    \label{fig:cum_diff_combined}
\end{figure}





For $\gamma = 0.0375$, the cumulative results are provided in Table \ref{tab:cum_ntopt_comparison_gamma0.01_0.0375_combined}. Again, despite the additional averaging steps and the simulation across all networks rather than just AM Base, the cumulative results are similar to those in Phase 1. Namely, the weighted time-integrated cars increases with increasing $nt_{opt}$, where the cumulative differences are provided in Figure \ref{fig:0.0375_cum_diff}. Also, the number of cars entered is larger for nonzero $nt_{opt}$, but is not directly increasing, which was also seen originally in Phase 1. Finally, the improvement in the weighted loss function between $nt_{opt} = 0$ and $nt_{opt}=1$ is fairly substantial, while the improvement between $nt_{opt} = 1$ and $nt_{opt} = 60$ is not as large, even though there is a significant difference in computation time between optimizing over one second vs one minute.



\subsubsection{Phase 3: Effect of Congestion Parameters on Afternoon Networks}\label{sec:phase-3} At approximately 15:25, the fire spread eastwards, becoming very close to the residential areas. Thus, we assume that the source roads near the fire become significantly congested as many people simultaneously evacuate these areas. We denote the level of congestion for these sources as $\gamma_2$, and relabel the other source roads as having a congestion level of $\gamma_1$. We examine $\gamma_2 \in [\gamma_1, 1]$ for both $\gamma_1 = 0.01, 0.0375$, since the density of residents attempting to evacuate from the east should be greater than or equal to the density in western Lahaina. 

For simplicity, we consider only $nt_{opt} = 1$ for all simulations, which corresponds to drivers making decisions based on what will optimize the network near-instantaneously (over the next second). This kind of decision-making aligns with a wildfire evacuation scenario, as drivers are unable to plan ahead very much compared to day-to-day traffic behavior, or other evacuation scenarios where evacuations can begin several days before (e.g. hurricane evacuations). This choice in $nt_{opt}$ is also supported by the results in Phase 1 and 2, where $nt_{opt}=1$ seemed to provide the best balance of improved network efficiency and computation time.

In Figure \ref{fig:gamma1_0.01_cum}, it can be seen that for low $\gamma_1 = 0.01$, the cumulative loss function of time-integrated cars on roads weighted by distance to the exit increases as $\gamma_2$ increases. This is due to the network remaining relatively uncongested for all time frames when $\gamma_2$ is also low. However, when $\gamma_2$ is high, even with the low $\gamma_1$, the final network is very congested, with more of HWY-30 becoming jammed, which can be seen in Figure \ref{fig:gamma1_0.01_finalsnapshots}. Interestingly, for all choices in $\gamma_2$, the optimization funnels more cars towards the southbound exits that were opened at the start of PM 5. This is likely because the congestion is concentrated towards the Eastern residential areas, and this exit is closer for evacuating citizens from this area.
\begin{figure}[htbp]
    \centering
    \includegraphics[width=0.75\textwidth]{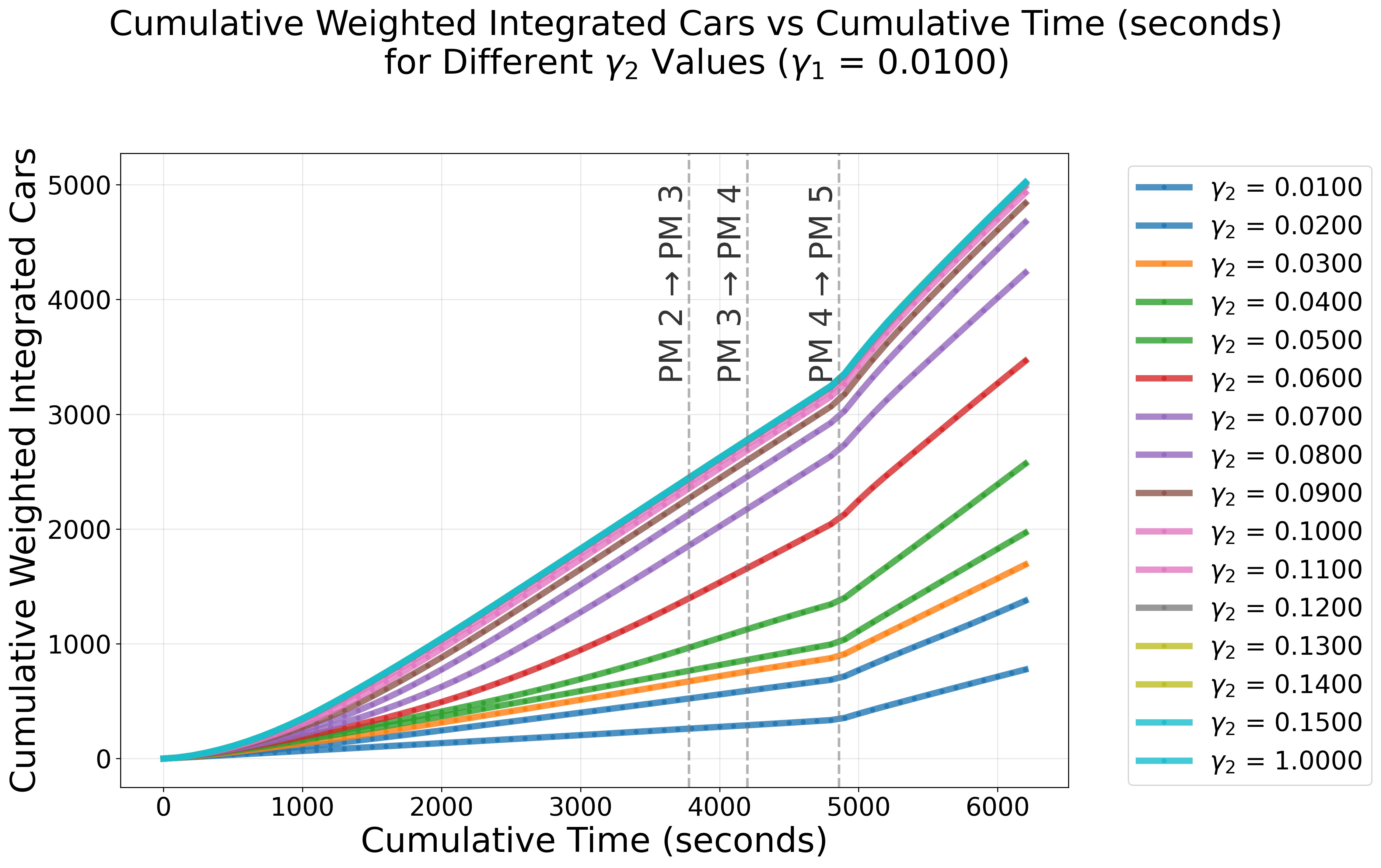}
     \caption{Cumulative time integrated cars for various $\gamma_2$ values with $\gamma_1 = 0.01$.}
    \label{fig:gamma1_0.01_cum}
\end{figure}

\begin{figure}[htbp]
     \centering
     \subfloat[]{\includegraphics[width=0.3\textwidth]{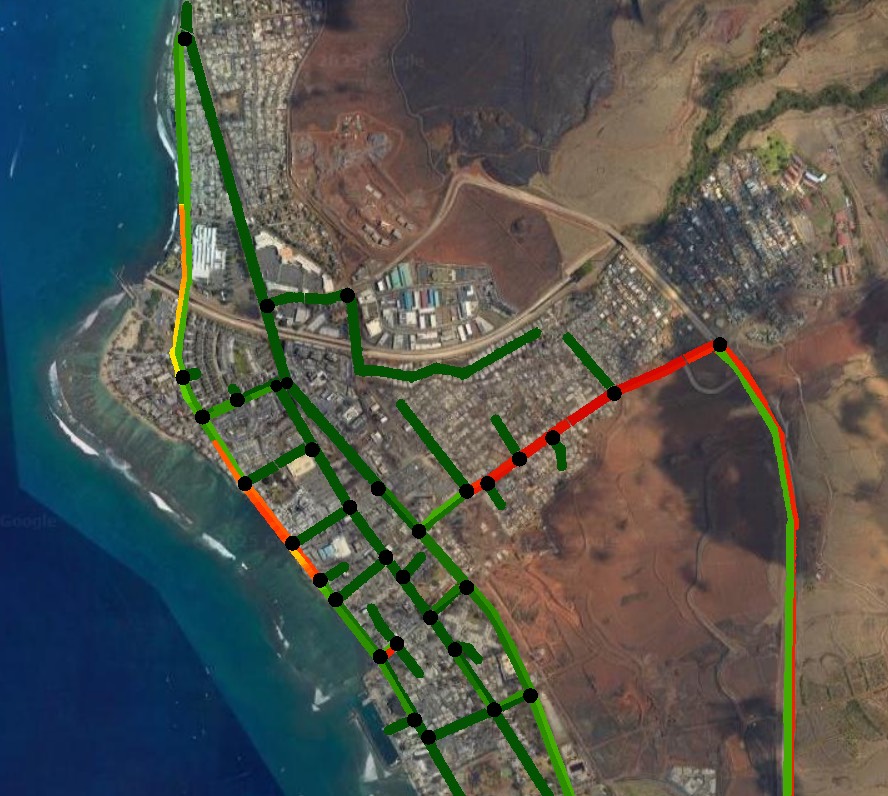}}\hskip1ex
    \subfloat[]{\includegraphics[width=0.3\textwidth]{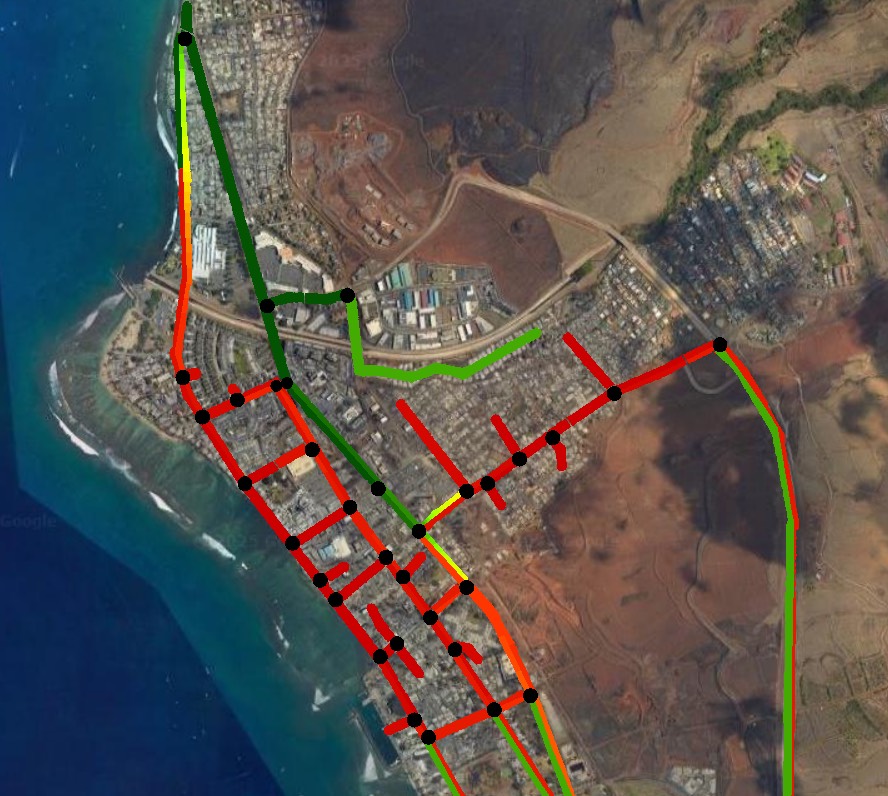}}\hskip1ex
    \subfloat[]{\includegraphics[width=0.3\textwidth]{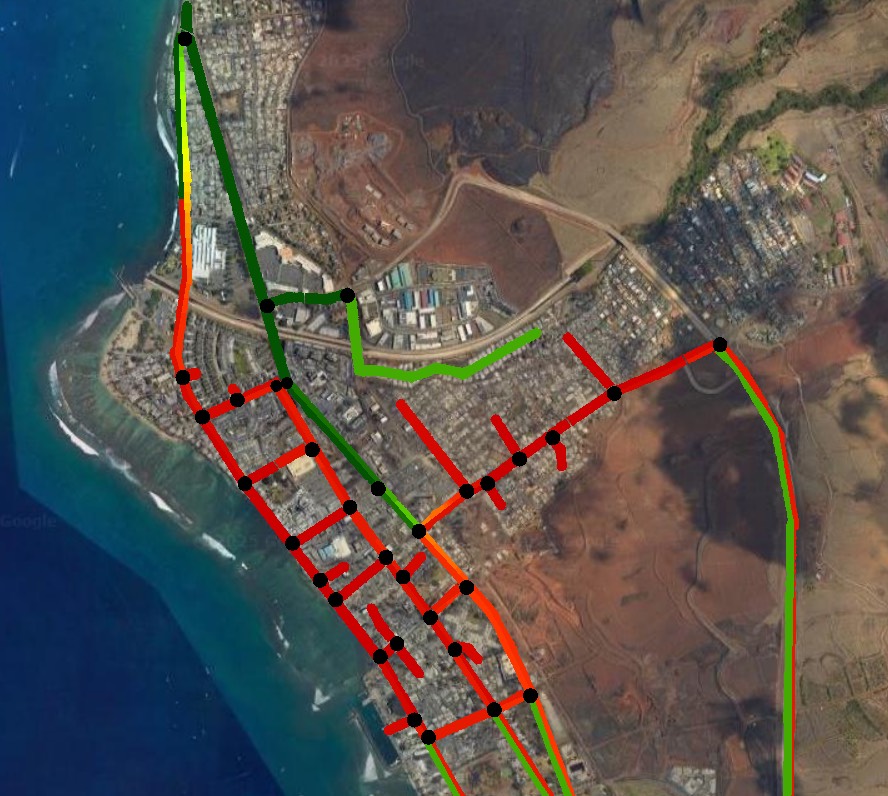}}\\
    \makebox[0.3\textwidth]{$\gamma_2 = 0.01$}\hskip1ex
    \makebox[0.3\textwidth]{$\gamma_2 = 0.10$}\hskip1ex
    \makebox[0.3\textwidth]{$\gamma_2 = 1.00$}
        \caption{The final network state (end of PM 5) for $\gamma_1 = 0.01$, with $\gamma_2 = 0.01, 0.10,$ and $1.00$. }
        \label{fig:gamma1_0.01_finalsnapshots}
\end{figure}

The cumulative loss function for $\gamma_1 = 0.0375$ is shown in Figure \ref{fig:gamma1_0.0375_cum}. Again, it is clear that larger $\gamma_2$ values result in a higher final cumulative loss value, due to the increased demand of the corresponding networks. However, the differences are much lower than those for $\gamma_1 = 0.01$, as the network is congested for all $\gamma_2$ values, which is shown in Figure \ref{fig:gamma1_0.0375_finalsnapshots}. It is noteworthy that even in the case where $\gamma_2 = \gamma_1 = 0.0375$, the optimization still results in congestion being more concentrated towards the southbound HWY-30 exit, similarly to when $\gamma_1 = 0.01$.
\begin{figure}[htbp]
    \centering
    \includegraphics[width=0.75\textwidth]{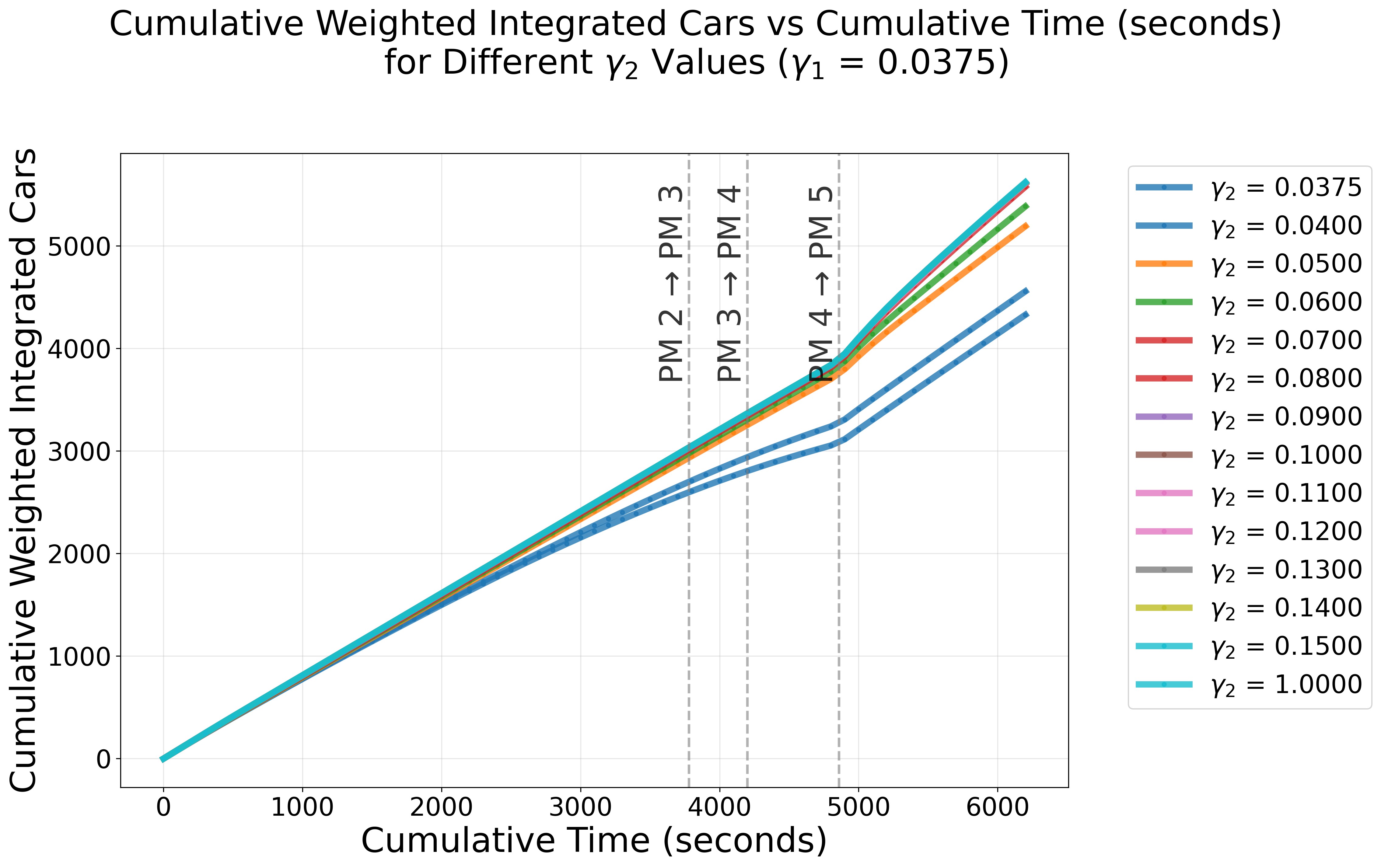}
     \caption{Cumulative time integrated cars for various $\gamma_2$ values with $\gamma_1 = 0.0375$.}
    \label{fig:gamma1_0.0375_cum}
\end{figure}

\begin{figure}[htbp]
     \centering
     \subfloat[]{\includegraphics[width=0.3\textwidth]{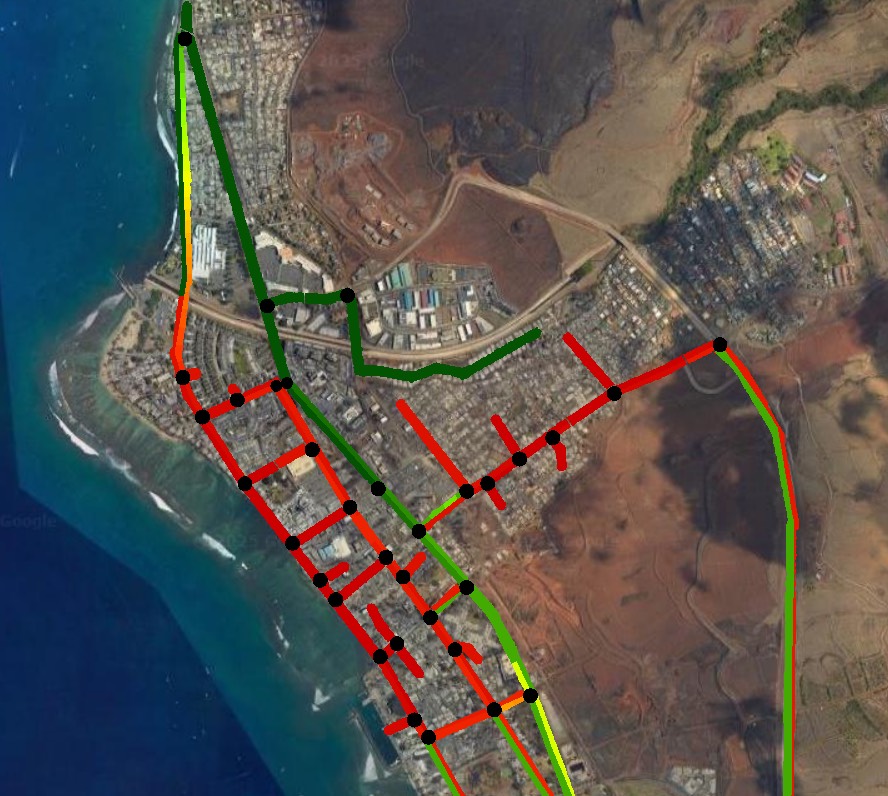}}\hskip1ex
    \subfloat[]{\includegraphics[width=0.3\textwidth]{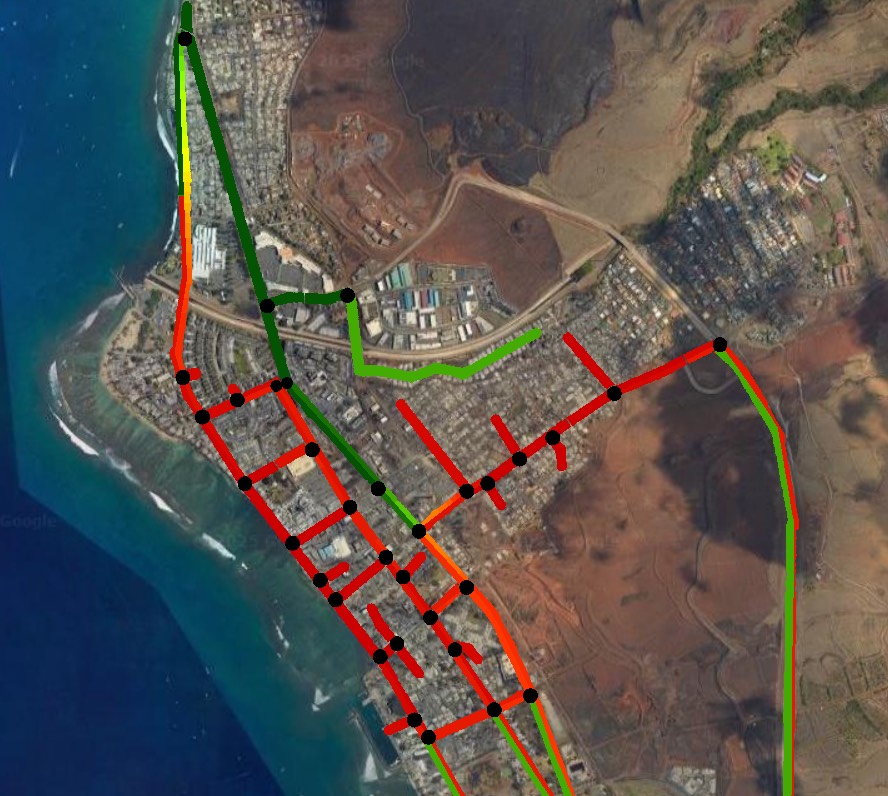}}\hskip1ex
    \subfloat[]{\includegraphics[width=0.3\textwidth]{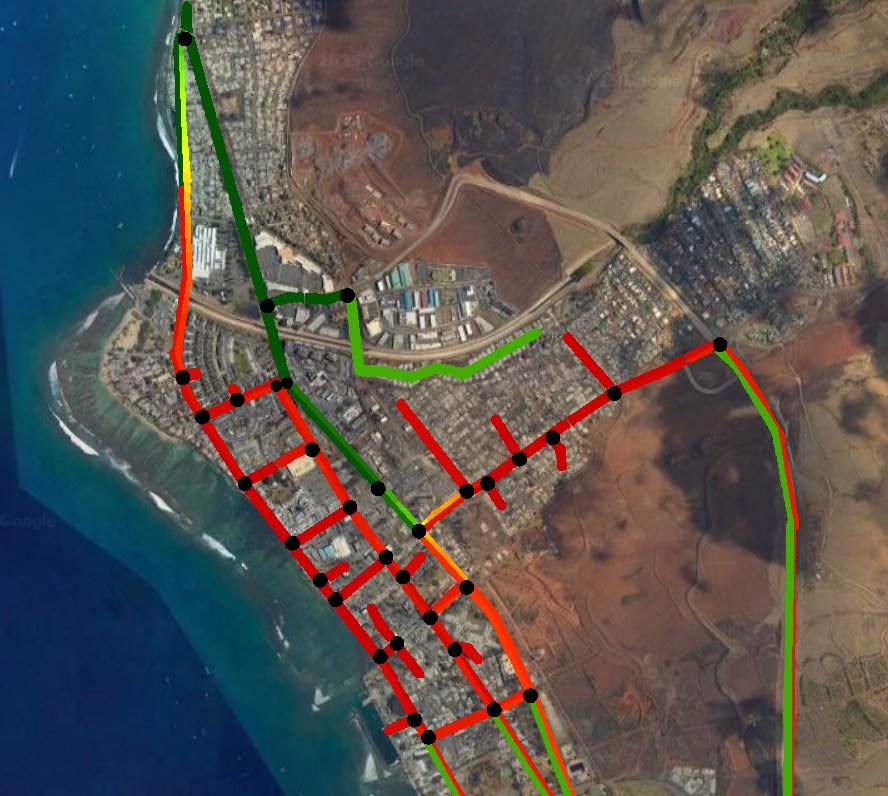}}\\
     \makebox[0.3\textwidth]{$\gamma_2 = 0.0375$}\hskip1ex
    \makebox[0.3\textwidth]{$\gamma_2 = 0.10$}\hskip1ex
    \makebox[0.3\textwidth]{$\gamma_2 = 1.00$}
        \caption{The final network state (end of PM 5) for $\gamma_1 = 0.0375$, with $\gamma_2 = 0.0375, 0.10,$ and $1.00$. }
        \label{fig:gamma1_0.0375_finalsnapshots}
\end{figure}
\subsubsection{Phase 4: Reversing Lanes on Northern Exit}\label{sec:phase-4}
It has been shown through the previous phases that the closure of the Lahaina Bypass, and the opening of the southbound exit on Honoapi'ilani Highway both had major impacts on the congestion in the network. To improve the performance of the network in between these events (corresponding to networks PM 2 through PM 4), one idea is to reverse the southbound lanes on the northern exit (this corresponds to the road segment labeled Hwy30[7] in all simulations/tables). This effectively increases the capacity of the exit road, allowing more cars to evacuate. This is a practical implementation for evacuations, as it would not require the construction of any additional roads. Furthermore, it can be assumed that no pedestrian vehicles will enter the network utilizing the reversed lanes, since cars that have evacuated safely would not re-enter until the fire has been contained. Thus, we examine the impact of the reversal of these lanes on the performance of PM networks 2,3, and 4.

We fix $\gamma_1 = 0.0375$ and $nt_{opt} = 1$, and examine the uncongested and congested regimes given by $\gamma_2 = 0.0375$ and $\gamma_2 = 0.1000$. We consider the reversal of one or both of the southbound lanes, corresponding to a total of three or four exit lanes. 

For the uncongested regime ($\gamma_2 = 0.0375$), the cumulative difference in the number of cars exited between the standard exit with 2 lanes and the exits with increased capacity is shown in Figure \ref{fig:gamma2_0.0375_cars_exit}, while snapshots of the final network at the end of PM 4 for the different lane values are given in \ref{fig:gamma2_0.0375_final_snapshots}. It can be seen that across all networks, the difference is positive for both 3 lanes and 4 lanes, with a final value of approximately 140 more cars escaping for 3 lanes, and 147 more cars escaping for 4 lanes. Both the 3 lanes and 4 lanes configurations show noticeably less congestion near the exit junction at the end of the simulations. Note that the two configurations have very similar dynamics, with only a slight difference in the final number of additional cars that were able to escape.

\begin{figure}[htbp]
    \centering\includegraphics[width=0.75\linewidth]{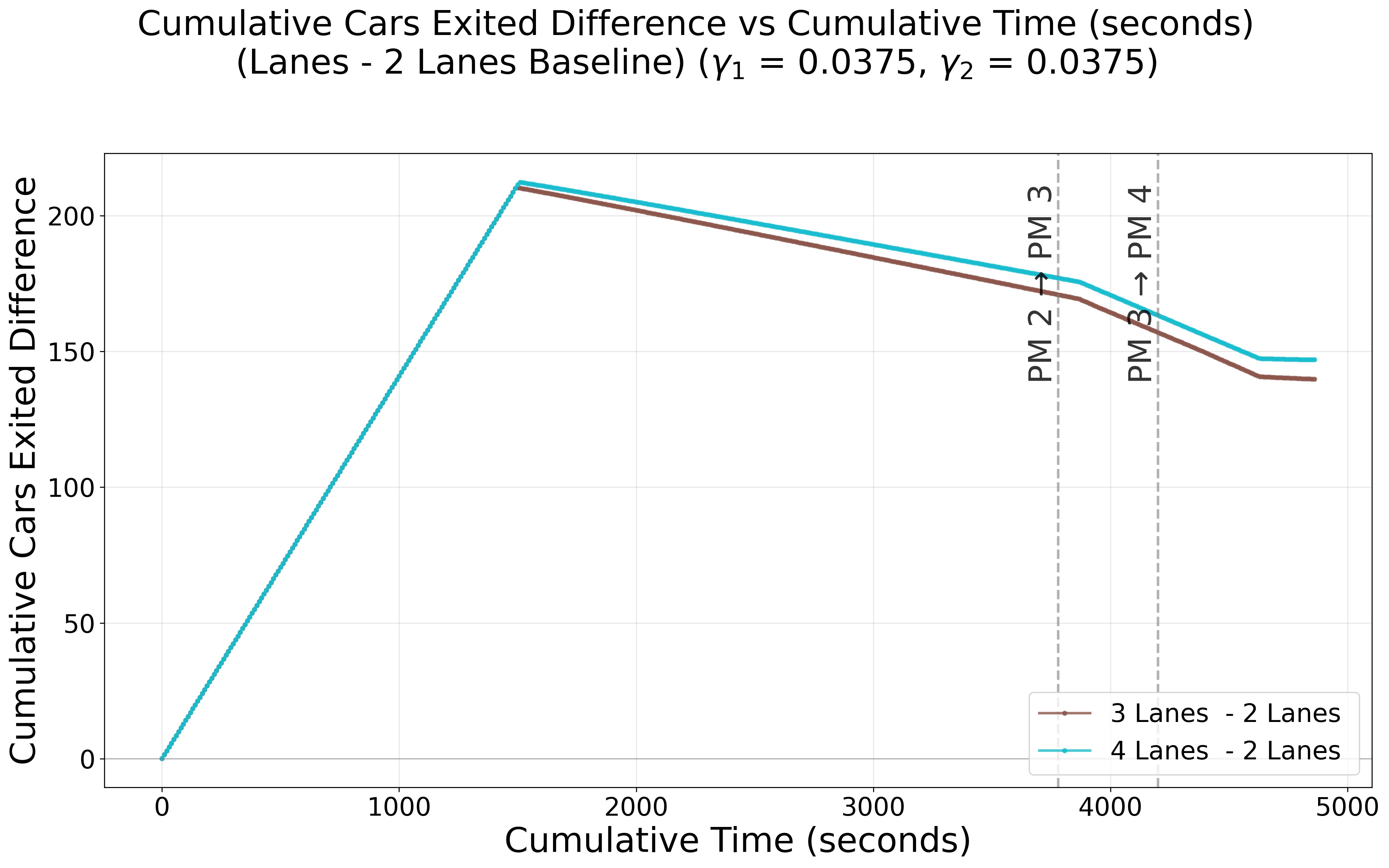}
    \caption{The cumulative difference in cars exited between 3 and 4 exit lanes and 2 lanes across PM networks 2, 3, and 4 for $\gamma_1 = 0.0375$, $\gamma_2 = 0.0375$.}
    \label{fig:gamma2_0.0375_cars_exit}
\end{figure}

\begin{figure}[htbp]
     \centering
     \subfloat[]{\includegraphics[width=0.3\textwidth]{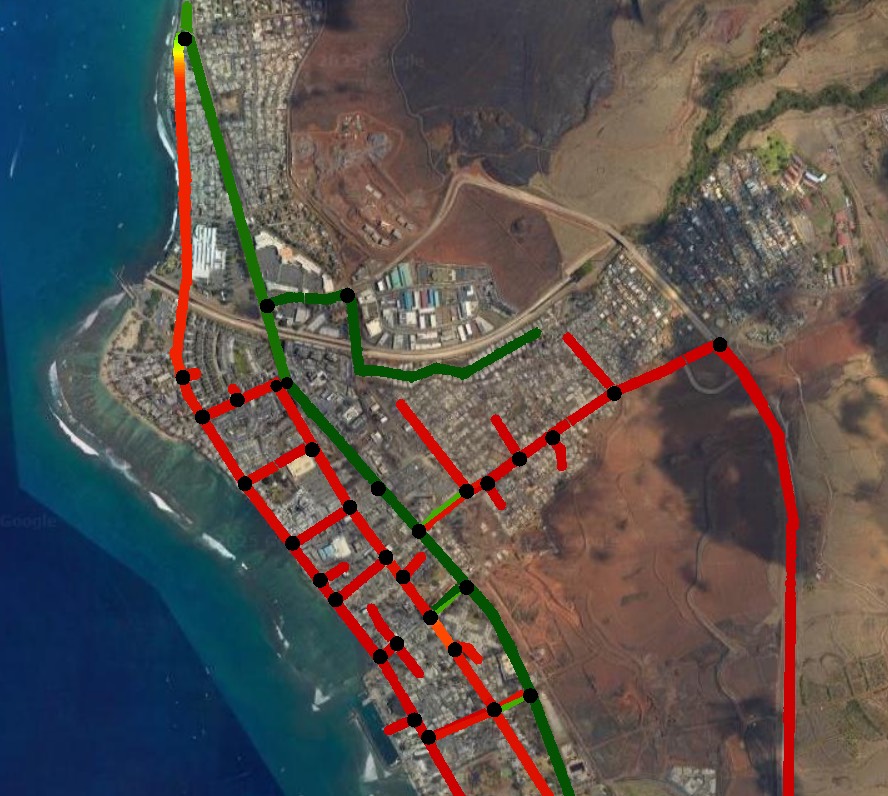}}\hskip1ex
    \subfloat[]{\includegraphics[width=0.3\textwidth]{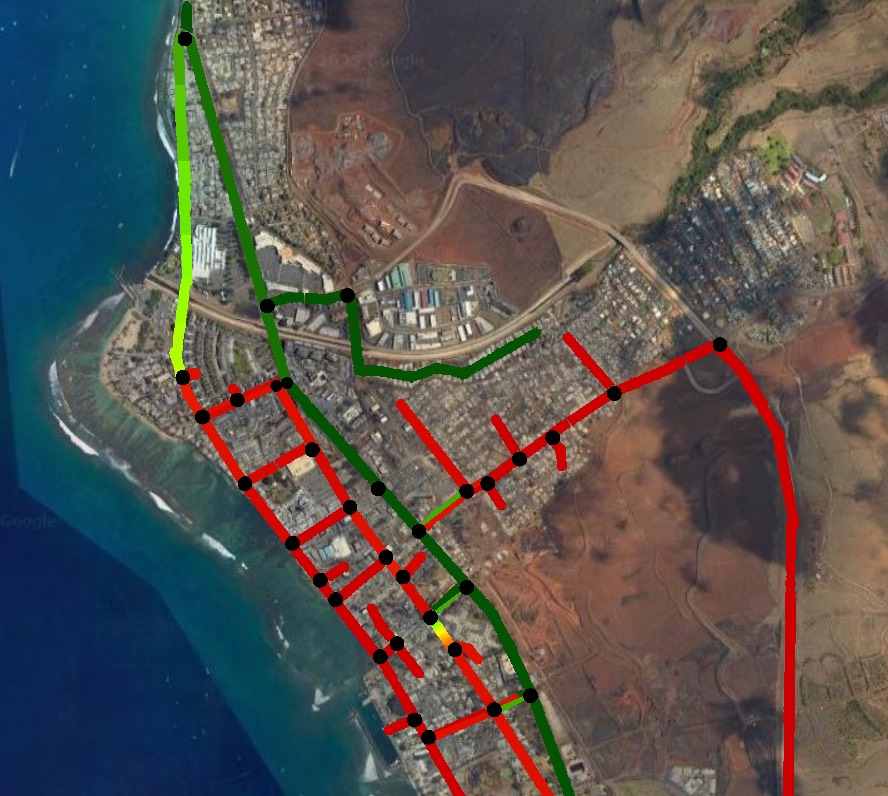}}\hskip1ex
    \subfloat[]{\includegraphics[width=0.3\textwidth]{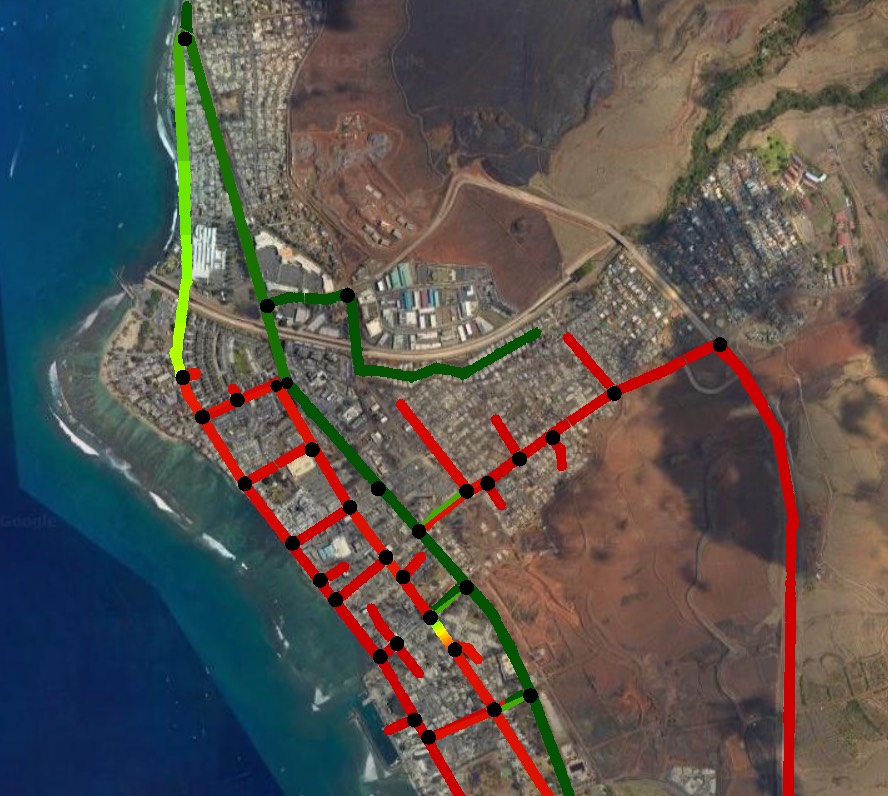}}\\
     \makebox[0.3\textwidth]{2 lanes}\hskip1ex
    \makebox[0.3\textwidth]{3 lanes}\hskip1ex
    \makebox[0.3\textwidth]{4 lanes}
        \caption{Final snapshots of PM 4 for 2,3, and 4 exit lanes with $\gamma_1 = 0.0375$, $\gamma_2 = 0.0375$. }
        \label{fig:gamma2_0.0375_final_snapshots}
\end{figure}

There are two notable phase transition points in Figure \ref{fig:gamma2_0.0375_cars_exit}, one at approximately 1,500 seconds during PM 2, and another at 4,400 seconds during PM 4. These transition points correspond to when Hwy-30 fully clears, resulting in a drop in the incoming flux at the exit junction, which causes the exit road to be very uncongested, and thus a drop in the rate of cars exiting. This occurs during PM 2 at the first transition point for the 3-lane and 4-lane configurations, which can be seen in the snapshots provided in Figure \ref{fig:gamma2_0.0375_3lane_snapshots}. Meanwhile, due to the decreased efficiency, this clearing does not occur for the 2-lane setting until the second transition point approximately 400 seconds into PM 4, which can be seen in Figure \ref{fig:gamma2_0.0375_2lane_snapshots}. 

Thus, in between the transition points, the configurations with additional exit lanes have a lower rate of cars exiting because they are less congested, resulting in the decrease in the difference of cumulative cars seen in the plot. Then, after the second transition point, since all versions have similar low rates of cars exiting, the difference plateaus until the end of the simulation.

\begin{figure}[htbp]
     \centering
     \subfloat[]{\includegraphics[width=0.3\textwidth]{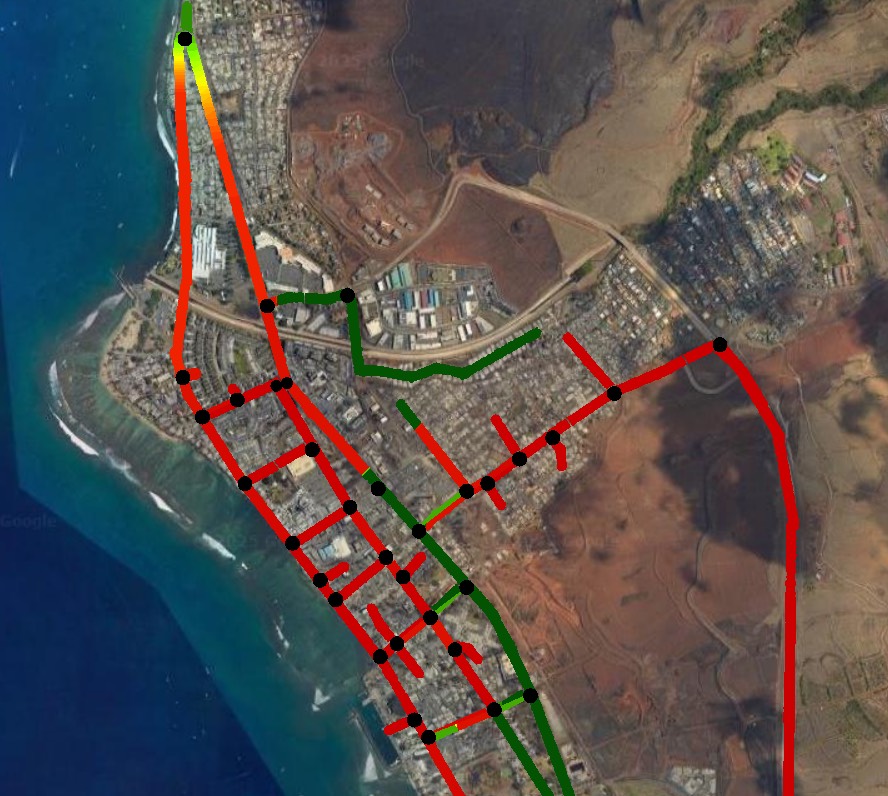}}\hskip1ex
    \subfloat[]{\includegraphics[width=0.3\textwidth]{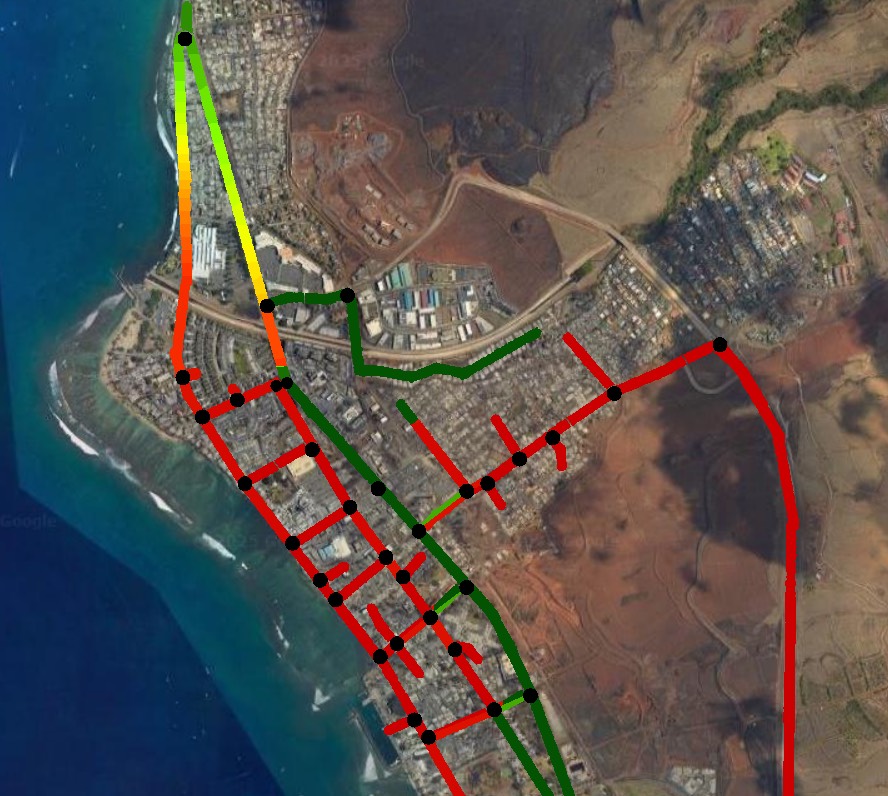}}\hskip1ex
    \subfloat[]{\includegraphics[width=0.3\textwidth]{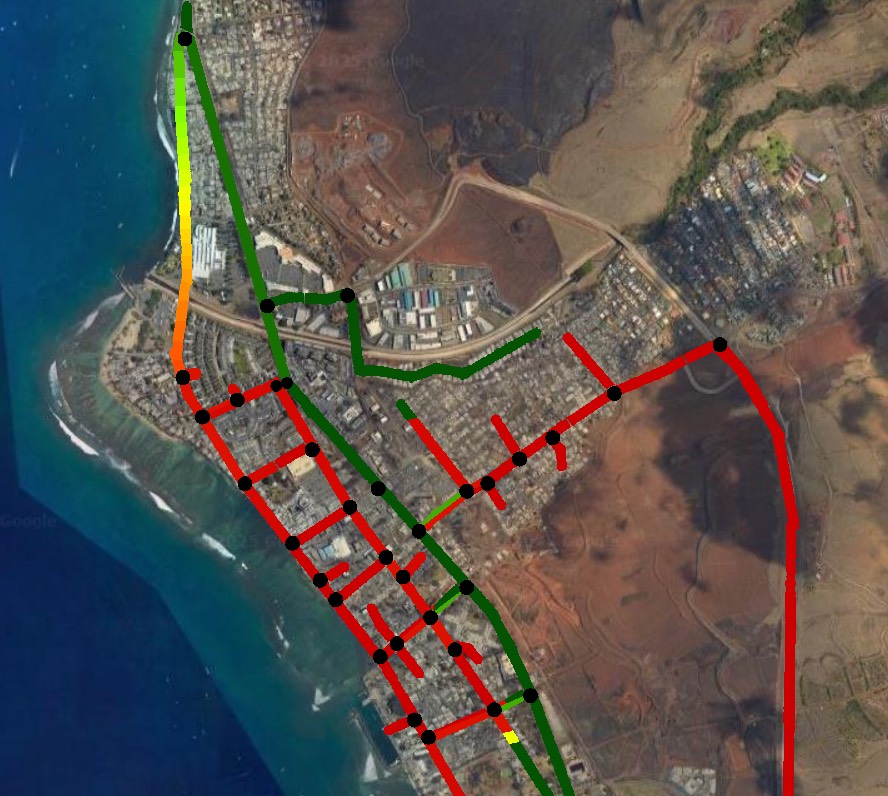}}\\
     \makebox[0.3\textwidth]{$t = 300$}\hskip1ex
    \makebox[0.3\textwidth]{$t = 1200$}\hskip1ex
    \makebox[0.3\textwidth]{$t = 1800$}
        \caption{Snapshots of PM 2 for 3 exit lanes with $\gamma_1 = 0.0375$, $\gamma_2 = 0.0375$.}
        \label{fig:gamma2_0.0375_3lane_snapshots}
\end{figure}

\begin{figure}[htbp]
     \centering
     \subfloat[]{\includegraphics[width=0.3\textwidth]{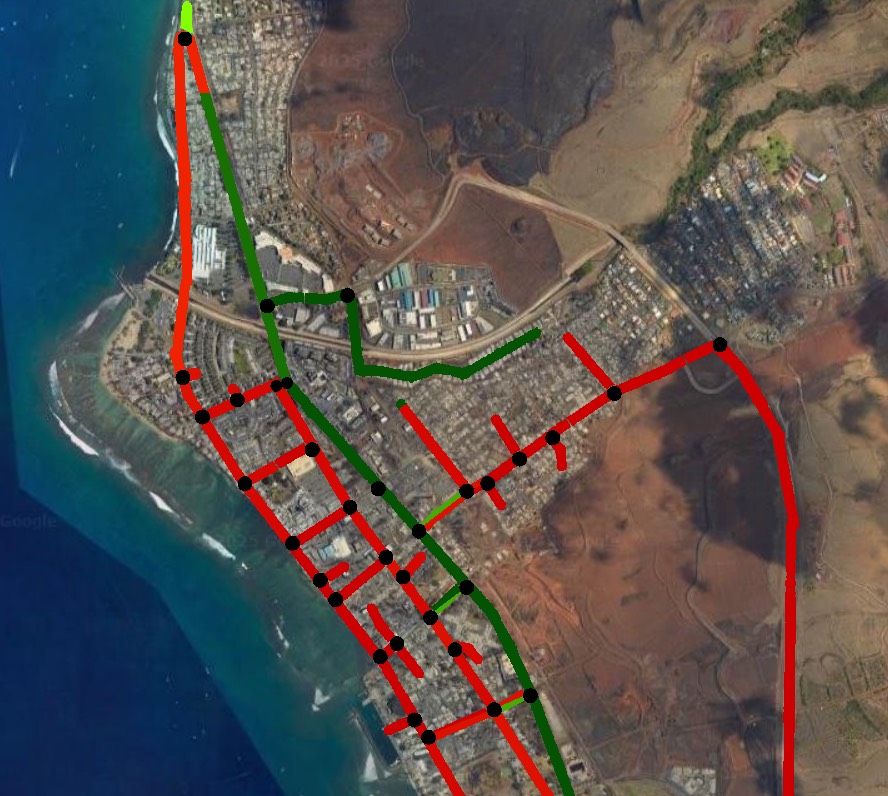}}\hskip1ex
    \subfloat[]{\includegraphics[width=0.3\textwidth]{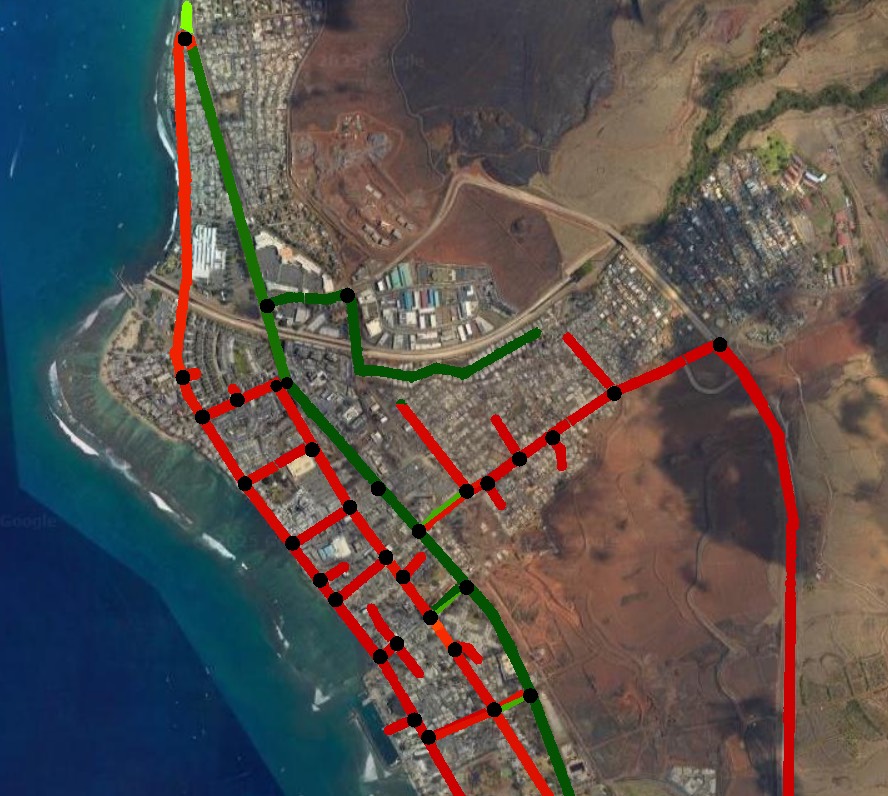}}\hskip1ex
    \subfloat[]{\includegraphics[width=0.3\textwidth]{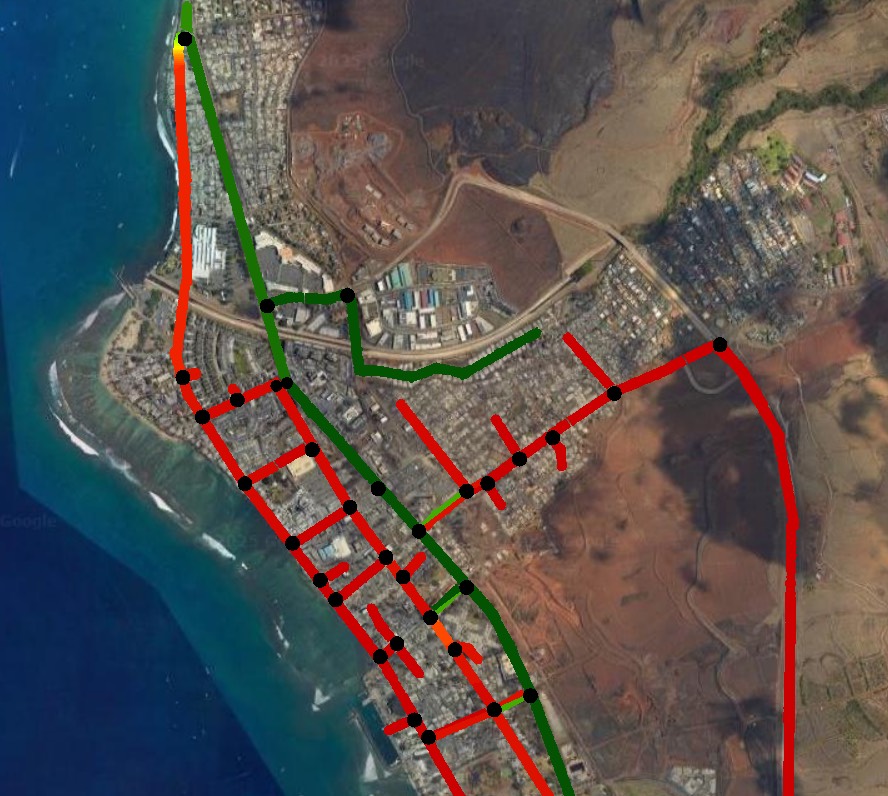}}\\
     \makebox[0.3\textwidth]{$t = 0$}\hskip1ex
    \makebox[0.3\textwidth]{$t = 300$}\hskip1ex
    \makebox[0.3\textwidth]{$t = 600$}
        \caption{Snapshots of PM 4 for 2 exit lanes with $\gamma_1 = 0.0375$, $\gamma_2 = 0.0375$. }
        \label{fig:gamma2_0.0375_2lane_snapshots}
\end{figure}
For the congested regime ($\gamma_2 = 0.1000$), the cumulative difference in the number of cars exited between the standard exit with 2 lanes and the exits with increased capacity is shown in Figure \ref{fig:gamma2_0.1000_cars_exit}, while snapshots of the final network at the end of PM 4 for the different lane values are given in \ref{fig:gamma2_0.1000_final_snapshots}.

\begin{figure}[htbp]    \centering\includegraphics[width=0.75\linewidth]{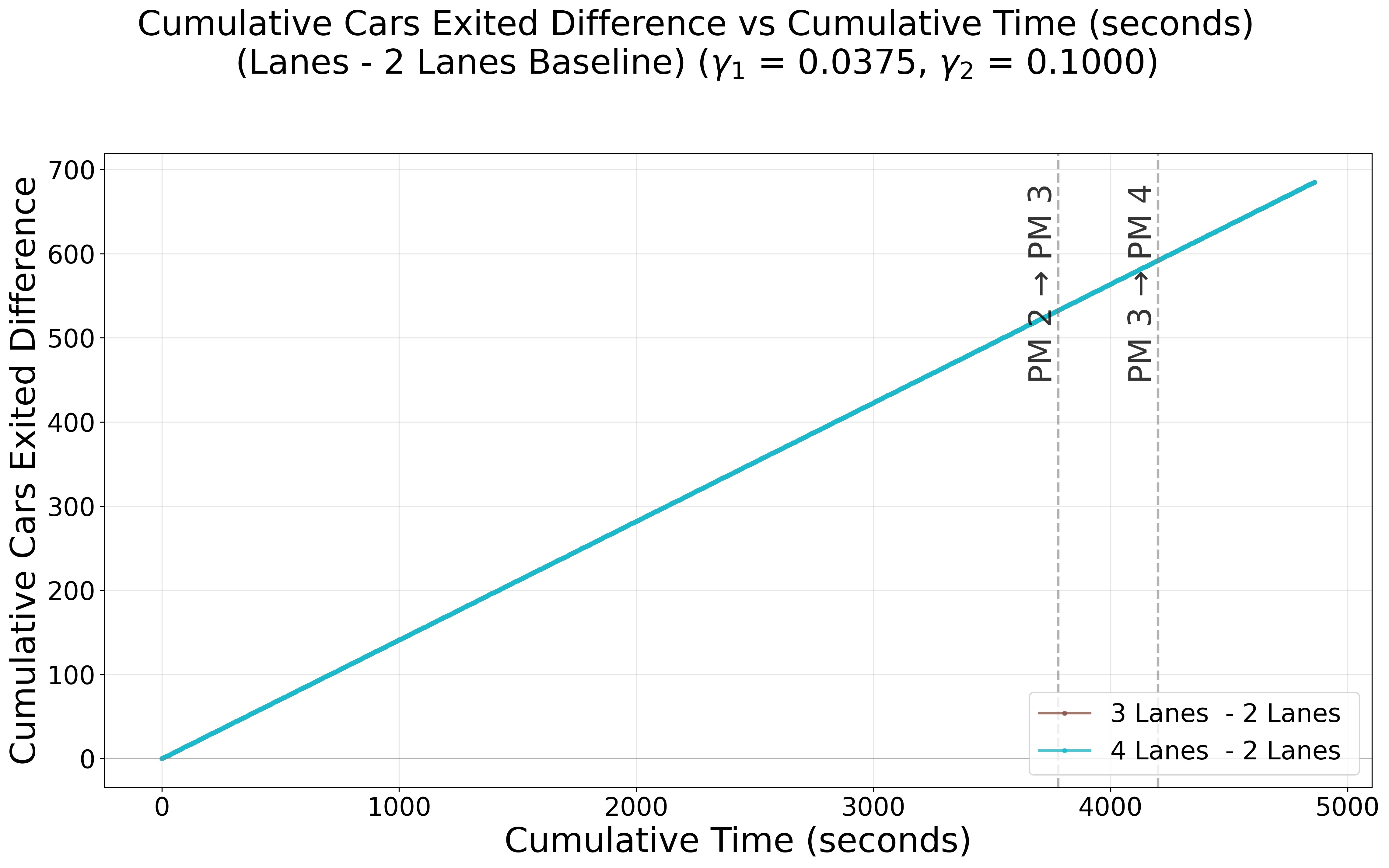}
    \caption{The cumulative difference in cars exited between 3 and 4 exit lanes and 2 lanes across PM networks 2, 3, and 4 for $\gamma_1 = 0.0375$, $\gamma_2 = 0.1000$.}
    \label{fig:gamma2_0.1000_cars_exit}
\end{figure}

\begin{figure}[htbp]
     \centering
     \subfloat[]{\includegraphics[width=0.3\textwidth]{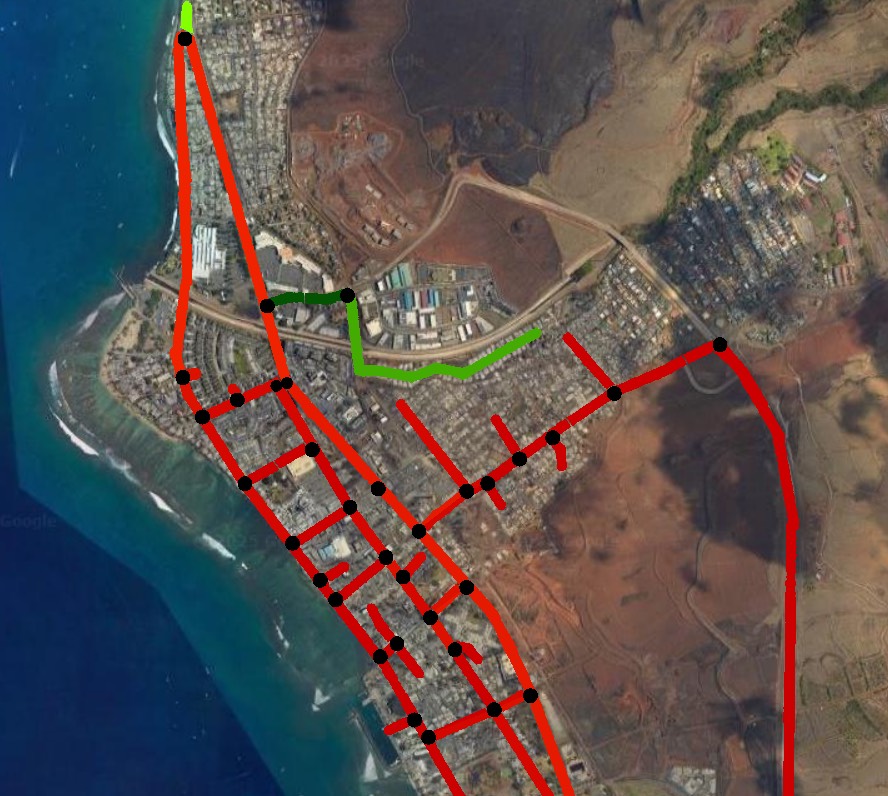}}\hskip1ex
    \subfloat[]{\includegraphics[width=0.3\textwidth]{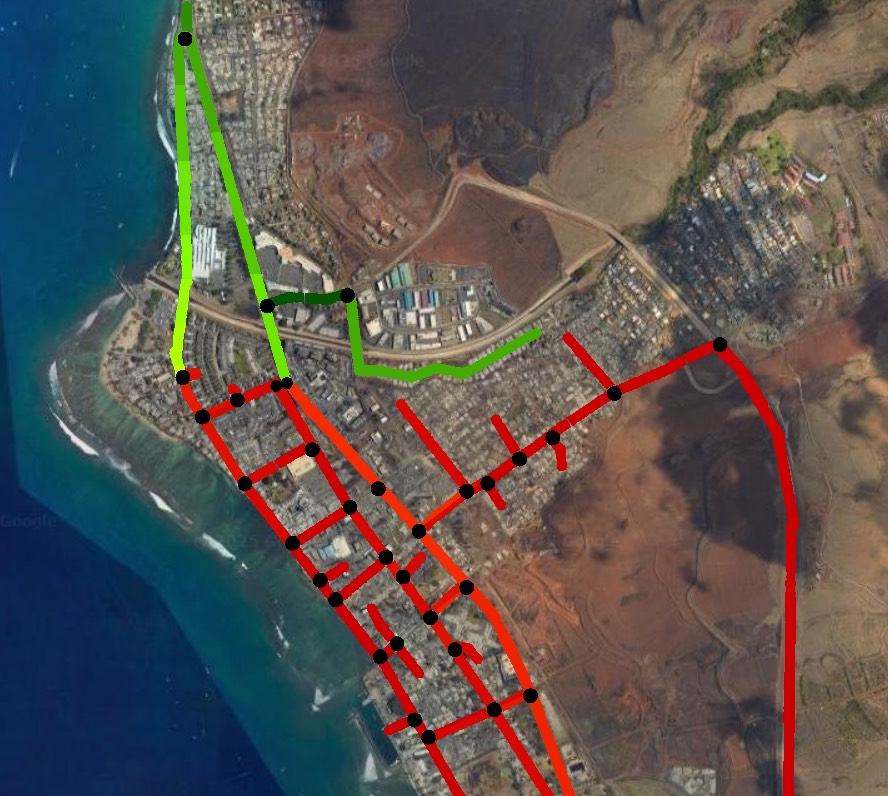}}\hskip1ex
    \subfloat[]{\includegraphics[width=0.3\textwidth]{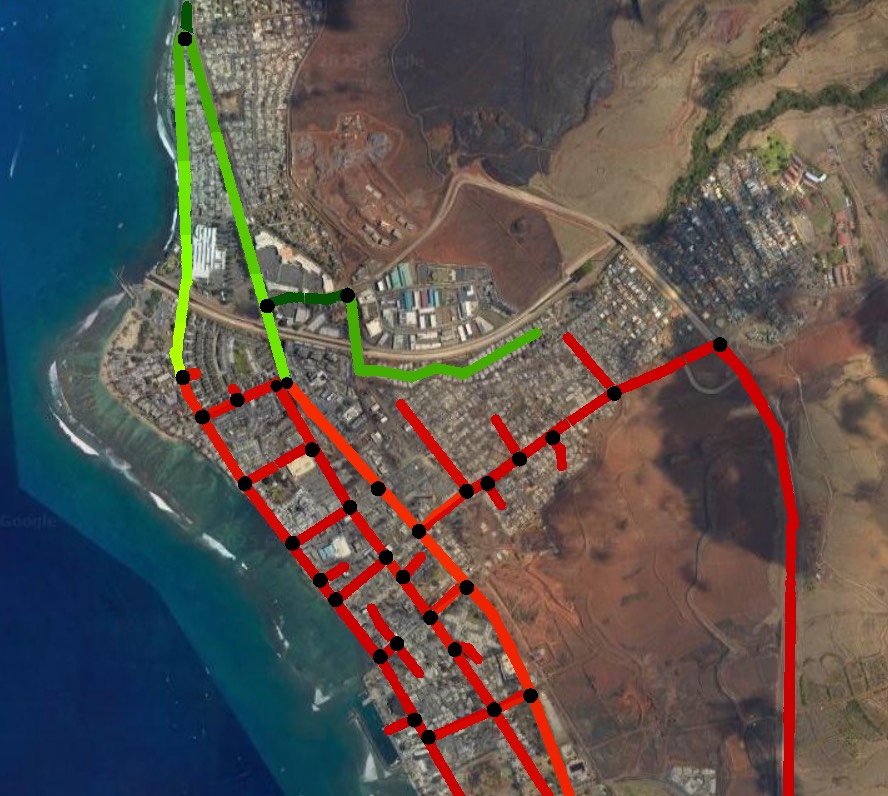}}\\
     \makebox[0.3\textwidth]{2 lanes}\hskip1ex
    \makebox[0.3\textwidth]{3 lanes}\hskip1ex
    \makebox[0.3\textwidth]{4 lanes}
        \caption{Final snapshots of PM 4 for 2,3, and 4 exit lanes with $\gamma_1 = 0.0375$, $\gamma_2 = 0.1000$. }
        \label{fig:gamma2_0.1000_final_snapshots}
\end{figure}
It can be seen that the performance of adding one exit lane and two exit lanes is nearly identical, with both resulting in an additional 685 vehicles escaping by the end of the simulation when compared to the default configuration. Similarly to the uncongested case, the addition of more exit lanes results in significantly less congestion near the exit. However, there are no phase transitions, as the network is too congested to fully clear any roads.

In both the uncongested and congested simulations, it is clear that reversing southbound lanes to provide additional exit lanes results in increased network efficiency, allowing more vehicles to escape when compared to the default setting of only two exit lanes. However, it is shown that adding two additional exit lanes results in little to no improvement over adding a single additional exit lane. Solving for the number of lanes required to match the maximum of the incoming fluxes, we have 
\begin{gather*}
    f_c^{in} = 1000 \cdot 2 + 500 = 2500 \Rightarrow 2500 = n \cdot 1000 \Rightarrow n = 2.5 \text{ lanes}
\end{gather*}
Thus, 2 exit lanes are insufficient to handle the max incoming flux, but with 3 exit lanes, the capacity is large enough to fully handle the incoming cars, so adding more exit lanes past 3 does not add any benefit. Therefore, we recommend that one southbound lane be reversed during evacuations, as this would allow the maximum number of vehicles to escape, while the last southbound lane can be reserved for emergency vehicles to enter the network. 

\section{Conclusions and Future Research}\label{section:conclusion}

This paper developed a macroscopic traffic flow framework for wildfire evacuation, combining game-theoretic junction resolution, evacuation-calibrated flux functions, network-wide preference optimization, and Godunov-based numerical simulation. Applied to the 2023 Lahaina wildfire, the framework yields the following conclusions.
\begin{itemize}[leftmargin=5pt]
\item Optimizing driver preferences to favor exit-facing roads improves evacuation throughput when the network is not yet saturated (Sections \ref{sec:phase-1}--\ref{sec:phase-2}).
\item Once roads are gridlocked, route optimization provides diminishing returns. The binding constraint is road capacity, not driver behavior (Section \ref{sec:phase-1}).
\item Reversing one southbound lane on Honoapi`ilani Highway (creating 3 total northbound exit lanes) captures nearly all achievable improvement, enabling 685 additional vehicles to escape in the congested regime. A fourth lane yields no further benefit and can be reserved for emergency vehicles (Section \ref{sec:phase-4}).
\item Opening southbound exits in the afternoon was critical for clearing residents, suggesting that future plans should prioritize having multiple exit directions over optimizing a single route (Section \ref{sec:phase-3}).
\end{itemize}
Underpinning these findings is a phase transition in exit lane capacity (Proposition \ref{closedform-CD}), which separates an \emph{exit-capacity-limited} regime from a \emph{supply-limited} regime at a threshold computable from flux ratios.

The phase transition identified here is structural, not Lahaina-specific. It arises whenever a network funnels traffic from multiple incoming roads through a single bottleneck exit. Wildland-urban interface communities with similar canyon-to-arterial topologies (Topanga Canyon, the Malibu corridor during the Woolsey fire, the Berkeley Hills) would exhibit analogous bottleneck-dominated behavior, and the critical lane threshold is computable from the same flux-ratio formula. More broadly, the end-to-end framework generalizes to other traffic networks and objectives (urban planning, event traffic management) through appropriate choice of the loss function. For such networks, our results suggest that targeted infrastructure investments in exit capacity may yield greater evacuation improvements than real-time route optimization alone. Future work includes extending the phase transition analysis to networks with multiple competing bottlenecks, incorporating real-time data assimilation (e.g., cell phone mobility data, connected vehicle trajectories) for dynamic re-optimization during an evolving evacuation, and adapting the numerical scheme to more general flux functions.

\section*{Acknowledgements}

We are grateful to the Fire Safety Research Institute (FSRI) and the authors of the Lahaina Fire Comprehensive Timeline Report \citep{lahainafirereport} for their thorough documentation of the incident, which was instrumental in reconstructing the timeline and network conditions used in our simulations. We also thank the State of Hawai\okina i Department of Transportation, Highways Division, Maui District, for providing recent traffic volume data for Honoapi\okina ilani Highway and the Lahaina Bypass via private communication, which informed our estimates of road types and capacities. Finally, we extend our heartfelt thanks to the people of Lahaina. We hope that this work can contribute, even in a small way, to improving evacuation preparedness and resilience for their community, and we wish them continued strength and recovery. This work was supported by the National Science Foundation under grants CCF-2345255 and CCF-2345256.

\newpage


\bibliographystyle{AIMS}
\bibliography{references}

\newpage
\appendix
\markboth{\MakeUppercase{Appendix}}{\MakeUppercase{Appendix}}
\section{Numerical Implementation Details}\label{appendix:num_impl}
\subsection{Additional Parameters}For all simulations, the other relevant inputs into Algorithm \ref{alg:num-opt} are as follows - we chose $n_{\text{iter}} = 100, s = 10, c = 0.5, f = 0.5, \tau_{\text{init}} = 1$, and $N_{\text{decay}} = 10.$ The number of sampled indicies $s = 10$ was chosen because the AM Base network (as well as most of the other networks that we test) has around 50 parameters, so sampling 10 directions at every iteration samples a reasonable $1/5$ fraction of the available directions. The decay factor $f = 0.5$ and maximum number of decays $N_{\text{decay}} = 10$ allow for steps as small as $(1/2)^{10}$ in norm to be taken, which is enough precision for optimization purposes. Lastly, the control factor $c = 0.5$ and initial step size $\tau_{\text{init}} = 1$ are standard choices when using the Armijo-Goldstein line search condition, and we chose $n_{\text{iter}} = 100$ after observing that further iterations did not improve the objective function by much.

\subsection{Network Diagrams}
Full network diagrams are provided below, illustrating the road modifications made at each stage of the simulation.
\begin{figure}[H]
     \centering
     \subfloat[AM base network with changes]{\includegraphics[width=0.45\textwidth]{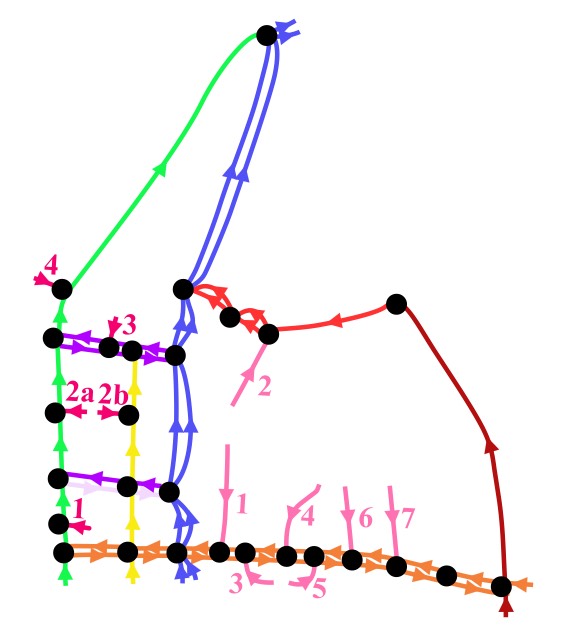}\label{fig:am2_temp}}\hskip5ex
    \subfloat[AM network 2]{\includegraphics[width=0.45\textwidth]{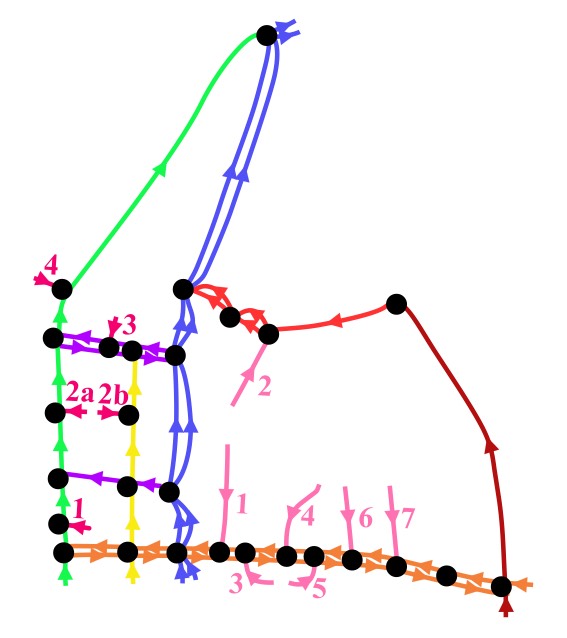}\label{fig:am2_final}}
        \caption{AM Network 2. To be simulated from 13:25-14:21. Figure (A) shows the roads to be removed from the AM base network in gray: (gray-purple) Eastbound Papalaua Street Segment from Front Street to Hwy-30. Figure (B) shows the final network to be simulated for AM network 2.}
        \label{fig:am2}
\end{figure}
\begin{figure}[H]
     \centering
     \subfloat[AM network 2 with changes]{\includegraphics[width=0.45\textwidth]{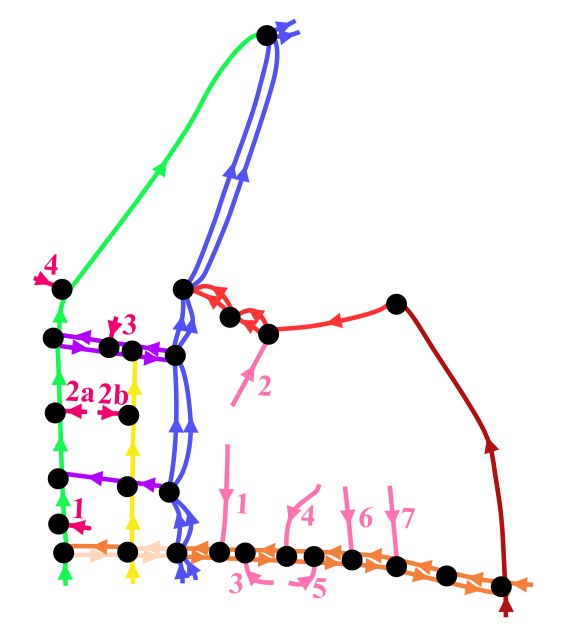}\label{fig:am3_temp}}\hskip5ex
    \subfloat[AM network 3]{\includegraphics[width=0.45\textwidth]{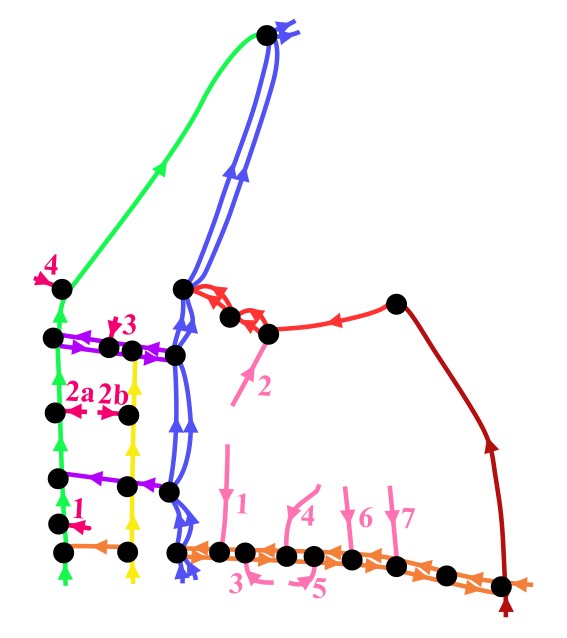}\label{fig:am3_final}}
        \caption{AM Network 3. To be simulated from 14:21-15:00. Figure (A) shows the roads to be removed from the PM base network in gray: (gray-orange) Eastbound Lahainaluna Road Segment from Front Street to Hwy-30, and Westbound Lahainaluna Road Segment from Waine'e Street to Hwy-30. Figure (B) shows the final network to be simulated for AM network 3.}
        \label{fig:am3}
\end{figure}
\begin{figure}[H]
     \centering
     \subfloat[PM base network with changes]{\includegraphics[width=0.45\textwidth, height=0.28\textheight]{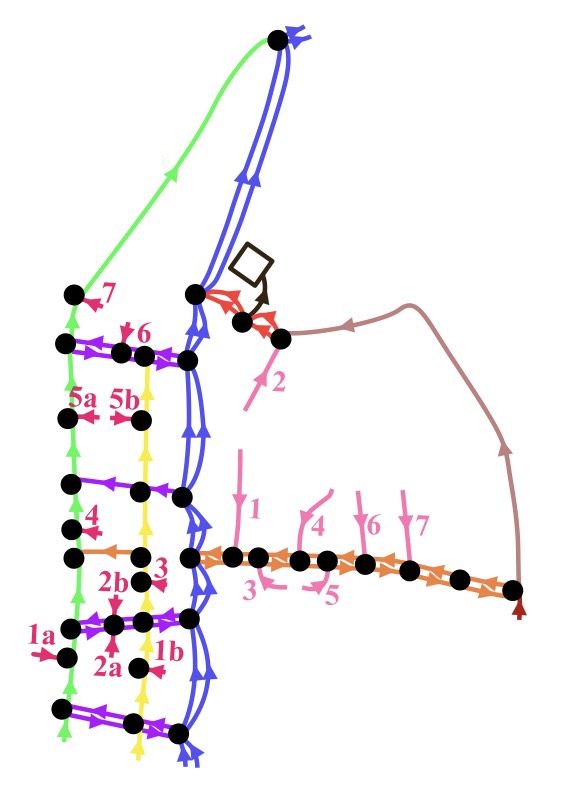}\label{fig:pm2_temp}}\hskip5ex
    \subfloat[PM network 2]{\includegraphics[width=0.45\textwidth, height=0.28\textheight]{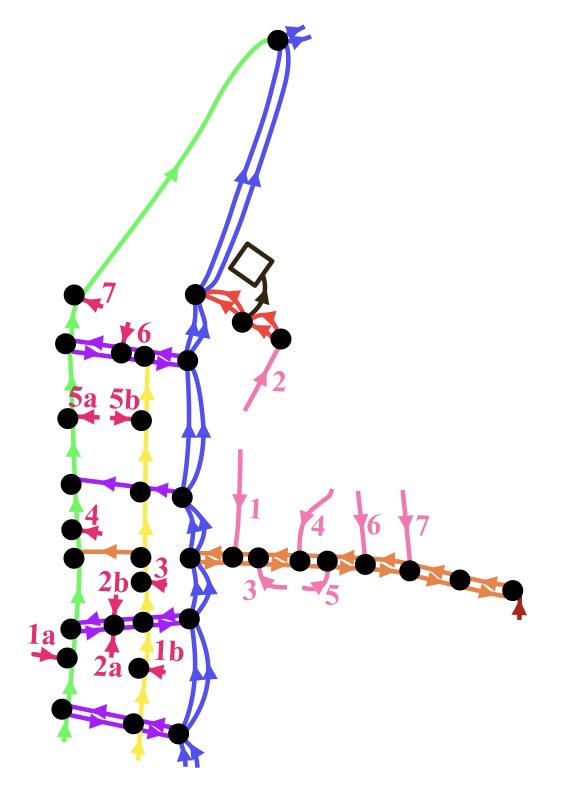}\label{fig:pm2_final}}
        \caption{PM Network 2. To be simulated from 15:25-16:29. Figure (A) shows the roads to be removed from the PM base network in gray: (gray-red) Lahaina Bypass Segment from Lahainaluna Road to Keawe Street Extension. Figure (B) shows the final network to be simulated for PM network 2. }
        \label{fig:pm2}
\end{figure}
\begin{figure}[H]
      \centering
      \subfloat[PM network 2 with changes]{\includegraphics[width=0.45\textwidth, height=0.28\textheight]{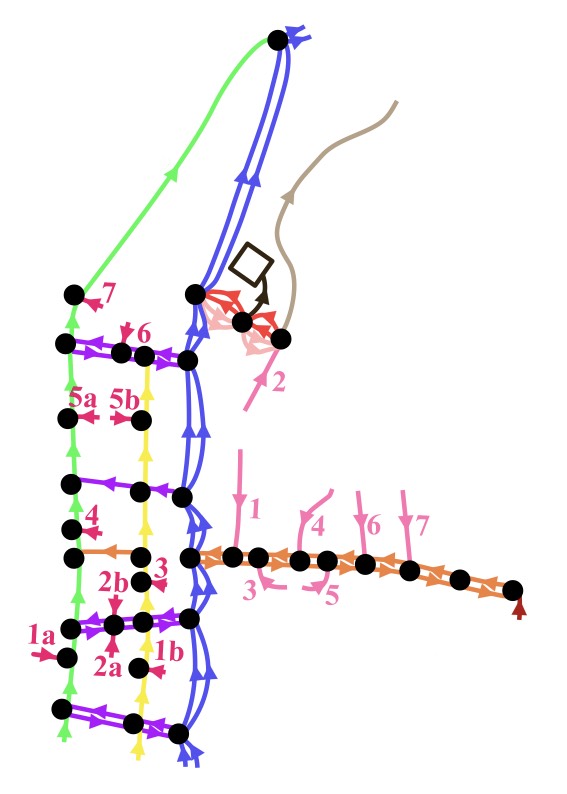}\label{fig:pm3_temp}}\hskip5ex
     \subfloat[PM network 3]{\includegraphics[width=0.45\textwidth, height=0.28\textheight]{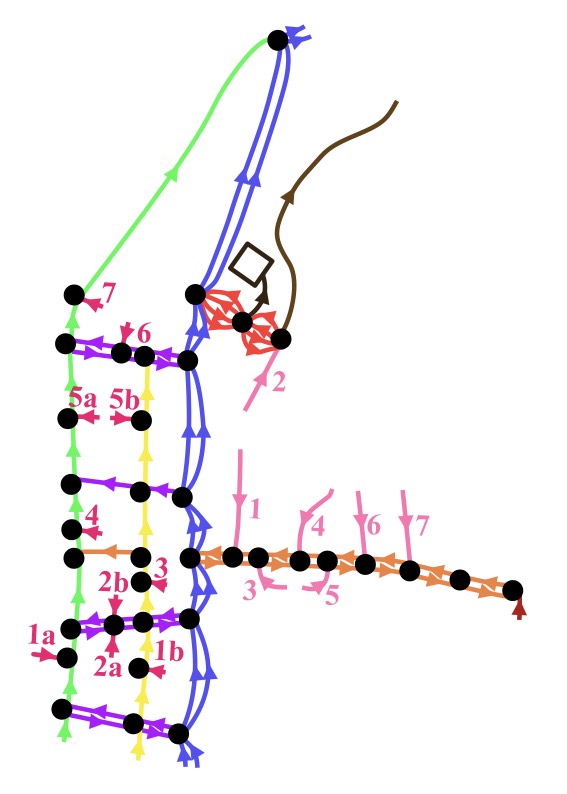}\label{fig:pm3_final}}
        \caption{PM Network 3. To be simulated from 16:29-16:35. Figure (A) shows the additional roads added to PM network 2 in gray: (gray-red) Eastbound Keawe Street Extension, (gray-brown) Oil Road. Figure (B) shows the final network to be simulated for PM network 3. }
        \label{fig:pm3}
\end{figure}
\begin{figure}[H]
     \centering
     \subfloat[PM network 3 with changes]{\includegraphics[width=0.45\textwidth, height=0.28\textheight]{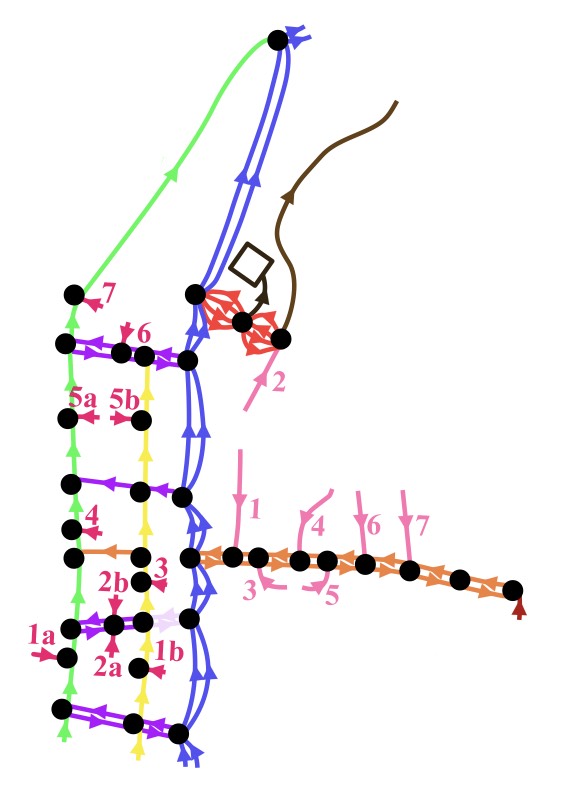}\label{fig:pm4_temp}}\hskip5ex
    \subfloat[PM network 4]{\includegraphics[width=0.45\textwidth, height=0.28\textheight]{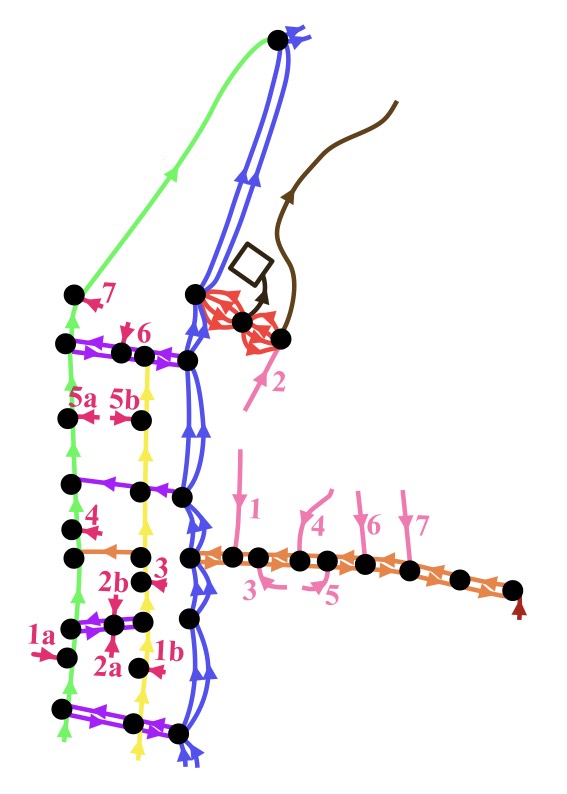}\label{fig:pm4_final}}
        \caption{PM Network 4. To be simulated from 16:35-16:47. Figure (A) shows the roads removed from PM network 3 in gray: (gray-purple) Dickenson Street segment connecting Waine'e Street and Hwy-30. 
        Figure (B) shows the final network to be simulated for PM network 4. }
        \label{fig:pm4}
\end{figure}
\begin{figure}[H]
     \centering
     \subfloat[PM network 4 with changes]{\includegraphics[width=0.45\textwidth, height=0.28\textheight]{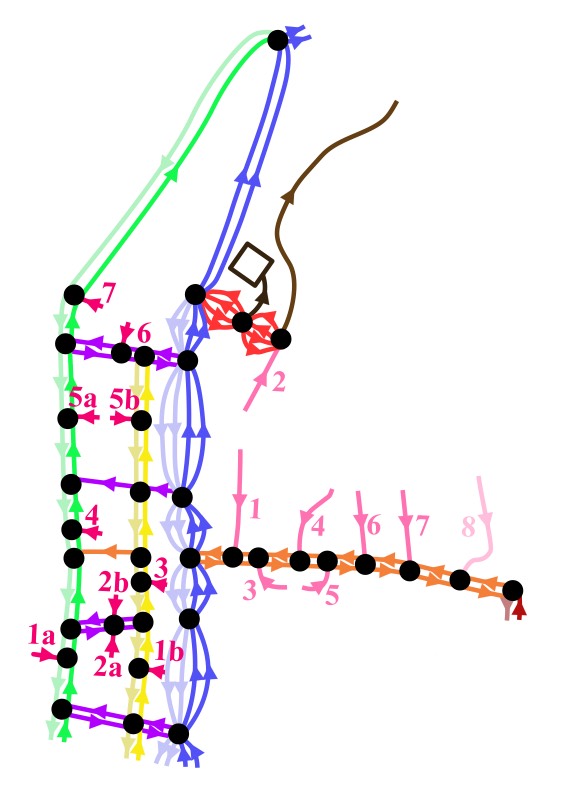}\label{fig:pm5_temp}}\hskip5ex
    \subfloat[PM network 5]{\includegraphics[width=0.45\textwidth, height=0.28\textheight]{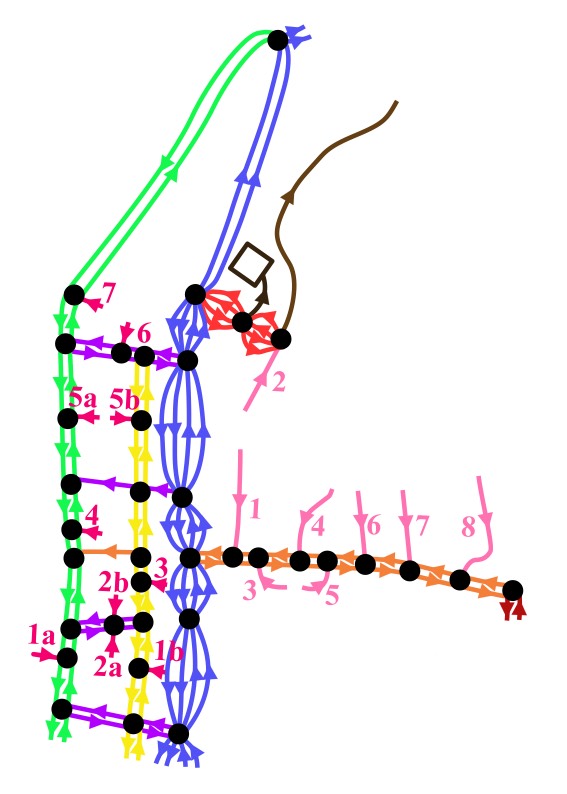}\label{fig:pm5_final}}
        \caption{PM Network 5. To be simulated from 16:47-17:10. Figure (A) shows the additional roads added to PM network 4 in gray: (gray-green) Southbound Front Street, (gray-yellow) Southbound Waine'e Street, (gray-blue) Southbound Hwy-30, (gray-maroon) Exit via Lahaina Bypass, (gray-light pink) Source via nondescript dirt road.
        Figure (B) shows the final network to be simulated for PM network 5. }
        \label{fig:pm5}
\end{figure}

\FloatBarrier

\subsection{Phase 1 Experiments}

We examine the combinations of parameters utilizing three metrics; at the end of the simulation time, we examine our loss function of weighted time-integrated cars in the network, the total number of cars that were able to enter the network, and the total number of cars that were able to exit the network. These three metrics are required to accurately capture the dynamics of the simulations in all congestion scenarios, as the loss function performs better with higher $\gamma$ values, while the other two metrics perform better with lower $\gamma$ values.

{
\setlength{\tabcolsep}{12pt} 
\begin{table}[htbp]
  \centering
  \begin{tabular}{lccc}
    \toprule
    \textbf{$nt_{opt}$} 
      & Weighted Time-Integrated Cars & Cars Entered & Cars Exited\\
    \midrule
    0 & 747.73 & 1,885.00 & 2,151.70\\
    1 & 728.50 & 1,885.00 & 2,164.13\\
    10 & 729.01 & 1,885.00 & 2,165.74\\
    60 & 735.51 & 1,885.00 & 2,166.35\\
    600 & 701.94 & 1,885.00 & 2,180.22\\
    \bottomrule
  \end{tabular}
    \caption{Optimization Metrics for Various $nt_{opt}$ Times (seconds) with $\gamma = 0.01$}\label{tab:ntopt_comparison_gamma0.01}
\end{table}
}
\begin{figure}[htbp]
     \centering
     \subfloat[]{\includegraphics[width=0.3\textwidth]{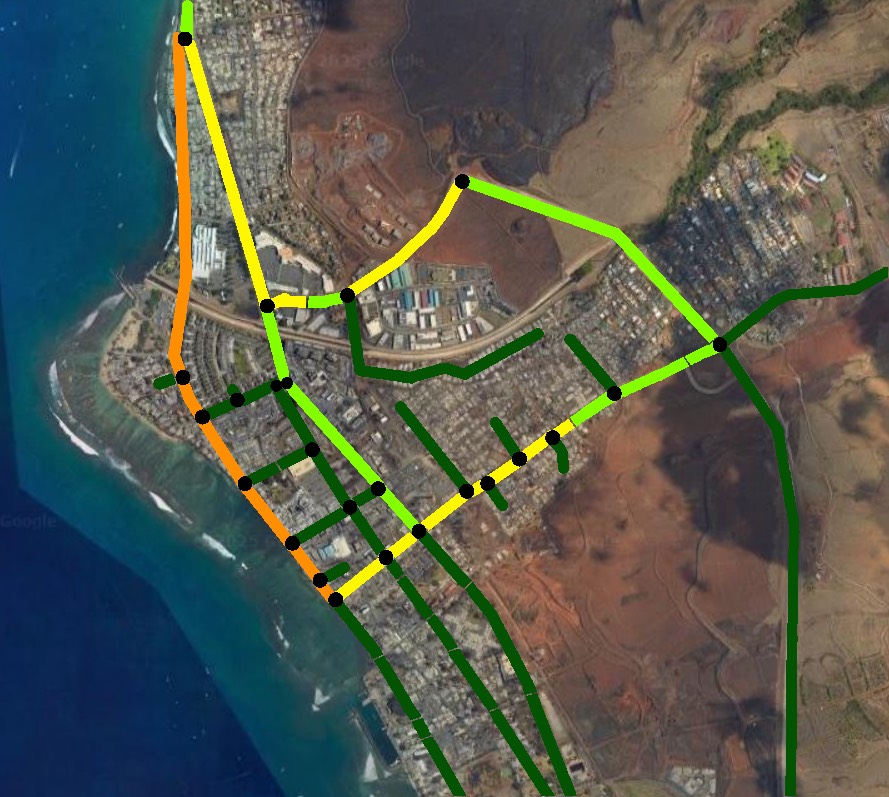}\label{fig:t_0_0.01}}
     \hskip1ex
    \subfloat[]{\includegraphics[width=0.3\textwidth]{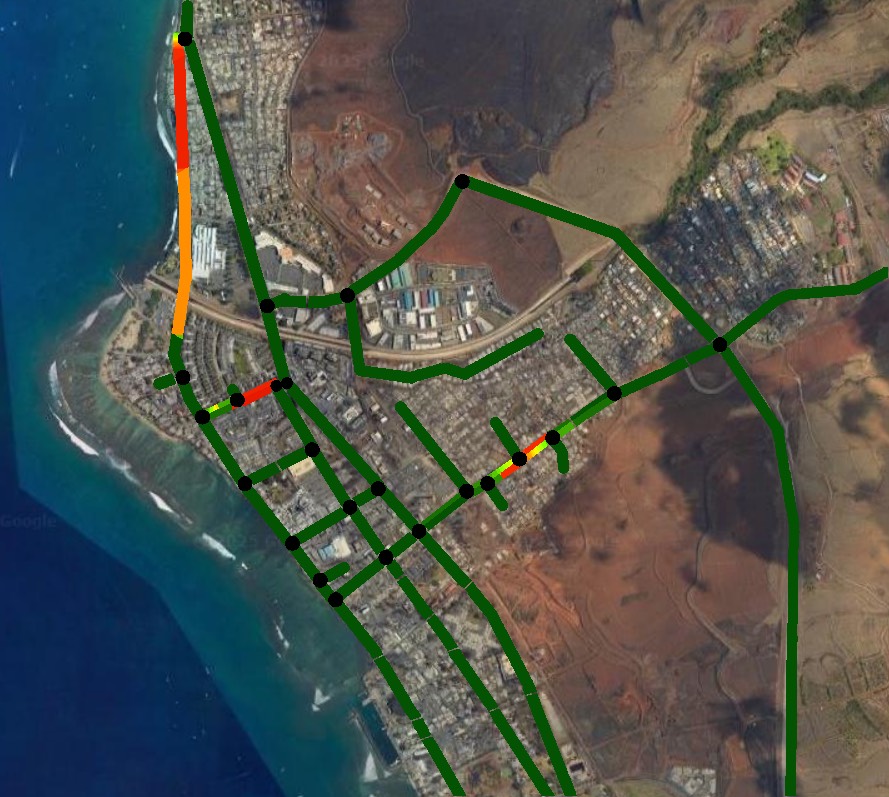}\label{fig:t_6k_0.01}}
     \hskip1ex
    \subfloat[]{\includegraphics[width=0.3\textwidth]{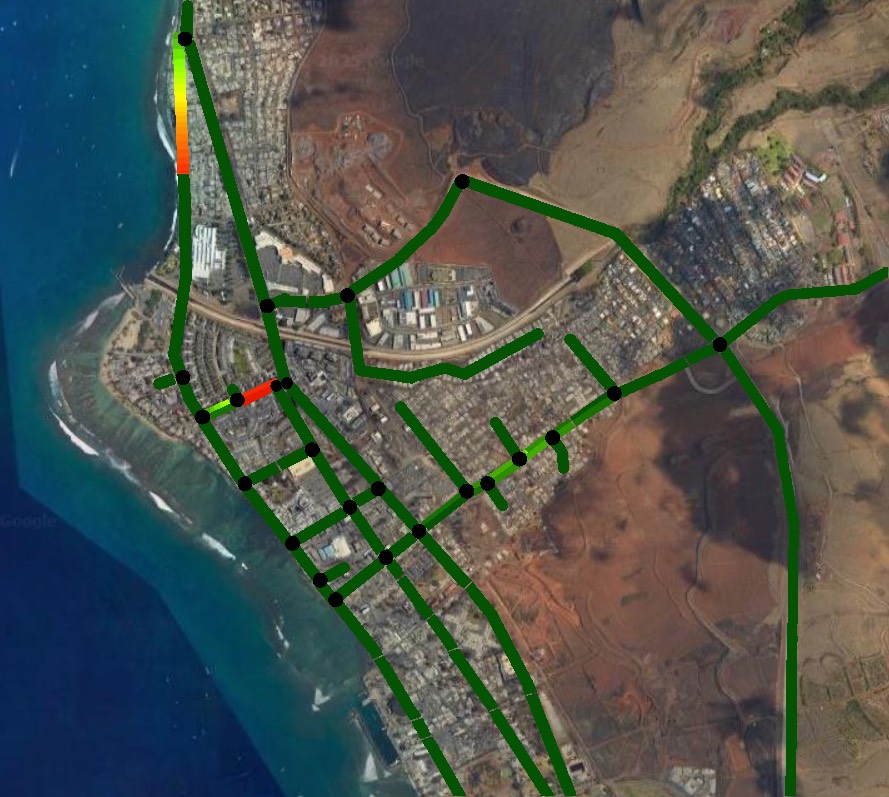}\label{fig:t_12k_0.01}}\\
    \makebox[0.3\textwidth]{$t=0$}\hskip1ex
    \makebox[0.3\textwidth]{$t=600$}\hskip1ex
    \makebox[0.3\textwidth]{$t=1200$}
        \caption{Snapshots of the AM Base Network at $t=0, 600, 1200$ seconds for $\gamma = 0.01$ with $nt_{opt}=0$ seconds.}
        \label{fig:gamma_0.01_snapshots_ntopt0}
\end{figure}
For $\gamma = 0.01$, source roads begin uncongested, and the network fully clears by the end of the simulation for all $nt_{opt}$ values. This can be seen in Figure \ref{fig:gamma_0.01_snapshots_ntopt0}. In Table \ref{tab:ntopt_comparison_gamma0.01}, it can be seen that the weighted time-integrated cars fluctuating, and it is the lowest for $nt_{opt} = 600$, which might suggest that optimizing for a longer period of time is less effective. However, this is not accurate - the loss function weights the time-integrated cars on a road by how close the road is to the nearest exit. When the network is congested, this represents pushing traffic towards the exits, which is ideal in an evacuation scenario. But when the network is very uncongested, as in this case with $\gamma=0.01$, this means that if there is any congestion on roads close to the exit, the loss function will be higher. With $nt_{opt}=0$, the network is actually more congested for a longer period of time, but this congestion happens to be near the exit, which is why the loss function seems to indicate better performance. This is confirmed by examining the total number of cars exited, showing that the optimized simulations result in more cars being able to exit the network by the end of simulation time, with $nt_{opt} = 600$ performing the best. Since all networks fully clear by the end of the simulation, these differences in the network's performance can be seen by examining snapshots of the network at time $t=600$ seconds, shown in Figure \ref{fig:gamma_0.01_snapshots}.
\begin{figure}[htbp]
     \centering
     \subfloat[ ]{\includegraphics[width=0.3\textwidth]{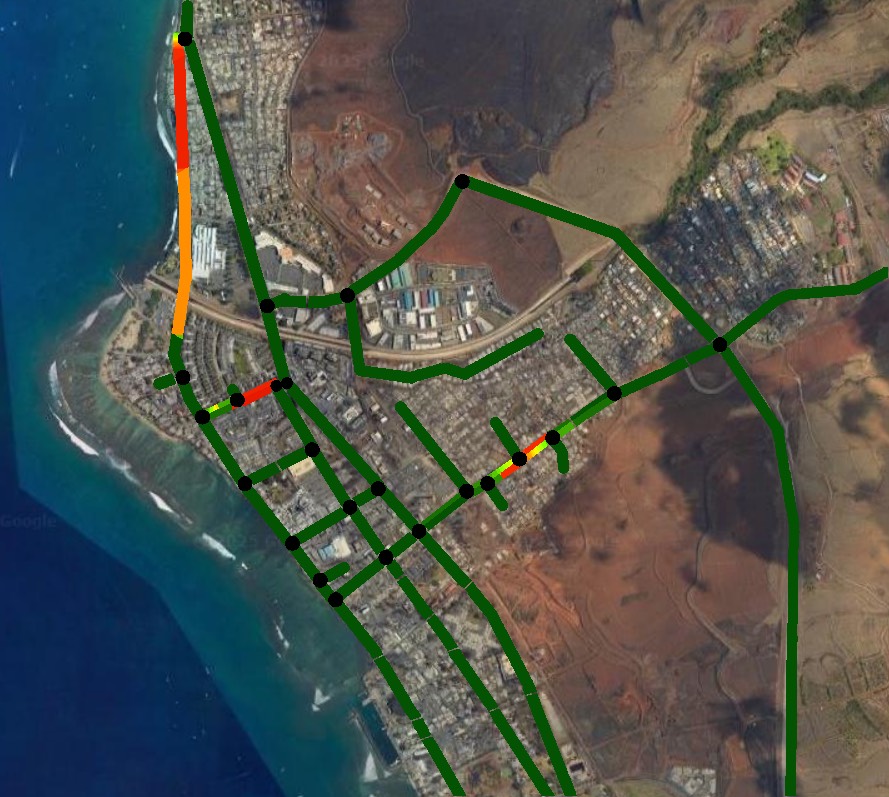}\label{fig:gamma0.01_ntopt0}}\hskip1ex
    \subfloat[]{\includegraphics[width=0.3\textwidth]{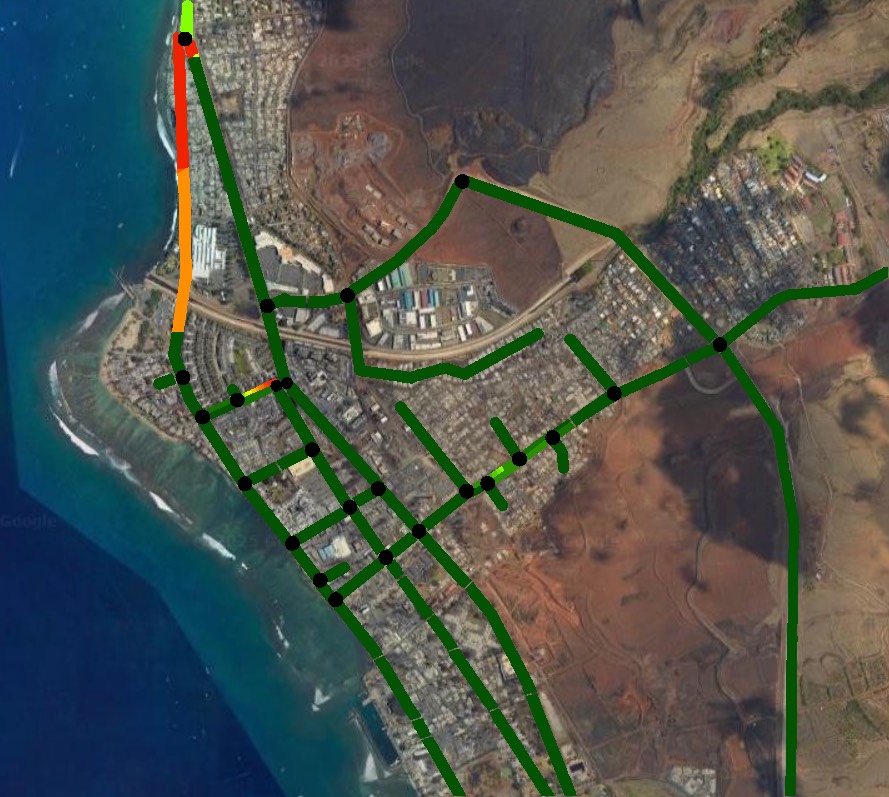}\label{fig:gamma0.01_ntopt10}}\hskip1ex
    \subfloat[]{\includegraphics[width=0.3\textwidth]{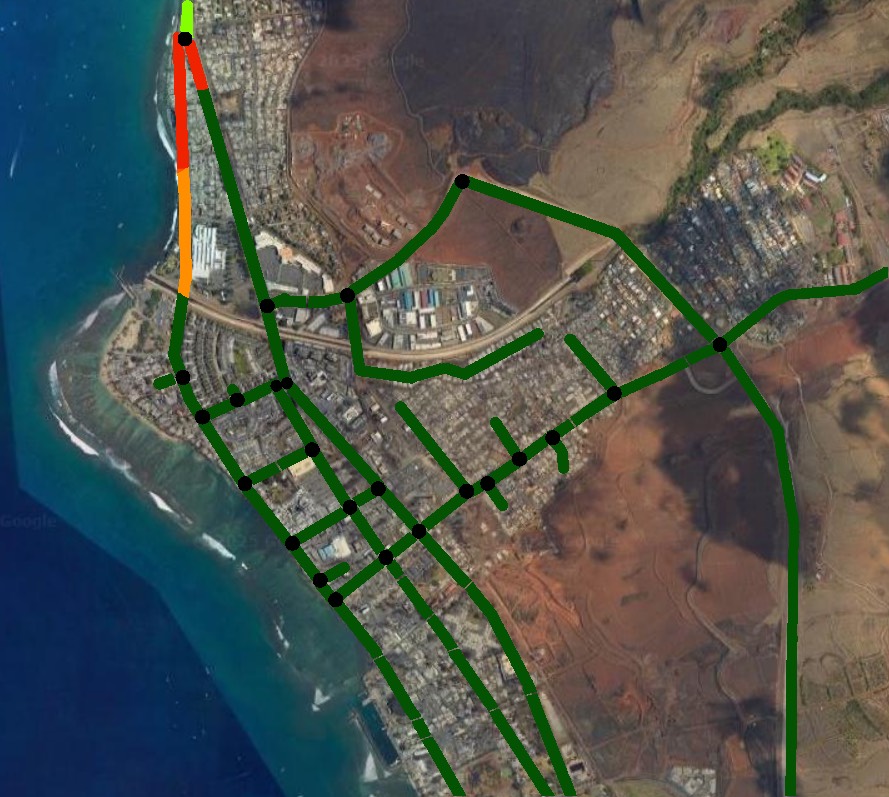}\label{fig:gamma0.01_ntopt60}}\\
    \makebox[0.3\textwidth]{$nt_{opt}=0$}\hskip1ex
    \makebox[0.3\textwidth]{$nt_{opt}=1$}\hskip1ex
    \makebox[0.3\textwidth]{$nt_{opt}=60$}
        \caption{Snapshots of the AM Base Network at $t=600$ for $\gamma = 0.01$ with $nt_{opt}=0,1$ and $60$ seconds.}
        \label{fig:gamma_0.01_snapshots}
\end{figure}

{
\setlength{\tabcolsep}{12pt} 
\begin{table}[htbp]
  \centering
  \begin{tabular}{lccc}
    \toprule
    \textbf{$nt_{opt}$} 
      & Weighted Time-Integrated Cars & Cars Entered & Cars Exited\\
    \midrule
    0 & 6,365.52 & 5,528.19 & 4,815.65\\
    1 & 6,485.01 & 5,546.28 & 4,815.65\\
    10 & 6,495.34 & 5,548.56 & 4,815.65\\
    60 & 6,591.61 & 5,547.20 & 4,815.65\\
    600 & 6,828.82 & 5,566.66 & 4,815.65\\
    \bottomrule
  \end{tabular} 
  \caption{Optimization Metrics for Various $nt_{opt}$ Times (seconds) with $\gamma = 0.0375$}\label{tab:ntopt_comparison_gamma0.0375}
\end{table}
}

\begin{figure}[H]
     \centering
     \subfloat[]{\includegraphics[width=0.3\textwidth]{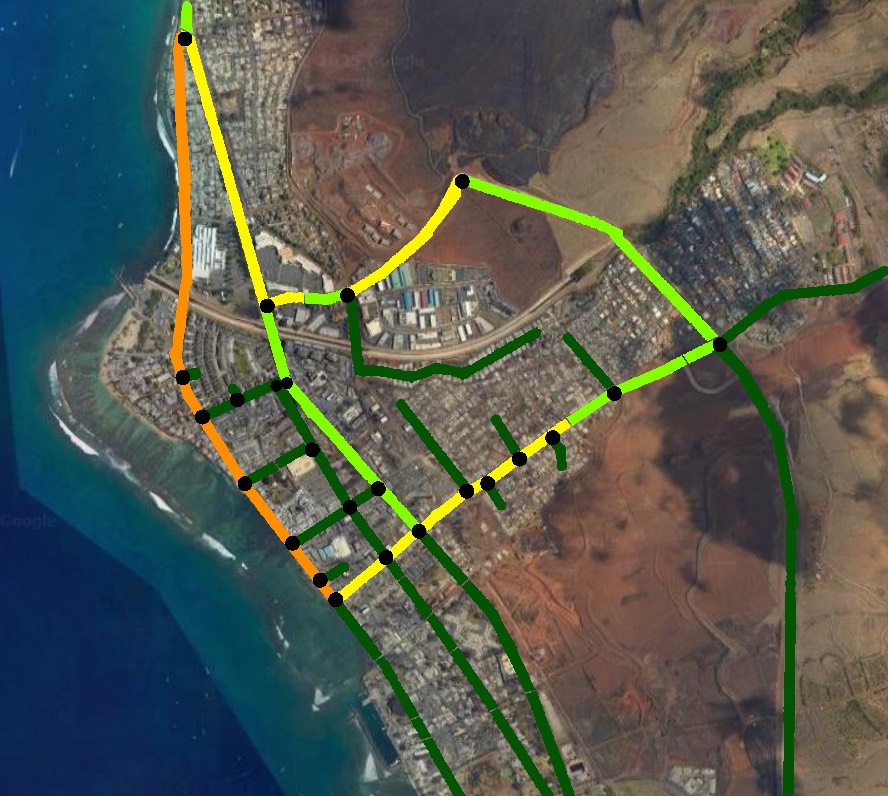}\label{fig:gamma0.0375_step0}}\hskip1ex
    \subfloat[]{\includegraphics[width=0.3\textwidth]{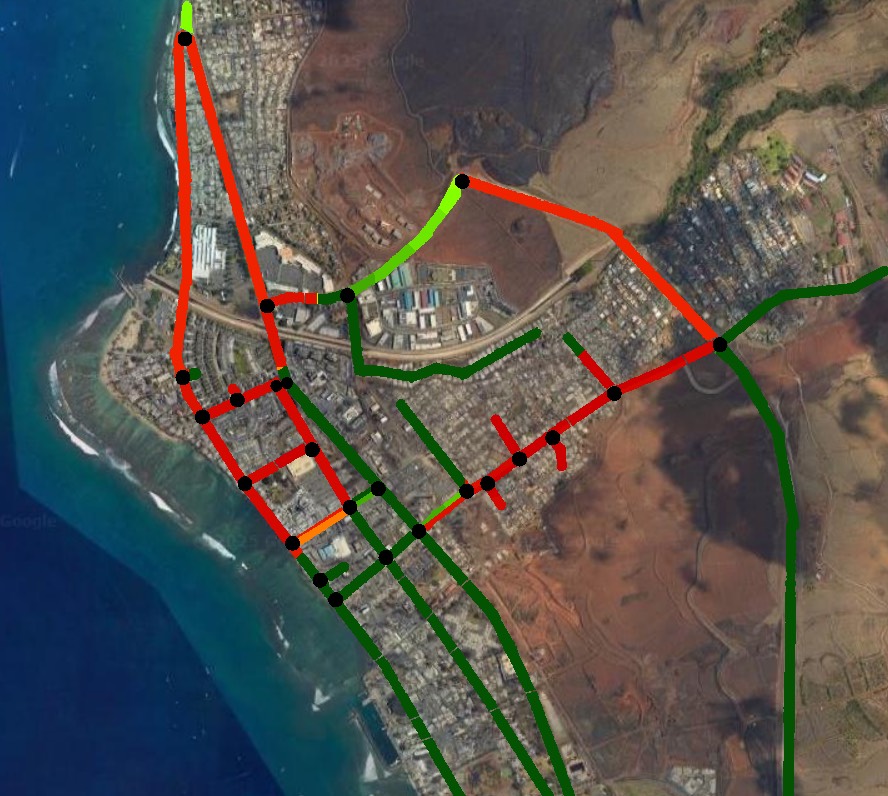}\label{fig:gamma0.0375_step30k}}\hskip1ex
    \subfloat[]{\includegraphics[width=0.3\textwidth]{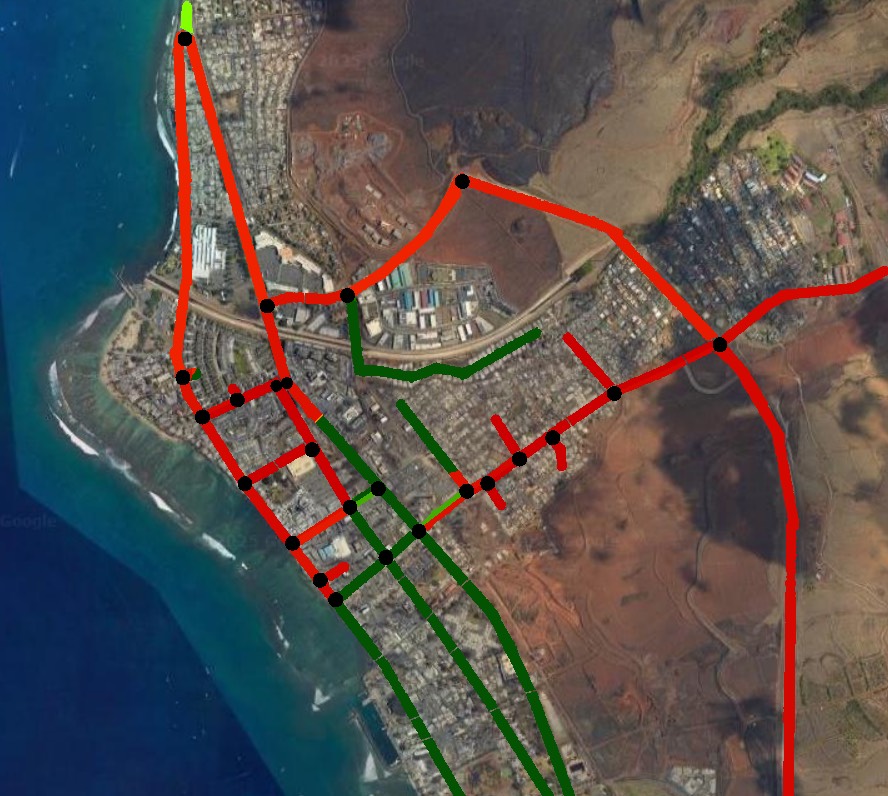}\label{fig:gamma0.0375_step60k}}\\
    \makebox[0.3\textwidth]{$t=0$}\hskip1ex
    \makebox[0.3\textwidth]{$t=3000$}\hskip1ex
    \makebox[0.3\textwidth]{$t=6000$}
        \caption{Snapshots of the AM Base Network at $t=0, 3000, 6000$ seconds for $\gamma = 0.0375$ with $nt_{opt}=0$ seconds.}
        \label{fig:gamma_0.0375_snapshots}
\end{figure}

When $\gamma = 0.0375$, the source roads are mildly congested initially, and the network slowly floods by the end of the simulation. Snapshots of the evolution for $nt_{opt}=0$ are given in Figure \ref{fig:gamma_0.0375_snapshots} for visualization, while snapshots for all the choices of $nt_{opt}$ are provided in the Github repository. Since the final network has congestion, the loss function acts as a reasonable metric for performance, with $nt_{opt} = 600$ performing the best. Table \ref{tab:ntopt_comparison_gamma0.0375} shows that the number of cars that can enter the network also increases with $nt_{opt}$, while the number of cars exiting is the same for all values. This is due to the entropy and junction conditions at the exit junction detailed in Section 2, where for any $nt_{opt}$ values, the congestion in the network overwhelms the system, so the maximum number of cars able to exit is the same. However, the loss function and cars entered show that the optimization still results in improved efficiency of the network. Figure \ref{fig:gamma_0.0375_snapshots_comparison} provides snapshots at $t = 900$ seconds displaying how increased $nt_{opt}$ values concentrate congestion towards the exit. 

\begin{figure}[htbp]
     \centering
     \subfloat[]{\includegraphics[width=0.30\textwidth]{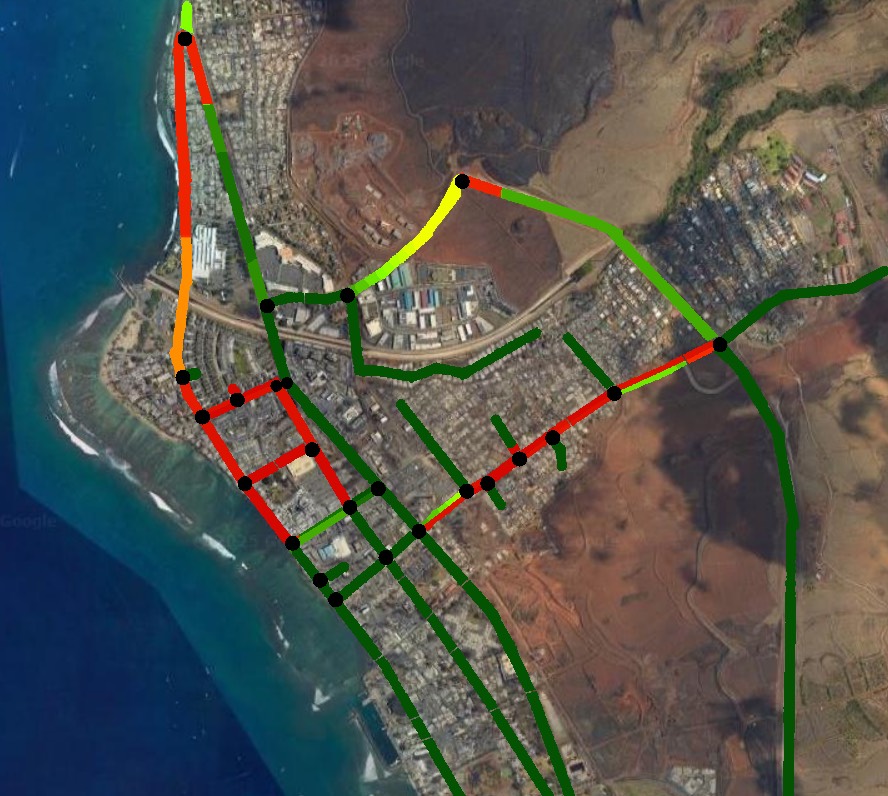}\label{fig:gamma0.0375_ntopt0}}\hskip1ex
    \subfloat[]{\includegraphics[width=0.3\textwidth]{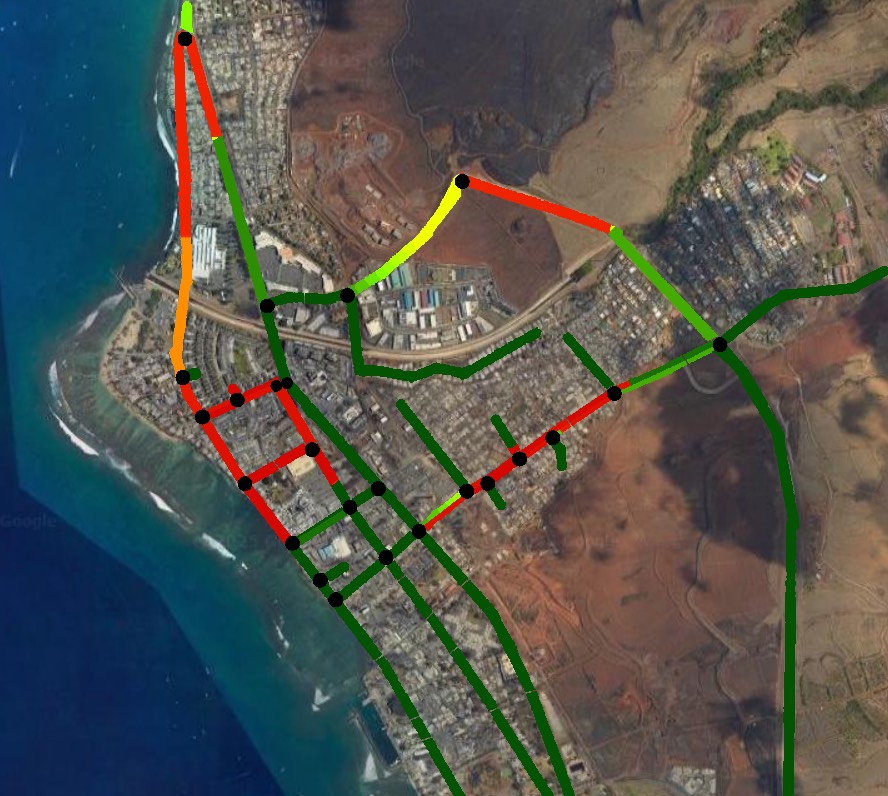}\label{fig:gamma0.00.0375_ntopt10}}\hskip1ex
    \subfloat[]{\includegraphics[width=0.3\textwidth]{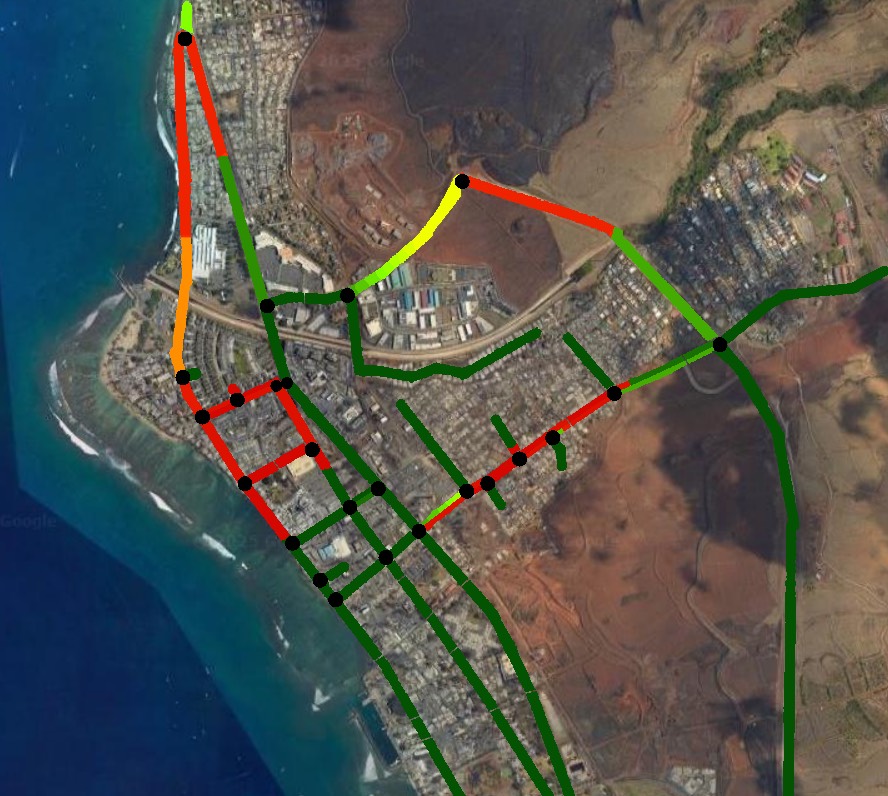}\label{fig:gamma0.0375_ntopt60}}\\
    \makebox[0.3\textwidth]{$nt_{opt}=0$}\hskip1ex
    \makebox[0.3\textwidth]{$nt_{opt}=1$}\hskip1ex
    \makebox[0.3\textwidth]{$nt_{opt}=60$}
        \caption{Snapshots of the AM Base Network at $t=900$ for $\gamma = 0.0375$ with $nt_{opt}=0,1$ and $60$ seconds.}
        \label{fig:gamma_0.0375_snapshots_comparison}
\end{figure}

\begin{table}[htbp]
\centering
\begin{tabular}{lccc}
\toprule
$nt_{opt}$ & $\gamma = 0.075$ & $\gamma = 0.125$ & $\gamma = 1.00$ \\
\midrule
0   & 7,754.90 & 7,857.96 & 7,874.50 \\
1   & 7,755.87 & 7,858.34 & 7,874.90 \\
10  & 7,755.91 & 7,858.39 & 7,874.91 \\
60  & 7,756.04 & 7,858.42 & 7,874.99 \\
600 & 7,776.05 & 7,858.61 & 7,875.26 \\
\bottomrule
\end{tabular}
\caption{Weighted Time-Integrated Cars by $nt_{opt}$ (seconds) and $\gamma$ for large $\gamma = 0.075, 0.125, 1.000$.}
\label{tab:wtcars_ntopt_comparison}
\end{table}

For $\gamma = 0.075, 0.125,$ and $1.0$, the weighted time integrated cars for varying $nt_{opt}$ values are given in Table \ref{tab:wtcars_ntopt_comparison} respectively. These choices in $\gamma$ values all result in a completely "flooded" network by the end of the simulation, with all road segments being LOS E. For $\gamma=0.075$, this occurs at 2,550 seconds (42.5 minutes), for $\gamma=0.125$ this occurs at 1,320 seconds (22 minutes), and for $\gamma=1.0$, it occurs at 1,110 seconds (18.5 minutes). Figure \ref{fig:gamma_0.075_snapshots} displays this behavior, providing snapshots of the network for $\gamma = 0.075, nt_{opt}=0$. Since the network begins congested and completely floods for all three choices of $\gamma$, the cars entered and cars exited are no longer useful metrics. The cars entered is 5,746, and cars exited is 4,816 for all three choices of $\gamma$. However, the loss function still shows improvement with increasing values of $nt_{opt}$, but the amount of improvement decreases as $\gamma$ increases. 

\section{Proof of Theorem~\ref{opt-sol}}\label{appendix:theorem_proof}

\begin{proof} $ $ \\
\begin{enumerate}[(i)]
\item If $r \leq 1$, then to maximize fluxes, it suffices to set $(\hat{\gamma}_1,\cdots,\hat{\gamma}_n)^T = (c_1,\cdots,c_n)^T$, equivalent to setting the incoming fluxes to be at the maximum capacities along the incoming roads. The resulting fluxes along the outgoing roads as prescribed by the drivers' preferences satisfies $(\hat{\gamma}_{n+1},\cdots,\hat{\gamma}_{n+m})^T \in \prod_{j=1}^m [0,c_{n+j}]$ since $r \leq 1$.
\item If $r > 1$, then the same trick as in (i) would not work as $(\hat{\gamma}_{n+1},\cdots,\hat{\gamma}_{n+m})^T \notin \prod_{j=1}^m [0,c_{n+j}]$. However, we are instead solving a different optimization problem involving the drivers' right of way as described in (B). Similar to (i), we set the incoming fluxes to be at their individual capacities by setting $\mu = 1$. For the value of $\lambda,$ notice that any value of $\lambda \in [0,1]$ would imply $(\hat{\gamma}_{n+1},\cdots,\hat{\gamma}_{n+m})^T \in \prod_{j=1}^m [0,c_{n+j}]$. However, we would still need to satisfy \eqref{PDE17}. One can easily check that the unique value of $\lambda$ can be obtained by solving $\sum_{i=1}^n \hat{\gamma}_i = \sum_{i=1}^n c_i = c_{\text{in}} = \lambda c_{\text{out}} = \sum_{j=n+1}^{n+m}\hat{\gamma}_j$. \vspace{5pt} \\
Physically, this corresponds to the scenario in which the outgoing roads are still able to accommodate for all the cars attempting to go through the junction onto the outgoing roads. Thus, we allow all the incoming cars to go through and scale the outgoing fluxes down from its capacity proportional to their individual capacities in order to respect the Kirchhoff Law of fluxes at the junction.
\item Analogous to (ii), we can see this by observing that maximizing the sum of incoming fluxes is the same as maximizing the sum of outgoing fluxes due to the Kirchhoff Law of fluxes at the junction. Since the total capacities along outgoing roads are smaller than that for incoming roads, we set the outgoing fluxes to be at its individual capacities by setting $\lambda = 1$. The value of $\mu$ can be obtained by utilizing the Kirchhoff Law of fluxes at the junction in a similar fashion as for $\lambda$ in (ii).\vspace{5pt} \\
Physically, this corresponds to the case of a congested junction due to the inability of the capacities of the outgoing roads to accommodate the fluxes of cars from the incoming roads. Thus, we allow as many cars to go through the outgoing roads as possible, and scale the fluxes from the incoming roads down proportional to its flux capacities.
\end{enumerate}
\end{proof}

\section{Detailed Computations for the Toy Network}\label{appendix:proofs}

In this appendix, we describe additional experiments on the toy network from Section~\ref{section:results} and provide the detailed computations. Since roads 2 and 3 share parameters and together form the residential route, we use the index $i = 2$ to represent the extended road that is the combination of roads 2 and 3 (with twice the length), and only distinguish them in the computation for the weighted time-integrated cars on the road.

We consider four experiments on this network, varying the initial data and source configuration.
\begin{itemize}
\item[\textbf{(A)}] \textit{Uncongested, single source.} Only road 1 is initialized at $\rho_{\text{init}} < \sigma_1$; all other roads start empty.
\item[\textbf{(B)}] \textit{Congested, single source.} Only road 1 is initialized at $\rho_{\text{init}} > \sigma_1$; all other roads start empty.
\item[\textbf{(C)}] \textit{Uncongested, all roads.} Roads 1, 2, 3, and 4 are initialized at $\rho_{\text{init}} < \sigma_i$, with road 5 empty.
\item[\textbf{(D)}] \textit{Congested, all roads.} Roads 1, 2, 3, and 4 are initialized at $\rho_{\text{init}} > \sigma_i$, with road 5 empty ($\rho_5(0,x) \equiv 0$).
\end{itemize}
Experiments (A)--(C) admit closed-form solutions that are independent of the number of exit lanes $n_5$; we derive these below along with the full closed-form expressions for Experiment~(D), including the detailed flux formulas that imply the phase transition in Proposition~\ref{closedform-CD}. We also provide a detailed rarefaction analysis for Experiment~(D).

\subsubsection*{Proposition~\ref{closedform-AB}: Closed-Form Expressions for Experiments (A) and (B)}

\begin{proposition}\label{closedform-AB}
(Closed Form Expressions for Experiments (A) and (B).) Suppose conditions~(i)--(iii) from Section~\ref{section:results} hold, and additionally:
\begin{itemize}
\item[(0)] \textit{(Initial data for experiments (A) and (B).)} $\rho_1(0,x)\equiv \rho_{\text{init}}$ for $x\in (0,L_1)$, and $\rho_i(0,x)\equiv 0$ for $x\in (0,L_i)$ for $i\in\{2,4,5\}$, where road $2$ denotes the extended road representing roads $2$ and $3$.
\item[(I)] $\rho_{\text{init}} < \sigma_1$,
\item[(II)] $f_1(\rho_\text{init}) < f_{c,i}$ for $i \in \{2,4\}$,
\item[(I')] $\rho_{\text{init}} \in (\sigma_1,1)$,
\item[(II')] $\max\left\{\frac{ \alpha f_{c,1}}{f_{c,2}}, \frac{(1-\alpha)f_{c,1}}{f_{c,4}}\right\} <1 $,
\end{itemize}
then:
\begin{itemize}
\item[(A)] Under assumptions (0), (I), and (II), the densities on roads $2,4,$ and $5$ are given by \eqref{example-eq4}--\eqref{example-eq7} and we have
\begin{multline}\label{example-CarsExited}
\text{Cars Exited}(T)
=
\rho_{\text{jam}}
\biggl[
(1-\alpha)\,v_1\rho_{\text{init}}\,
\Bigl(T-\tzero-\tauFive\Bigr)_+ \\
+\,
\alpha\,v_1\rho_{\text{init}}\,
\Bigl(T-\tone-\tauFive\Bigr)_+
\biggr]
\end{multline}
and
\begin{equation}\label{example-WeightedIntTotal}
\text{Weighted Time-Integrated Cars}(T) = W_1(T) + W_2(T) + W_3(T) + W_4(T) + W_5(T),
\end{equation}
where the expressions for $W_1(T)$ through $W_5(T)$ are given in \eqref{example-W1}--\eqref{example-W5} below. Here, $(s)_+ := \max\{0,s\}.$
\item[(B)] Under assumptions (0), (I'), and (II'), the densities on roads $2,4,$ and $5$ are given by \eqref{example-eq4}--\eqref{example-eq7} after replacing $\rho_{\text{init}}$ with $\sigma_1$ and hence $f_1(\rho_{\text{init}})$ with $f_{c,1}=v_1\sigma_1$, while the density along road 1 is given by \eqref{example-eq14}. Consequently, we have
\begin{equation}\label{example-CarsExited-B}
\text{Cars Exited}(T)
=
\rho_{\text{jam}}
\left[
(1-\alpha)\,f_{c,1}\,
\Big(T-\tzero-\tauFive\Big)_+
+
\alpha\,f_{c,1}\,
\Big(T-\tone-\tauFive\Big)_+
\right]
\end{equation}
and
\begin{equation}\label{example-WeightedIntTotal-B}
\text{Weighted Time-Integrated Cars}(T)
=
W_1^{(B)}(T)
+
\widetilde{W}_2(T)+\widetilde{W}_3(T)+\widetilde{W}_4(T)+\widetilde{W}_5(T),
\end{equation}
where $\widetilde{W}_i(T)$ for $i\in\{2,3,4,5\}$ are obtained from \eqref{example-W2}--\eqref{example-W5} by replacing $\rho_{\text{init}}$ with $\sigma_1$, and $W_1^{(B)}(T)$ is given in \eqref{example-W1-caseB}.
\end{itemize}
\end{proposition}

\textit{Proof of Proposition \ref{closedform-AB}}.

\textbf{Case (A).} At $t = 0$, we have a constant initial density $\rho_{\text{init}}$ along road $1$, and zero density along all the other roads. Hence, at the left junction as in Figure \ref{fig:toy}, since $\rho_{1,0} = \rho_{\text{init}} \in (0,\sigma_1)$ and $\rho_{2,0} = \rho_{4,0} = 0$, by \eqref{PDE7} and \eqref{PDE8}, we have $c_1 = \gamma_{1,0} = f_1(\rho_{\text{init}}) < \sigma_1, c_2 = f_2(\sigma_2) = f_{c,2},$ and $c_4 = f_4(\sigma_4) = f_{c,4}$. Furthermore, we have \begin{equation}\label{example_eq2}
A = \begin{pmatrix}
\alpha \\ 1- \alpha
\end{pmatrix}
\end{equation}
and by \eqref{opt-sol1} and \eqref{opt-sol2}, $r_1 = \frac{\alpha f_1(\rho_{\text{init}})}{f_{c,2}}$, $r_2 = \frac{(1-\alpha) f_1(\rho_{\text{init}})}{f_{c,4}}$, and hence
\begin{equation}\label{example_eq3}
r :=\max\left\{\frac{ \alpha f_1(\rho_{\text{init}})}{f_{c,2}}, \frac{(1-\alpha)f_1(\rho_{\text{init}})}{f_{c,4}}\right\}.
\end{equation}
By assumption (II), $r < 1$ and hence we are in case (i) of Theorem \ref{opt-sol}. This implies by \eqref{PDE18} that $\hat{\gamma}_1 = f_1(\rho_{\text{init}})$, $\hat{\gamma}_2 = \alpha f_1(\rho_{\text{init}})$, and $\hat{\gamma}_4 = (1-\alpha)f_1(\rho_\text{init})$, and hence by assumption (II) and the fact that $\alpha \in (0,1)$, $\hat{\rho}_1 = \rho_\text{init}, \hat{\rho}_2 = \frac{\alpha f_1(\rho_\text{init})}{v_2}$, and $\hat{\rho}_4 = \frac{(1-\alpha) f_1(\rho_\text{init})}{v_4}$.

Next, let us focus our attention to each road segment connected to the left junction.
\begin{itemize}
\item Along road $1$, we have $\rho_{1,0} = \hat{\rho}_1 = \rho_\text{init}$. This implies that the density is constant along the road even up and including the junction point, and thus there are no waves propagating forward and/or backwards.
\item Along road $2$, we have $\rho_{2,0} = 0$ and $\hat{\rho}_2 = \frac{\alpha f_1(\rho_\text{init})}{v_2} < \sigma_2$ by assumption (II). This implies that on both sides of the junction, the density is in the uncongested regime. This implies that $f'(\rho_{2,0}) = f'(\hat{\rho}_2) = v_2$, and hence we have a travelling wave with speed $v_2$ propagating down road 2. A similar argument for road $4$ with wave speed of $v_4$ holds too.
\end{itemize}

Since $v_2 < v_4$ and roads 2 and 4 have the same total length $L$, information along road $4$ will first propagate and hit the right boundary, at time $t = \frac{L}{v_4}$, we have $\rho_{4,0} = \frac{(1-\alpha) f_1(\rho_\text{init})}{v_4}$, $\rho_{2,0} = 0$ and $\rho_{5,0} = 0$. Since $\rho_{4,0} < \sigma_4$ by (II), we then have at the right junction, $c_2 = 0, c_4 = (1-\alpha) f_1(\rho_\text{init})$, and $c_5 = f_{c,5}$, with $A = \begin{pmatrix}
1 & 1
\end{pmatrix}.$ The only value of $r$ to compute is thus $r = \frac{c_2 + c_4}{c_5} = \frac{(1-\alpha)f_1(\rho_\text{init})}{ f_{c,5}}$. Since $f_{c,5} = f_{c,1}$ by (V), we have $r = \frac{(1-\alpha)f_1(\rho_\text{init})}{f_{c,1}}$. Since $f_1(\rho_\text{init}) < f_{c,1}$, then $r < 1$. By Theorem \ref{opt-sol}, we have $\hat{\gamma}_2 = 0$, $\hat{\gamma_4} = (1-\alpha) f_1(\rho_\text{init})$, and $\hat{\gamma_5} = \hat{\gamma}_2 + \hat{\gamma}_4 = (1-\alpha) f_1(\rho_\text{init})$. By (V), $\hat{\rho}_2 = 0, \hat{\rho}_4 = \frac{(1-\alpha) f_1(\rho_\text{init})}{v_4}$, and $\hat{\rho}_5 = \frac{(1-\alpha) f_1(\rho_\text{init})}{v_1}$.

As per usual, we will focus our attention back to each road segment connected to the right junction.
\begin{itemize}
\item Along road 5, we have $\rho_{5,0} = 0$ and $\hat{\rho}_5 = \frac{(1-\alpha) f_1(\rho_\text{init})}{v_1} < \sigma_1 = \sigma_5$. This implies that on both sides of the junction, the density is in the uncongested regime, and hence the wave speed on both sides would be $v_1.$
\item Along road $2$, we have $\rho_{2,0} = \hat{\rho}_2 = 0$ so the road is still empty for the region close to the junction.
\item Along road 4, we have $\hat{\rho}_4 = \rho_{4,0} = \frac{(1-\alpha)f_1(\rho_\text{init})}{v_4}$, and they are both in the uncongested regime. This implies that on both sides of the junction, the wave speed would be the same at $v_4$. Hence, no new waves are propagated back along road 4.
\end{itemize}

Moving on, at $t = \frac{L}{v_2} > \frac{L}{v_4}$, the wave from the left junction along road 2 starts to hit the right junction. Henceforth, we have $\rho_{2,0} = \frac{\alpha f_1(\rho_\text{init})}{v_2}$, $\rho_{4,0} = \frac{(1-\alpha)f_1(\rho_\text{init})}{v_4}$, and $\rho_{5,0} = \frac{(1-\alpha)f_1(\rho_\text{init})}{v_1}$. Since these densities are in the uncongested regime, we have $c_2 = \alpha f_1(\rho_\text{init})$ and $c_4 = (1-\alpha) f_1(\rho_\text{init})$, and $c_5 = f_{c,5}$ by \eqref{PDE7} and \eqref{PDE8}. With $A = \begin{pmatrix}
1 & 1
\end{pmatrix}$ and (V), the corresponding value of $r$ is given by $r = \frac{c_2 + c_4}{c_5} = \frac{f_1(\rho_\text{init})}{f_{5,c}} = \frac{f_1(\rho_\text{init})}{n_5 f_{1,c}} \leq \frac{f_1(\rho_\text{init})}{f_{1,c}} < 1.$ By Theorem \ref{opt-sol}, we have $\hat{\gamma}_2 = \alpha f_1(\rho_\text{init}), \hat{\gamma}_4 = (1-\alpha)f_1(\rho_\text{init})$ and hence $\hat{\gamma_5} = f_1(\rho_\text{init})$. This implies that $\hat{\rho}_2 = \rho_{2,0}, \hat{\rho}_4 = \rho_{4,0},$ while $\hat{\rho}_5 = \frac{f_1(\rho_\text{init})}{v_1} = \rho_{\text{init}}$.

Zooming in onto each road on the right junction, we then have the following:
\begin{itemize}
    \item Along roads 2 and 4, since $\rho_{i,0} = \hat{\rho}_i$ for $i \in \{2,4\}$ and they are both in the uncongested regime, then this phenomenon persists with no waves sent backwards along any of these roads.
    \item Along road 5, even though $\rho_{5,0} \neq \hat{\rho}_5$, the densities are both in the uncongested regime, which implies that no waves gets propagated forward.
\end{itemize}

Last but not least, notice that along road 5, for all $t > \frac{L}{v_4}$ (including when $t \in \ml \frac{L}{v_4}, \frac{L}{v_2}\mr$ and $t > \frac{L}{v_2}$), the wave speeds carrying the difference in densities are constant at $v_5 = v_1$. Hence, the first instance for which the wave carrying information about $\rho_5 \neq 0$ propagating from the right junction is at $t = \frac{L}{v_4} + \frac{L_5}{v_5}$. Since the right boundary of the exit road has a non-reflecting boundary condition, then there are no waves sent back along the road itself. Once both waves about the difference in density across it reaches the right boundary, this density distribution along road 5 would be a stationary state at the constant value of $\frac{f_1(\rho_\text{init})}{v_1} = \frac{v_1 \rho_\text{init}}{v_1} = \rho_\text{init}$.

Summarizing our analysis above, we have the following expressions for the density evolution on each of the roads: \\

For $t \in \ml 0 ,\frac{L}{v_4} \mr$, we have
\begin{equation}\label{example-eq4}
\begin{aligned}
\rho_1(t,x) &= \rho_\text{init} \quad &\text{ for } x \in (0,L_1)\\
\rho_2(t,x) &= \begin{cases}
\frac{\alpha f_1(\rho_\text{init})}{v_2} &\text{ for } x \in \ml 0, v_4 t\mr \\
0 &\text{ otherwise }
\end{cases} \\
\rho_4(t,x) &= \begin{cases}
\frac{(1-\alpha) f_1(\rho_\text{init})}{v_4} &\text{ for } x \in \ml 0, v_4 t\mr \\
0 &\text{ otherwise }
\end{cases} \\
\rho_5(t,x) &= 0  \quad &\text{ for } x \in (0,L_5).
\end{aligned}
\end{equation}
At $t \in \ml \frac{L}{v_4}, \frac{L}{v_2} \mr$, we have
\begin{equation}\label{example-eq5}
\begin{aligned}
\rho_1(t,x) &= \rho_\text{init} \quad &\text{ for } x \in (0,L_1)\\
\rho_2(t,x) &= \begin{cases}
\frac{\alpha f_1(\rho_\text{init})}{v_2} &\text{ for } x \in \ml 0, v_2 t\mr \\
0 &\text{ otherwise }
\end{cases} \\
\rho_4(t,x) &= \frac{(1-\alpha)f_1(\rho_\text{init})}{v_4} \quad &\text{ for } x \in (0,L) \\
\rho_5(t,x) &= \begin{cases}
\frac{(1-\alpha)f_1(\rho_\text{init})}{v_1} \quad &\text{ for } x \in \ml 0, \min
\left\{v_1\ml t - \frac{L}{v_4}\mr,L_5\right\} \mr \\
0 \quad &\text{ otherwise.}
\end{cases}
\end{aligned}
\end{equation}
Last but not least, we then have for $t > \frac{L}{v_2}$,
\begin{equation}\label{example-eq6}
\begin{aligned}
\rho_1(t,x) &= \rho_\text{init} \quad &&\text{ for } x \in (0,L_1)\\
\rho_2(t,x) &=
\frac{\alpha f_1(\rho_\text{init})}{v_2} &&\text{ for } x \in (0,L) \\
\rho_4(t,x) &=
\frac{(1-\alpha) f_1(\rho_\text{init})}{v_4} &&\text{ for } x \in \ml 0, L\mr  \\
\end{aligned}
\end{equation}
and
\begin{equation}\label{example-eq7}
\rho_5(t,x) = \begin{cases}
\rho_\text{init}, & x \in \bigl[0,\, \min\bigl\{v_1(t - \tfrac{L}{v_2}),\, L_5\bigr\}\bigr], \\[0.4em]
(1-\alpha)\rho_\text{init}, & x \in \bigl[\max\bigl\{v_1(t - \tfrac{L}{v_2}),\, L_5\bigr\},\,
\min\bigl\{v_1(t - \tfrac{L}{v_4}),\, L_5\bigr\}\bigr], \\[0.4em]
0, & \text{otherwise.}
\end{cases}
\end{equation}

To compute the number of cars exited, we then have

\begin{equation}\label{example-eq8}
\text{Cars Exited}(T) = \rho_\text{jam}\int_0^T f_5(L_5^-,t)\,\D t
\end{equation}
which evaluates to
\begin{equation}\label{example-eq8b}
= \begin{cases}
0, & t \in \bigl[0,\, \tfrac{L_5}{v_1} + \tfrac{L}{v_4}\bigr], \\[0.5em]
\rho_{\text{jam}}(1-\alpha)f_1(\rho_\text{init})\bigl(t - \tfrac{L_5}{v_1} - \tfrac{L}{v_4}\bigr), & t \in \bigl[\tfrac{L_5}{v_1} + \tfrac{L}{v_4},\, \tfrac{L_5}{v_1} + \tfrac{L}{v_2}\bigr], \\[0.5em]
\rho_{\text{jam}}f_1(\rho_\text{init})\bigl(t - \tfrac{L}{v_2} - \tfrac{L}{v_4}\bigr) \\
\quad +\, \rho_{\text{jam}}(1-\alpha)f_1(\rho_\text{init})\bigl(\tfrac{L}{v_2} - \tfrac{L}{v_4}\bigr), & t > \tfrac{L_5}{v_1} + \tfrac{L}{v_2}.
\end{cases}
\end{equation}

In other words, this simplifies to the expression in \eqref{example-CarsExited}.

To compute the weighted time integrated cars on the road, we first compute the relevant weights from the directed graph as shown in Figure \ref{fig:toy}. As road $5$ is the exit road, we have $d_5 = 0, d_3 = d_4 = 1, d_2 = d_1 = 2.$ Next, we proceed by computing this on a per-road basis.

For road 1, we have
\begin{equation}\label{example-W1}
W_1(T) = \frac{1}{4}\int_0^{L_1}\int_0^T \rho_\text{init} \D t \D x = \frac{\rho_\text{init}TL_1}{4}.
\end{equation}

For roads 2 and 3, we first break down the contribution of the density along the whole role into roads 2 and 3 individually. This yields
\begin{equation}\label{example-eq10}
\rho_2(x,t) = \begin{cases}
\frac{\alpha f_1(\rho_\text{init})}{v_2} &\text{ for } x \in \ml 0, \max\{v_2t,\frac{L}{2}\}\mr, \\
0 &\text{ otherwise}
\end{cases}
\end{equation}
and
\begin{equation}\label{example-eq11}
\rho_3(x,t) = \begin{cases}
\frac{\alpha f_1(\rho_\text{init})}{v_2} &\text{ for } x \in \ml \min\left\{0,v_2 \ml t - \frac{L}{2v_2}\mr\right \}, \max\left \{v_2 \ml t - \frac{L}{2v_2}\mr,\frac{L}{2}\right\}\mr, \\
0 &\text{ otherwise}.
\end{cases}
\end{equation}
Together with $f_1(\rho_\text{init}) = v_1 \rho_\text{init}$, these then imply
\begin{equation}\label{example-W2}
W_2(T) =
\frac{\alpha v_1\rho_{\text{init}}}{8}\,\min\!\left\{T,\frac{L}{2v_2}\right\}^2
+
\frac{\alpha v_1\rho_{\text{init}}}{8v_2}\,L\,\Big(T-\frac{L}{2v_2}\Big)_+
\end{equation}
and
\begin{equation}\label{example-W3}
\begin{aligned}
W_3(T)
&=
\frac{\alpha v_1\rho_{\text{init}}}{2}\left[
\frac{1}{2}\Big(T-\frac{L}{2v_2}\Big)_+^2
\frac{1}{2}\Big(T-\frac{L}{v_2}\Big)_+^2
\right] \\
&-
\frac{\alpha v_1\rho_{\text{init}}}{4v_2}L\Big(T-\frac{L}{2v_2}\Big)_+
+
\frac{\alpha v_1\rho_{\text{init}}}{4v_2}L\Big(T-\frac{L}{v_2}\Big)_+.
\end{aligned}
\end{equation}
One can perform a similar argument to show that
\begin{equation}\label{example-W4}
\begin{aligned}
W_4(T)
=
\frac{(1-\alpha)v_1\rho_{\text{init}}}{4}\,\min\!\left\{T,\frac{L}{v_4}\right\}^2
+
\frac{(1-\alpha)v_1\rho_{\text{init}}}{2v_4}\,L\,\Big(T-\frac{L}{v_4}\Big)_+.
\end{aligned}
\end{equation}

For road 5, using the min/max cases in the expression for the density on the road from \eqref{example-eq7}, one can work through the double integral and obtain
\begin{equation}\label{example-W5}
\begin{aligned}
W_5(T)
&=
(1-\alpha)\rho_{\text{init}}\left[
\frac{1}{2}\min\!\left\{\Big(T-\frac{L}{v_4}\Big)_+,\,\frac{L_5}{v_1}\right\}^2
+\frac{L_5}{v_1}\left(\Big(T-\frac{L}{v_4}\Big)_+ - \frac{L_5}{v_1}\right)_+
\right]\\
&\quad+
\alpha\rho_{\text{init}}\left[
\frac{1}{2}\min\!\left\{\Big(T-\frac{L}{v_2}\Big)_+,\,\frac{L_5}{v_1}\right\}^2
+\frac{L_5}{v_1}\left(\Big(T-\frac{L}{v_2}\Big)_+ - \frac{L_5}{v_1}\right)_+
\right].
\end{aligned}
\end{equation}

\textbf{Case (B).}
Similar to Case (A), we first investigate what happens at $t = 0$ on the left boundary. The only difference here is that $\rho_{1,0} = \rho_\text{init} \in (\sigma_1,1)$, and hence $c_1 = f_{1,c}$. The equivalent expression for $r$ as to \eqref{example_eq3} is given by
\begin{equation}\label{example-eq12}
r := \max\left\{\frac{ \alpha f_{c,1}}{f_{c,2}}, \frac{(1-\alpha)f_{c,1}}{f_{c,4}}\right\}.
\end{equation}
By (II'), $r < 1$ so we are in the same case in Theorem \eqref{opt-sol}. This then implies that $\hat{\gamma}_1 = f_{c,1}$ and thus $\hat{\gamma}_2 = \alpha f_{c,1}$, and $\hat{\gamma}_4 = (1-\alpha)f_{c,1}.$ Henceforth, we have $\hat{\rho}_1 = \sigma_1$, $\hat\rho_2 = \frac{\alpha f_{c,1}}{v_2}$, and $\hat\rho_4 = \frac{(1-\alpha) f_{c,1}}{v_4}$. By (II'), these new densities at the junction along roads 2 and 4 lie in the uncongested regime and thus the same traveling wave dynamics propagate along roads 2 and 4 at constant speeds of $v_2$ and $v_4$.

The evolution on roads $2,4,$ and $5$ then follows the same arguments as in Case (A), and can be obtained from \eqref{example-eq4}--\eqref{example-eq7} by replacing $\rho_\text{init}$ with $\sigma_1$ and hence $f_1(\rho_\text{init})$ with $f_{c,1}$.

To determine the dynamics on road $1$, note that the left state is $\rho_{\text{init}}$ while the junction state is $\hat{\rho}_1 = \sigma_1$. Since $\rho_{\text{init}} > \sigma_1$ and the flux is concave, this Riemann problem produces a backward-moving rarefaction fan connecting these two states. Since the transitional region lies in the congested regime, we then have $\ml f_1'\mr^{-1}(\xi) =  \sigma_1 - \frac{\xi}{2A_1}$. Hence, we have
\begin{equation}\label{example-eq13}
\rho_1(t,x) = \begin{cases}
\rho_\text{init} \quad &\text{ for } x \in (0,\max\{0,L_1 - v_1t\}), \\
\sigma_1 - \frac{x-L_1}{2A_1t} \quad &\text{ for } x \in (\max\{0,L_1 - v_1t \},L_1).
\end{cases}
\end{equation}
As the rarefaction wave hits the left boundary along road 1, our choice of boundary condition thus switches to absorb the incoming wave as described in Section \ref{numerical-scheme}. Consequently, the expression for cars exited up till time $T$ remains unchanged as in \eqref{example-eq8} apart from mapping $\rho_\text{init}$ to $\sigma_1$. Furthermore, we do the same mapping for $W_2$ to $W_5$. However, for $W_1,$ we compute the integral directly from \eqref{example-eq13}, given by
\begin{equation}\label{example-eq14}
\int_0^{L_1}\rho_1(t,x)\,\D x
=
\begin{cases}
\rho_{\text{init}}(L_1-v_1t)+\sigma_1v_1t+\frac{v_1^2}{4A_1}\,t, \quad &0<t\leq \frac{L_1}{v_1},\\
\sigma_1L_1+\frac{L_1^2}{4A_1t}, \quad &t>\frac{L_1}{v_1}.
\end{cases}
\end{equation}
Therefore,
\begin{equation}\label{example-W1-caseB}
W_1^{(B)}(T)
=
\frac{1}{4}
\begin{cases}
\rho_{\text{init}}L_1T
+
\dfrac{1}{2}\!\left(v_1(\sigma_1-\rho_{\text{init}})+\dfrac{v_1^2}{4A_1}\right)T^2,
& 0\le T\le \dfrac{L_1}{v_1},\\[1.5em]
\rho_{\text{init}}\dfrac{L_1^2}{v_1}
+
\dfrac{1}{2}\!\left(v_1(\sigma_1-\rho_{\text{init}})+\dfrac{v_1^2}{4A_1}\right)\!\left(\dfrac{L_1}{v_1}\right)^{\!2}
\\[0.5em]
\quad +\,
\sigma_1L_1\!\left(T-\dfrac{L_1}{v_1}\right)
+
\dfrac{L_1^2}{4A_1}\ln\!\left(\dfrac{Tv_1}{L_1}\right),
& T>\dfrac{L_1}{v_1}.
\end{cases}
\end{equation}
\qed

\begin{figure}[htbp]
\centering

\begin{minipage}[t]{0.48\textwidth}
    \centering
    \includegraphics[width=\linewidth]{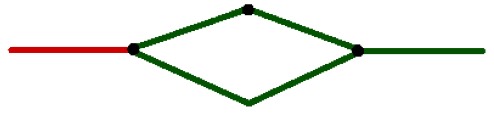}
\end{minipage}\hfill
\begin{minipage}[t]{0.48\textwidth}
    \centering
    \includegraphics[width=\linewidth]{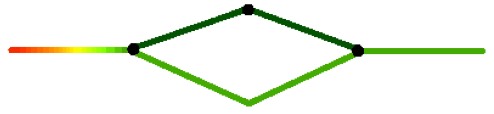}
\end{minipage}

\vspace{0.5em}

\begin{minipage}[t]{0.48\textwidth}
    \centering
    \includegraphics[width=\linewidth]{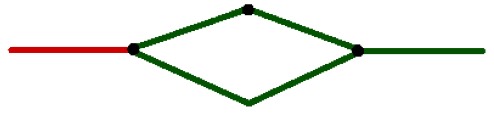}
\end{minipage}\hfill
\begin{minipage}[t]{0.48\textwidth}
    \centering
    \includegraphics[width=\linewidth]{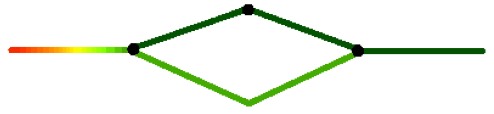}
\end{minipage}
\caption{Snapshots for the congested toy network with only road 1 initialized ($\rho_{\text{init}}=0.9$) at $t = 0$ seconds (left) and $t = 1000$ seconds (right), for $n_5 = 1$ (top) and $n_5 = 2$ (bottom).}
\label{fig:toy_empty}
\end{figure}

\subsubsection*{Proposition~\ref{closedform-CD-full}: Closed-Form Expressions for Experiments (C) and (D)}

\begin{proposition}\label{closedform-CD-full}
(Closed-Form Expressions for Experiments (C) and (D).) Suppose conditions~(i)--(iii) from Section~\ref{section:results} hold, with road $2$ denoting the extended road representing roads $2$ and $3$.
\begin{itemize}
\item[\textbf{(C)}] Under uncongested initial data $\rho_{\text{init}} < \sigma_i$ for $i\in\{1,2,4\}$ with road~5 empty, and additionally assuming
\begin{equation}\label{assump-C-exit-not-binding}
\max\left\{ f_2(\rho_{\text{init}})+f_4(\rho_{\text{init}}),\ f_2(\rho_{\text{init}})+(1-\alpha)f_1(\rho_{\text{init}}),\ f_1(\rho_{\text{init}})\right\} < f_{c,5},
\end{equation}
the right-junction outflow rate is piecewise constant:
\begin{equation}\label{gamma5-C}
\gamma_5(t) =
\begin{cases}
f_2(\rho_{\text{init}})+f_4(\rho_{\text{init}}), & 0\le t < \tzero,\\
f_2(\rho_{\text{init}})+(1-\alpha)f_1(\rho_{\text{init}}), & \tzero \le t < \tone,\\
f_1(\rho_{\text{init}}), & t\ge \tone,
\end{cases}
\end{equation}
and the cars exited satisfy
\begin{equation}\label{example-CarsExited-C}
\text{Cars Exited}(T)
=
\rho_{\text{jam}}\int_0^{(T-\tauFive)_+}\gamma_5(t)\,\D t.
\end{equation}
Equivalently, with $(s)_+ := \max\{0,s\}$,
\begin{equation}\label{example-CarsExited-C-kinks}
\begin{aligned}
\text{Cars Exited}(T)
=\rho_{\text{jam}}\Big[&
\left(f_2(\rho_{\text{init}})+f_4(\rho_{\text{init}})\right)\Big(T-\tauFive\Big)_+ \\
&+\left((1-\alpha)f_1(\rho_{\text{init}})-f_4(\rho_{\text{init}})\right)\Big(T-\tauFive-\tzero\Big)_+ \\
&+\left(\alpha f_1(\rho_{\text{init}})-f_2(\rho_{\text{init}})\right)\Big(T-\tauFive-\tone\Big)_+\Big].
\end{aligned}
\end{equation}
Moreover, $\hat{\rho}_5(t)<\sigma_5$ for all $t\ge 0$, so road $5$ does not saturate.

\item[\textbf{(D)}] Under congested initial data $\rho_{\text{init}}>\sigma_i$ for $i\in\{1,2,4\}$ with road~5 initially empty ($\rho_5(0,x)\equiv 0$), the exit throughput is
\begin{equation}\label{gamma5-D-def}
\gamma_5=\min\{f_{c,2}+f_{c,4},\,f_{c,5}\}=\min\{f_{c,2}+f_{c,4},\,n_5 f_{c,1}\}.
\end{equation}
The incoming fluxes from roads $2$ and $4$ at the right junction are
\begin{equation}\label{gamma24-D}
(\gamma_2,\gamma_4)=
\begin{cases}
\left( \frac{n_5}{n_5^*}f_{c,2},\ \frac{n_5}{n_5^*}f_{c,4}\right), & n_5<n_5^\star, \\
(f_{c,2},f_{c,4}), & n_5\ge n_5^\star.
\end{cases}
\end{equation}
Under the non-depletion assumption ($T \le \min\{\rho_{\text{init}}L/\gamma_2,\,\rho_{\text{init}}L/\gamma_4\}$), $\gamma_5$ is constant on $[0,T]$ and integrating gives
\begin{equation}\label{example-CarsExited-D}
\text{Cars Exited}(T)=\rho_{\text{jam}}\,\Big(T-\tauFive\Big)_+ \times \begin{cases}
n_5 f_{c,1} \quad &\text{ for } n_5 < n_5^*,  \\
f_{c,2} + f_{c,4} \quad &\text{ for } n_5 > n_5^*.\\
\end{cases}
\end{equation}
In particular, the phase transition at $n_5^\star$ in Proposition~\ref{closedform-CD} follows from \eqref{gamma5-D-def}.
\end{itemize}
\end{proposition}

\begin{proof}
\textit{Experiment (C).} At $t=0$, by the uncongested initial data, \eqref{PDE7} and \eqref{PDE8} give $c_1=f_1(\rho_{\text{init}})$, $c_2=f_{c,2}$, and $c_4=f_{c,4}$. Since the outgoing roads can accommodate the split, we are in case (i) of Theorem~\ref{opt-sol} at the left junction, giving outgoing fluxes $\hat{\gamma}_2=\alpha f_1(\rho_{\text{init}})$ and $\hat{\gamma}_4=(1-\alpha)f_1(\rho_{\text{init}})$. This induces traveling waves on roads $2$ and $4$ propagating right with speeds $v_2$ and $v_4$.

At the right junction, take $A=(1\ \ 1)$. For $0\le t<\tzero$, the boundary states on roads $2$ and $4$ are still $\rho_{\text{init}}$, hence $c_2=f_2(\rho_{\text{init}})$ and $c_4=f_4(\rho_{\text{init}})$. Under \eqref{assump-C-exit-not-binding}, $r\le 1$ so $\gamma_5(t)=f_2(\rho_{\text{init}})+f_4(\rho_{\text{init}})$. For $\tzero\le t<\tone$, the boundary state on road $4$ switches, giving $c_4=(1-\alpha)f_1(\rho_{\text{init}})$; again $r\le 1$ so $\gamma_5(t)=f_2(\rho_{\text{init}})+(1-\alpha)f_1(\rho_{\text{init}})$. For $t\ge \tone$, both boundary states have switched, giving $\gamma_5(t)=f_1(\rho_{\text{init}})$. This establishes \eqref{gamma5-C}. Since road $5$ is initially empty and remains in free flow, information reaches the exit after travel time $\tauFive$, yielding \eqref{example-CarsExited-C} and \eqref{example-CarsExited-C-kinks}.

\textit{Experiment (D).} Under conditions~(i)--(iii) with congested initial data ($\rho_{\text{init}}>\sigma_i$ for $i\in\{1,2,4\}$) and road~5 initially empty. At time $t=0$ and for as long as roads $2$ and $4$ remain congested at the right junction, the incoming capacities are $c_2=f_{c,2}$ and $c_4=f_{c,4}$ by \eqref{PDE7}. Since road $5$ is initially empty, $c_5=f_{c,5}=n_5 f_{c,1}$ by \eqref{PDE8}. With $A=(1\ \ 1)$, we have
$$
r=\frac{c_2+c_4}{c_5}=\frac{f_{c,2}+f_{c,4}}{n_5 f_{c,1}},
$$
so $r\le 1$ is equivalent to $n_5\ge n_5^\star$ as in \eqref{n5-star}. If $n_5\ge n_5^\star$, Theorem~\ref{opt-sol}(i) gives $\gamma_5=c_2+c_4=f_{c,2}+f_{c,4}<f_{c,5}$ (strict if $n_5>n_5^\star$), so the exit does not saturate. If $n_5<n_5^\star$, then $c_{\text{in}}>c_{\text{out}}$ and Theorem~\ref{opt-sol}(iii) yields $\gamma_5=c_5=f_{c,5}$, i.e., the exit saturates. This establishes \eqref{gamma5-D-def}--\eqref{example-CarsExited-D}.
\end{proof}

\subsubsection*{Detailed Rarefaction Analysis for Experiment (D)}

The backward rarefaction fans along roads $2$ and $4$ in Figure \ref{fig:toy_flooded} become clearly visible only after the phase transition at $n_5^\star$. For $n_5>n_5^\star$, we are in case (i) of Theorem \ref{opt-sol} so $\gamma_i=f_{c,i}$ for $i\in\{2,4\}$, hence $\hat{\rho}_i=\sigma_i$, producing a large drop from the congested interior state $\rho_{\text{init}}=0.9$ to
\begin{equation*}
\sigma_2 = \frac{f_{c,2}}{v_2} = \frac{(400/\rho_\text{jam})}{15}=\frac{2}{15}\approx 0.133,
\qquad
\sigma_4=\frac{f_{c,4}}{v_4}=\frac{(500/\rho_\text{jam})}{20}=\frac{2.5}{20}=0.125,
\end{equation*}
where $\rho_\text{jam}=200$ is the jam density used in the toy simulations (see \eqref{jam-density}). In contrast, for $n_5<n_5^\star$ we are in case (iii) of Theorem \ref{opt-sol}, so $(\gamma_2,\gamma_4)$ are reduced as in \eqref{gamma24-D}; taking the admissible congested-branch choice $\hat{\rho}_i=\ml f_i^{-1}\mr^{+}(\gamma_i)\in[\sigma_i,1]$ yields junction states that remain in the congested regime. For example, when $n_5=1$ this gives $\hat{\rho}_2\approx 0.71$ and $\hat{\rho}_4\approx 0.71$, much closer to $\rho_{\text{init}}=0.9$ than to $\sigma_i$. Therefore the corresponding backward rarefaction is comparatively weak, whereas after the phase transition it spans a much wider density range and is visually pronounced.

To verify that the backward waves are rarefactions rather than shocks, note that a backward shock would require $\hat{\rho}_i>\rho_{\text{init}}$, equivalently $\gamma_i<f_i(\rho_{\text{init}})$. Using $A_i=-f_{c,i}/(1-\sigma_i)^2$ (see \eqref{jam-density}), we compute
\begin{equation}
\begin{aligned}
f_2(\rho_{\text{init}})&=f_{c,2}\left(1-\left(\frac{\frac{9}{10}-\frac{2}{15}}{1-\frac{2}{15}}\right)^2\right)
=\frac{147}{676}\,f_{c,2}, \text{ and } \\
f_4(\rho_{\text{init}})&=f_{c,4}\left(1-\left(\frac{\frac{9}{10}-\frac{1}{8}}{1-\frac{1}{8}}\right)^2\right)
=\frac{264}{1225}\,f_{c,4}.
\end{aligned}
\end{equation}
In the pre-transition case $n_5=1$, \eqref{gamma24-D} gives $\gamma_i=\frac{5}{9}f_{c,i}$. Since $\frac{5}{9}>\frac{147}{676}$ and $\frac{5}{9}>\frac{264}{1225}$, we have $\gamma_i>f_i(\rho_{\text{init}})$ and hence $\hat{\rho}_i<\rho_{\text{init}}$; therefore the backward waves are rarefactions. After the phase transition, $\gamma_i=f_{c,i}>f_i(\rho_{\text{init}})$ is immediate, and the rarefactions become much stronger because $\hat{\rho}_i=\sigma_i\ll\rho_{\text{init}}$.

\begin{landscape}
\section{Detailed Road Data}\label{appendix:roaddata}
The following tables provide the detailed road data used to compute the flux functions for each road in the network, along with the resulting normalized fluxes.
\begin{table}[H]
\centering
\scalebox{0.75}{ 
\begin{tabular} {|l|l|l|l|l|l|l|p{8cm}|} 
\hline 
& & (mi) & & (mi/hr) & & (veh/hr/lane) & (per 1 lane) \\ 
\# & Name & Length & Lanes & Speed Limit & Road Class & $f_{max}$ & Normalized Flux \\ 
\hline
Hwy30[0] & Hwy-30:& 0.01& 2 & 35 & Parkway & 875 & $f(\rho) = \begin{cases}
35 \rho, & 0< \rho \leq 0.125\\
-5.714\rho^2+1.429\rho+4.286, & 0.125 < \rho \leq 1
\end{cases}$ \\ 
&(source) $\rightarrow$ Prison St &(default) && && &\\
Hwy30[1] & Hwy-30: & 0.28 & 2 & 35 & Parkway & 875 & $f(\rho) = \begin{cases}
35 \rho, & 0< \rho \leq 0.125\\
-5.714\rho^2+1.429\rho+4.286, & 0.125 < \rho \leq 1
\end{cases}$ \\ 
&Prison St $\rightarrow$ Dickenson St&&& && &\\
Hwy30[2] & Hwy-30:& 0.16 & 2 & 35 & Parkway & 875 & $f(\rho) = \begin{cases}
35 \rho, & 0< \rho \leq 0.125\\
-5.714\rho^2+1.429\rho+4.286, & 0.125 < \rho \leq 1
\end{cases}$ \\
& Dickenson St $\rightarrow$ Lahainaluna Rd &&& && &\\
Hwy30[3] & Hwy-30: & 0.12 & 2 & 40 & Parkway & 1000 & $f(\rho) = \begin{cases}
40 \rho, & 0< \rho \leq 0.125\\
-6.531\rho^2+1.633\rho+4.898, & 0.125 < \rho \leq 1
\end{cases}$ \\ 
&Lahainaluna Rd $\rightarrow$ Papalaua St &&& && &\\
Hwy30[4] & Hwy-30: & 0.32 & 2 & 40 & Parkway & 1000 & $f(\rho) = \begin{cases}
40 \rho, & 0< \rho \leq 0.125\\
-6.531\rho^2+1.633\rho+4.898, & 0.125 < \rho \leq 1
\end{cases}$ \\ 
&Papalaua St $\rightarrow$ Kenui St&&& && &\\
Hwy30[5] & Hwy-30:& 0.17 & 2 & 40 & Parkway & 1000 & $f(\rho) = \begin{cases}
40 \rho, & 0< \rho \leq 0.125\\
-6.531\rho^2+1.633\rho+4.898, & 0.125 < \rho \leq 1
\end{cases}$ \\ 
&Kenui St $\rightarrow$ Keawe St &&& && &\\
Hwy30[6] & Hwy-30:& 0.66 & 2 & 40 & Parkway & 1000 & $f(\rho) = \begin{cases}
40 \rho, & 0< \rho \leq 0.125\\
-6.531\rho^2+1.633\rho+4.898, & 0.125 < \rho \leq 1
\end{cases}$ \\ 
&Keawe St $\rightarrow$ Front St &&& && &\\
Hwy30[7] & Hwy-30:  & 0.01 & 2 & 40 & Parkway & 1000 & $f(\rho) = \begin{cases}
40 \rho, & 0< \rho \leq 0.125\\
-6.531\rho^2+1.633\rho+4.898, & 0.125 < \rho \leq 1
\end{cases}$ \\ 
&Front St $\rightarrow$ (exit)&(default) && && &\\
\hline
\end{tabular}
}
\caption{Honopiilani Highway (Hwy30) road data used for Lahaina.}
\label{table:hwy30data}
\end{table}

\clearpage  
\begin{table}
\scalebox{0.75}{
\begin{tabular} {|l|l|l|l|l|l|l|l|}
\hline & &(mi) &  &(mi/hr) &  &(veh/hr/lane)   & (per 1 lane) \\
\# & Name & Length & Lanes & Speed Limit &  Road Class& $f_{max}$& Normalized Flux\\
\hline
Front[0] & Front:& 0.01&1& 20 & Major&500  & $f(\rho) = \begin{cases}
20 \rho, \quad&0< \rho \leq 0.125\\
-3.265\rho^2+0.816\rho+2.449, \quad& 0.125 < \rho \leq 1
\end{cases}$\\
& (source) $\rightarrow$ Prison St &(default)&& &Collector & &\\
Front[1] & Front: &0.06&1& 20 & Major&500  & $f(\rho) = \begin{cases}
20 \rho, \quad&0< \rho \leq 0.125\\
-3.265\rho^2+0.816\rho+2.449, \quad& 0.125 < \rho \leq 1
\end{cases}$\\
& Prison St $\rightarrow$ Canal St&&& &Collector & &\\
Front[2] & Front: &0.14&1& 20 & Major&500  & $f(\rho) = \begin{cases}
20 \rho, \quad&0< \rho \leq 0.125\\
-3.265\rho^2+0.816\rho+2.449, \quad& 0.125 < \rho \leq 1
\end{cases}$\\
& Canal St $\rightarrow$ Dickenson St&&& &Collector & &\\
Front[3] & Front: & 0.16&1& 20 & Major&500  & $f(\rho) = \begin{cases}
20 \rho, \quad&0< \rho \leq 0.125\\
-3.265\rho^2+0.816\rho+2.449, \quad& 0.125 < \rho \leq 1
\end{cases}$\\
& Dickenson St $\rightarrow$ Lahainaluna Rd&&& &Collector & &\\
Front[4] & Front:& 0.05 &1& 20 & Major&500  & $f(\rho) = \begin{cases}
20 \rho, \quad&0< \rho \leq 0.125\\
-3.265\rho^2+0.816\rho+2.449, \quad& 0.125 < \rho \leq 1
\end{cases}$\\
& Lahainaluna Rd $\rightarrow$ Wahie Ln &&& &Collector & &\\
Front[5] & Front:& 0.10 &1& 20 & Major&500  &$f(\rho) = \begin{cases}
20 \rho, \quad&0< \rho \leq 0.125\\
-3.265\rho^2+0.816\rho+2.449, \quad& 0.125 < \rho \leq 1
\end{cases}$\\
& Wahie Ln $\rightarrow$ Papalaua St &&& &Collector & &\\
Front[6] & Front:& 0.17 &1& 20 & Major&500  & $f(\rho) = \begin{cases}
20 \rho, \quad&0< \rho \leq 0.125\\
-3.265\rho^2+0.816\rho+2.449, \quad& 0.125 < \rho \leq 1
\end{cases}$\\
& Papalaua St $\rightarrow$ Baker St &&& &Collector & &\\
Front[7] & Front:& 0.17 &1& 20 & Major&500  & $f(\rho) = \begin{cases}
20 \rho, \quad&0< \rho \leq 0.125\\
-3.265\rho^2+0.816\rho+2.449, \quad& 0.125 < \rho \leq 1
\end{cases}$\\
& Baker St $\rightarrow$ Kenui St &&& &Collector & &\\
Front[8] & Front:& 0.10&1& 20 & Major&500  & $f(\rho) = \begin{cases}
20 \rho, \quad&0< \rho \leq 0.125\\
-3.265\rho^2+0.816\rho+2.449, \quad& 0.125 < \rho \leq 1
\end{cases}$\\
& Kenui St $\rightarrow$ Puunoa Pl &&& &Collector & &\\
Front[9] & Front:& 0.78 &1& 20 & Major&500  & $f(\rho) = \begin{cases}
20 \rho, \quad&0< \rho \leq 0.125\\
-3.265\rho^2+0.816\rho+2.449, \quad& 0.125 < \rho \leq 1
\end{cases}$\\
& Puunoa Pl $\rightarrow$ Hwy-30 &&& &Collector & &\\
\hline
\end{tabular}}
\caption{Front Street (Front) road data used for Lahaina.}
\label{table:Frontdata}
\end{table}

\begin{table}
\scalebox{0.75}{
\begin{tabular} {|l|l|l|l|l|l|l|l|}
\hline & &(mi) &  &(mi/hr) &  &(veh/hr/lane)   & (per 1 lane) \\
\# & Name & Length & Lanes & Speed Limit &  Road Class & $f_{max}$& Normalized Flux\\
\hline
Wainee[0] & Waine'e:& 0.01&1& 20 & Local&300  & $f(\rho) = \begin{cases}
20 \rho, \quad&0< \rho \leq 0.075\\
-1.753\rho^2+0.263\rho +1.490, \quad&  0.075 < \rho \leq 1
\end{cases}$\\
& (source) $\rightarrow$ Prison St &(default)&& & Street & &\\
Wainee[1] & Waine'e: &0.14&1& 20 & Local&300  &  $f(\rho) = \begin{cases}
20 \rho, \quad&0< \rho \leq 0.075\\
-1.753\rho^2+0.263\rho +1.490, \quad&  0.075 < \rho \leq 1
\end{cases}$\\
& Prison St $\rightarrow$ Hale St&&& & Street & &\\
Wainee[2] & Waine'e: &0.10&1& 20 & Local &300  &  $f(\rho) = \begin{cases}
20 \rho, \quad&0< \rho \leq 0.075\\
-1.753\rho^2+0.263\rho +1.490, \quad&  0.075 < \rho \leq 1
\end{cases}$\\
& Hale St $\rightarrow$ Dickenson St&&& & Street & &\\
Wainee[3] & Waine'e:& 0.11&1& 20 & Local&300  &  $f(\rho) = \begin{cases}
20 \rho, \quad&0< \rho \leq 0.075\\
-1.753\rho^2+0.263\rho +1.490, \quad&  0.075 < \rho \leq 1
\end{cases}$\\
& Dickenson St $\rightarrow$ Panaewa St &&& & Street & &\\
Wainee[4] & Waine'e:& 0.05&1& 20 & Local&300  &  $f(\rho) = \begin{cases}
20 \rho, \quad&0< \rho \leq 0.075\\
-1.753\rho^2+0.263\rho +1.490, \quad&  0.075 < \rho \leq 1
\end{cases}$\\
& Panaewa St $\rightarrow$ Lahainaluna Rd &&& & Street & &\\
Wainee[5] & Waine'e:& 0.14&1& 20 &Minor&400  &  $f(\rho) = \begin{cases}
20 \rho, \quad&0< \rho \leq 0.1\\
-2.469\rho^2+0.494\rho +1.975, \quad&  0.1< \rho \leq 1
\end{cases}$\\
& Lahainaluna Rd $\rightarrow$ Papalaua St&&& &Collector & &\\
Wainee[6] & Waine'e:& 0.16 &1& 20 &Minor&400  & $f(\rho) = \begin{cases}
20 \rho, \quad&0< \rho \leq 0.1\\
-2.469\rho^2+0.494\rho +1.975, \quad&  0.1< \rho \leq 1
\end{cases}$\\
& Papalaua St $\rightarrow$ Baker St &&& &Collector & &\\
Wainee[7] & Waine'e:& 0.16 &1& 20 &Minor&400  & $f(\rho) = \begin{cases}
20 \rho, \quad&0< \rho \leq 0.1\\
-2.469\rho^2+0.494\rho +1.975, \quad&  0.1< \rho \leq 1
\end{cases}$\\
& Baker St $\rightarrow$ Kenui St&&& &Collector & &\\
\hline
\end{tabular}}
\caption{Waine'e Street (Wainee) road data used for Lahaina.}
\label{table:Waineedata}
\end{table}

\begin{table}

\scalebox{0.75}{
\begin{tabular} {|l|l|l|l|l|l|l|l|}
\hline & &(mi) &  &(mi/hr) &  &(veh/hr/lane)   & (per 1 lane) \\
\# & Name & Length & Lanes & Speed Limit &  Road Class& $f_{max}$& Normalized Flux\\
\hline
Prison[0]&Prison:&0.16&1&20&Local&300& $f(\rho) = \begin{cases}
20 \rho, \quad&0< \rho \leq 0.075\\
-1.753\rho^2+0.263\rho +1.490, \quad&  0.075 < \rho \leq 1
\end{cases}$\\
&Front St$\rightarrow$ Wainee St&&&&Street&&\\
Prison[1]&Prison:&0.08&1&20&Local&300& $f(\rho) = \begin{cases}
20 \rho, \quad&0< \rho \leq 0.075\\
-1.753\rho^2+0.263\rho +1.490, \quad&  0.075 < \rho \leq 1
\end{cases}$\\
&Wainee St$\rightarrow$ Hwy-30&&&&Street&&\\
Dicken[0]&Dickenson:&0.05&1&20&Minor&400&$f(\rho) = \begin{cases}
20 \rho, \quad&0< \rho \leq 0.1\\
-2.469\rho^2+0.494\rho +1.975, \quad&  0.1< \rho \leq 1
\end{cases}$\\
&Front St $\rightarrow$Luakini St&&&&Collector&&\\
Dicken[1]&Dickenson:&0.09&1&20&Minor&400&$f(\rho) = \begin{cases}
20 \rho, \quad&0< \rho \leq 0.1\\
-2.469\rho^2+0.494\rho +1.975, \quad&  0.1< \rho \leq 1
\end{cases}$\\
&Luakini St$\rightarrow$ Wainee St&&&&Collector&&\\
Dicken[2]&Dickenson:&0.11&1&20&Minor&400&$f(\rho) = \begin{cases}
20 \rho, \quad&0< \rho \leq 0.1\\
-2.469\rho^2+0.494\rho +1.975, \quad&  0.1< \rho \leq 1
\end{cases}$\\
&Wainee St$\rightarrow$ Hwy-30&&&&Collector&&\\
Papal[0]&Papalaua:&0.15&1&20&Major&500& $f(\rho) = \begin{cases}
20 \rho, \quad&0< \rho \leq 0.125\\
-3.265\rho^2+0.816\rho+2.449, \quad& 0.125 < \rho \leq 1
\end{cases}$\\
&Front St$\rightarrow$ Wainee St&&&&Collector&&\\
Papal[1]&Papalaua:&0.07&1&20&Major&500& $f(\rho) = \begin{cases}
20 \rho, \quad&0< \rho \leq 0.125\\
-3.265\rho^2+0.816\rho+2.449, \quad& 0.125 < \rho \leq 1
\end{cases}$\\
&Wainee St$\rightarrow$ Hwy-30&&&&Collector&&\\
Kenui[0]&Kenui:&0.10&1&20&Minor&400&$f(\rho) = \begin{cases}
20 \rho, \quad&0< \rho \leq 0.1\\
-2.469\rho^2+0.494\rho +1.975, \quad&  0.1< \rho \leq 1
\end{cases}$\\
&Front St $\rightarrow$ Kahoma Vlg&&&&Collector&&\\
Kenui[1]&Kenui:&0.08&1&20&Minor&400&$f(\rho) = \begin{cases}
20 \rho, \quad&0< \rho \leq 0.1\\
-2.469\rho^2+0.494\rho +1.975, \quad&  0.1< \rho \leq 1
\end{cases}$\\
&Kahoma Vlg$\rightarrow$ Wainee St&&&&Collector&&\\
Kenui[2]&Kenui:&0.02&1&20&Minor&400&$f(\rho) = \begin{cases}
20 \rho, \quad&0< \rho \leq 0.1\\
-2.469\rho^2+0.494\rho +1.975, \quad&  0.1< \rho \leq 1
\end{cases}$\\
&Wainee St$\rightarrow$ Hwy-30&&&&Collector&&\\
\hline
\end{tabular}}
\caption{Residential Roads West of Hwy-30 road data used for Lahaina.}
\label{table:resroaddata}
\end{table}
\begin{table}
\scalebox{0.75}{
\begin{tabular} {|l|l|l|l|l|l|l|l|}
\hline & &(mi) &  &(mi/hr) &  &(veh/hr/lane)   & (per 1 lane) \\
\# & Name & Length & Lanes & Speed Limit &  Road Class& $f_{max}$& Normalized Flux\\
\hline
Keawe[0] & Keawe:& 0.10&2& 25 & Major&550  & $f(\rho) = \begin{cases}
25 \rho, \quad&0< \rho \leq 0.11\\
-3.472\rho^2+0.764\rho+2.708, \quad&  0.11< \rho \leq 1
\end{cases}$\\
& Hwy-30 $\rightarrow$ Gateway Shopping Ctr&&& &Collector & &\\
Keawe[1] & Keawe:& 0.09&2& 25 & Major &550  & $f(\rho) = \begin{cases}
25 \rho, \quad&0< \rho \leq 0.11\\
-3.472\rho^2+0.764\rho+2.708, \quad&  0.11< \rho \leq 1
\end{cases}$\\
& Gateway Shopping Ctr $\rightarrow$ Oil Rd &&& &Collector & &\\
LB[0] & Lahaina Bypass:& 0.01&1& 30 & Arterial/&650  & $f(\rho) = \begin{cases}
30 \rho, \quad&0< \rho \leq 0.108\\
-4.085\rho^2+0.885\rho + 3.202, \quad& 0.108< \rho \leq 1
\end{cases}$\\
& (source)$\rightarrow$ Lahainaluna Rd& (default) && &Parkway& &\\
LB[1] & Lahaina Bypass:& 1.06 &1& 30 & Arterial/&650  & $f(\rho) = \begin{cases}
30 \rho, \quad&0< \rho \leq 0.108\\
-4.085\rho^2+0.885\rho + 3.202, \quad& 0.108< \rho \leq 1
\end{cases}$\\
& Lahainaluna Rd $\rightarrow$ Oil Rd &&& &Parkway& &\\
\hline
\end{tabular}}
\caption{Keawe Street Extension (Keawe) and Lahaina Bypass (LB) road data used for Lahaina.}
\label{table:LBdata}
\end{table}

\begin{table}
\scalebox{0.65}{
\begin{tabular} {|l|l|l|l|l|l|l|l|}
\hline & &(mi) &  &(mi/hr) &  &(veh/hr/lane)   & (per 1 lane) \\
\# & Name & Length & Lanes & Speed Limit &  Road Class & $f_{max}$& Normalized Flux\\
\hline
LL[0] & Lahainaluna: & 0.14 & 1 & 20 & Major & 500 & $f(\rho) = \begin{cases}
20 \rho, \quad&0< \rho \leq 0.125\\
-3.265\rho^2+0.816\rho+2.449, \quad& 0.125 < \rho \leq 1
\end{cases}$\\
& Front St $\rightarrow$ Wainee St& && &Collector & &\\
LL[1] & Lahainaluna:& 0.09 & 1 & 20 & Major & 500 & $f(\rho) = \begin{cases}
20 \rho, \quad&0< \rho \leq 0.125\\
-3.265\rho^2+0.816\rho+2.449, \quad& 0.125 < \rho \leq 1
\end{cases}$\\
& Wainee St $\rightarrow$  Hwy-30 & && &Collector & &\\
LL[2] & Lahainaluna:& 0.14 & 1 & 20 & Major & 500 &$f(\rho) = \begin{cases}
20 \rho, \quad&0< \rho \leq 0.125\\
-3.265\rho^2+0.816\rho+2.449, \quad& 0.125 < \rho \leq 1
\end{cases}$\\
& Hwy-30 $\rightarrow$ Kuhua St & && &Collector & &\\
LL[3] & Lahainaluna: & 0.05 & 1 & 20 & Major& 500 & $f(\rho) = \begin{cases}
20 \rho, \quad&0< \rho \leq 0.125\\
-3.265\rho^2+0.816\rho+2.449, \quad& 0.125 < \rho \leq 1
\end{cases}$\\
& Kuhua St $\rightarrow$ Pauoa St & && &Collector & &\\
LL[4] & Lahainaluna: & 0.09 & 1 & 20 & Major& 500 & $f(\rho) = \begin{cases}
20 \rho, \quad&0< \rho \leq 0.125\\
-3.265\rho^2+0.816\rho+2.449, \quad& 0.125 < \rho \leq 1
\end{cases}$\\
& Pauoa St $\rightarrow$ Kale St& && &Collector & &\\
LL[5] & Lahainaluna:& 0.08 & 1 & 20 & Major & 500 & $f(\rho) = \begin{cases}
20 \rho, \quad&0< \rho \leq 0.125\\
-3.265\rho^2+0.816\rho+2.449, \quad& 0.125 < \rho \leq 1
\end{cases}$\\
& Kale St $\rightarrow$ Paunau St & && &Collector & &\\
LL[6] & Lahainaluna: & 0.06 & 1 & 20 & Major & 500 & $f(\rho) = \begin{cases}
20 \rho, \quad&0< \rho \leq 0.125\\
-3.265\rho^2+0.816\rho+2.449, \quad& 0.125 < \rho \leq 1
\end{cases}$\\
& Paunau St $\rightarrow$ Kelawea St& && &Collector & &\\
LL[7] & Lahainaluna:& 0.12 & 1 & 30 & Major& 600 & $f(\rho) = \begin{cases}
30 \rho, \quad&0< \rho \leq 0.1\\
-3.701\rho^2+0.741\rho+2.963, \quad&  0.1 < \rho \leq 1
\end{cases}$\\
&  Kelawea St $\rightarrow$ Kalena St & && &Collector & &\\
LL[8] & Lahainaluna:& 0.13 & 1 & 30 & Major& 600 & $f(\rho) = \begin{cases}
30 \rho, \quad&0< \rho \leq 0.1\\
-3.701\rho^2+0.741\rho+2.963, \quad&  0.1 < \rho \leq 1
\end{cases}$\\
& Kalena St $\rightarrow$ Dirt Road & && &Collector & &\\
LL[9] & Lahainaluna: & 0.03 & 1 & 30 & Major& 600 & $f(\rho) = \begin{cases}
30 \rho, \quad&0< \rho \leq 0.1\\
-3.701\rho^2+0.741\rho+2.963, \quad&  0.1 < \rho \leq 1
\end{cases}$\\
& Dirt Road $\rightarrow$ Lahaina Bypass& && &Collector & &\\
LL[10] & Lahainaluna:& 0.01 & 1 & 30 & Major & 600 & $f(\rho) = \begin{cases}
30 \rho, \quad&0< \rho \leq 0.1\\
-3.701\rho^2+0.741\rho+2.963, \quad&  0.1 < \rho \leq 1
\end{cases}$\\
& Lahaina Bypass $\rightarrow$ (source) &(default)&& &Collector & &\\
\hline
\end{tabular}}

\caption{LahainaLuna Road (LL) road data used for Lahaina.}
\label{table:LLdata}
\end{table}

\begin{table}
\scalebox{0.75}{
\begin{tabular} {|l|l|l|l|l|l|l|l|}
\hline & &(mi) &  &(mi/hr) &  &(veh/hr/lane)   & (per 1 lane) \\
\# & Name & Length & Lanes & Speed Limit &  Road Class & $f_{max}$& Normalized Flux\\
\hline
1&Kuhua:&0.28&1&20&Local&300&$f(\rho) = \begin{cases}
20 \rho, \quad&0< \rho \leq 0.075\\
-1.753\rho^2+0.263\rho +1.490, \quad&  0.075 < \rho \leq 1
\end{cases}$\\
&Lahainaluna Rd $\rightarrow$(source) &&& &Street& &\\
2&Komo Mai:&0.18&1&20&Local&300&$f(\rho) = \begin{cases}
20 \rho, \quad&0< \rho \leq 0.075\\
-1.753\rho^2+0.263\rho +1.490, \quad&  0.075 < \rho \leq 1
\end{cases}$\\
&(source)$\rightarrow$Keawe St Ext&&& &Street& &\\
3&Pauoa:&0.18&1&20&Local&300&$f(\rho) = \begin{cases}
20 \rho, \quad&0< \rho \leq 0.075\\
-1.753\rho^2+0.263\rho +1.490, \quad&  0.075 < \rho \leq 1
\end{cases}$\\
&Lahainaluna Rd $\rightarrow$(source) &&& &Street& &\\
4&Kale:&0.18&1&20&Local&300&$f(\rho) = \begin{cases}
20 \rho, \quad&0< \rho \leq 0.075\\
-1.753\rho^2+0.263\rho +1.490, \quad&  0.075 < \rho \leq 1
\end{cases}$\\
&Lahainaluna Rd $\rightarrow$(source) &&& &Street& &\\
5&Paunau:&0.18&1&20&Local&300&$f(\rho) = \begin{cases}
20 \rho, \quad&0< \rho \leq 0.075\\
-1.753\rho^2+0.263\rho +1.490, \quad&  0.075 < \rho \leq 1
\end{cases}$\\
&Lahainaluna Rd $\rightarrow$(source) &&& &Street& &\\
6&Kelawea:&0.15&1&20&Local&300&$f(\rho) = \begin{cases}
20 \rho, \quad&0< \rho \leq 0.075\\
-1.753\rho^2+0.263\rho +1.490, \quad&  0.075 < \rho \leq 1
\end{cases}$\\
&Lahainaluna Rd $\rightarrow$(source) &&& &Street& &\\
7&Kalena:&0.18&1&20&Local&300&$f(\rho) = \begin{cases}
20 \rho, \quad&0< \rho \leq 0.075\\
-1.753\rho^2+0.263\rho +1.490, \quad&  0.075 < \rho \leq 1
\end{cases}$\\
&Lahainaluna Rd $\rightarrow$(source) &&& &Street& &\\
8&Nondescript Dirt Road:&0.25&1&20&Local&300&$f(\rho) = \begin{cases}
20 \rho, \quad&0< \rho \leq 0.075\\
-1.753\rho^2+0.263\rho +1.490, \quad&  0.075 < \rho \leq 1
\end{cases}$\\
&Lahainaluna Rd $\rightarrow$(source) &&& &Street& &\\
\hline
\end{tabular}}
\caption{Source Roads with non-default lengths East of Hwy-30 road data used for Lahaina.}
\label{table:sourcedata}
\end{table}

\FloatBarrier
\section{Initial Density Data}\label{appendix:initdata}
Table \ref{table:allroads_initdata} includes road data used to estimate the initial density along each road segment.
\begin{table}[H]
\centering
\scalebox{0.7}{ 
\begin{tabular} {|l|l|l|l|l|l|l|} 
\hline  
Name & Segment & AADT (veh) & Google Data Available? (Y/N) & LOS & $v_0$ (mi/hr) & Normalized Initial Density (per 1 lane)\\ 
\hline
Highway 30 && 19,796 & Y & B & 24.5 &  $\rho_0 =0.178$\\ 
&(source) $\rightarrow$ Lahainaluna Rd &&&&&\\
& & 19,796 & Y & B & 28 &  $\rho_0 =0.178$\\  
&Lahainaluna Rd $\rightarrow$  Kenui St&&&&&\\
&&19,796&Y&C&20&$\rho_0 = 0.245$ \\ 
& Kenui St $\rightarrow$ Keawe St &&& &&\\
&& 37,300&Y & B & 28 & $\rho_0 = 0.178$\\
& Keawe St $\rightarrow$ (exit) &&&&&\\
\hline
Front Street && 6,060 & N & D & 8 & $\rho_0 = 0.300$\\
&(source) $\rightarrow$ Hwy-30 &&&&&\\
\hline
Waine'e Street && 0 & N & A & 20 & $\rho_0 = 0$\\
&(source) $\rightarrow$ Lahainaluna Rd &&&&&\\
&& 3,939 & N & A & 20 & $\rho_0 = 0.056$\\
&Lahainaluna Rd $\rightarrow$ Kenui St &&&&&\\
\hline 
Prison Street && 0 & N & A & 20 & $\rho_0 = 0$\\
&Front St $\rightarrow$ Hwy-30 &&&&&\\
\hline 
Dicckenson Street && 3,333 & N & A & 20 & $\rho_0 = 0.047$\\
&Front St $\rightarrow$ Hwy-30 &&&&&\\
\hline 
Papalaua Street && 3,434 & N & A & 20 & $\rho_0 = 0.049$\\
&Front St $\rightarrow$ Hwy-30 &&&&&\\
\hline 
Kenui Street && 2,652 & N & A & 20 & $\rho_0 = 0.038$\\
&Front St $\rightarrow$ Hwy-30 &&&&&\\
\hline
Keawe Street &&20,196&Y&C&12.5&$\rho_0 = 0.217$\\
&Hwy-30 $\rightarrow$ Gateway Shopping Ctr &&& &&\\
&& 20,196 & Y & B & 17.5 & $\rho_0 = 0.157$\\
&Gateway Shopping Ctr $\rightarrow$ Oil Rd&&& &&\\
&& 20,196 & Y & C & 12.5 & $\rho_0 = 0.217$\\
&Oil Rd $\rightarrow$ Lahaina Bypass&&& &&\\
\hline
Lahaina Bypass &&16,218 & Y & B & 21 & $\rho_0 = 0.154$\\
&(source) $\rightarrow$ Keawe St Ext&&&&&\\
\hline
LahainaLuna Road && 8,585 & N & C & 10 & $\rho_0 = 0.245$\\
&Front St $\rightarrow$ Wainee St&&&&&\\
&& 8,585 & Y & B & 14 & $\rho_0 = 0.178$\\
&Wainee St $\rightarrow$ Kelawea St&&&&&\\
&& 8,585 & Y & B & 21 & $\rho_0 = 0.143$\\
&Kelawea St $\rightarrow$ (source)&&&&&\\
\hline
\end{tabular}
}
\caption{Hwy-30/Honopiilani Highway (Hwy30) initial road data used for AM Base Network.}
\label{table:allroads_initdata}
\end{table}
\FloatBarrier
\end{landscape}
\end{document}